\documentclass{amsart}

\usepackage{amssymb, amsmath, esint, xcolor}
\usepackage[backref=page]{hyperref}

\usepackage{mathtools}
\mathtoolsset{showonlyrefs}

\usepackage{enumitem}

\renewcommand{\[}{\begin{equation}\begin{aligned}}
\renewcommand{\]}{\end{aligned} \end{equation}}

\newcommand{\dist}{\overline{\mathrm{dist}}} 
\newcommand{\dd}{\mathbf{d}_{\mathcal{C}_{n,k}}}

\usepackage{verbatim}

\newtheorem{thm}{Theorem}
\newtheorem{prop}[thm]{Proposition}

\newtheorem{cor}[thm]{Corollary}
\newtheorem{conj}[thm]{Conjecture}

\theoremstyle{remark}

\theoremstyle{definition}
\newtheorem{definition}[thm]{Definition}

\author{G\'abor Sz\'ekelyhidi}
\address{Department of Mathematics, Northwestern University, Evanston,
  IL, USA}
\email{gaborsz@northwestern.edu}

\title{Mean convex flows with surgery}
\date{}

\begin{document}

\begin{abstract}
  We construct a mean curvature flow with surgery starting with any
  compact mean convex hypersurface in $\mathbb{R}^{n+1}$, extending previous results of
  Huisken-Sinestrari, Brendle-Huisken, and Haslhofer-Kleiner for
  2-convex flows. In contrast with previous constructions of mean
  curvature flow with surgery, the topological surgeries
  are performed by the flow itself through nondegenerate cylindrical
  singularities. At the same time the flow needs to be slightly adjusted at
  finitely many smooth times. 
\end{abstract}

\maketitle
\section{Introduction}
The mean curvature flow deforms a hypersurface $L\subset
\mathbb{R}^{n+1}$ in the direction of its mean curvature, and can be
viewed as the negative $L^2$-gradient flow of the area functional. Since the
foundational studies by Brakke~\cite{Brakke} and
Huisken~\cite{Hui84}, the flow has become a central object in
geometric analysis and the study of hypersurfaces. 
A fundamental feature of the flow is that typically
singularities will form in finite
time. In the setting of convex hypersurfaces Huisken~\cite{Hui84} showed that the
flow eventually becomes graphical over a shrinking sphere that
collapses to a point. More general initial conditions can lead to more
complicated singularities, changing the topology of the
evolving hypersurface. For instance, suitable dumbbell shaped surfaces in
$\mathbb{R}^3$ are decomposed by the flow into two
pieces, which then continue to collapse into round points.
Understanding the behavior of the flow through such
singularities is essential for geometric and topological
applications.

A successful approach under certain assumptions
is to perform surgery at the singular times:
one removes high-curvature regions, caps off the
resulting boundary components, and restarts the flow on the resulting
smooth hypersurface. If this process terminates after finitely many
surgeries, then it yields a complete picture of the topology of the
initial hypersurface. This program was carried out in full for the
class of 2-convex hypersurfaces with $n > 2$ -- i.e. where the principal curvatures
$\lambda_1 \leq \ldots \leq \lambda_n$ satisfy $\lambda_1+\lambda_2 >
0$ -- by Huisken-Sinestrari in a landmark paper
\cite{HS08}, and for $n=2$ by Brendle-Huisken~\cite{BH15},
Haslhofer-Kleiner~\cite{HK16}. The key geometric consequence of the 2-convexity
assumption is that all singularities are either spherical, or are
modeled on the cylinder $S^{n-1}\times\mathbb{R}$, by
White~\cite{White02}. Powerful canonical
neighborhood theorems in \cite{HS08, HK16, BH15} give enough
information in neighborhoods of such singularities to successfully
execute the surgery procedure.

While the 2-convexity hypothesis is natural, it is somewhat
restrictive. In this paper we will instead consider the larger class
of mean convex hypersurfaces, i.e. those that satisfy 
$\lambda_1 + \ldots + \lambda_n > 0$. In this case, by White~\cite{White02}, the singularities
can be modelled on more general cylinders of the form $S^{n-k}\times
\mathbb{R}^k$. Compared to the 2-convex case, the higher dimensional Euclidean
factors with $k\geq 2$ present serious difficulties in extending the surgery program to this
setting. Instead of attempting to 
change the topology of the evolving hypersurface near singularities
by a cutting and pasting procedure, our approach is to show that after suitable
perturbations the flow admits only nondegenerate singularities in the
sense of Angenent-Vel\'azquez~\cite{AV97}, more recently studied
extensively by Sun-Xue~\cite{SX22} and Sun-Wang-Xue~\cite{SWX25}. This
approach is suggested by the following conjecture. See
Sun-Wang-Xue~\cite[Conjecture 1.4]{SWX25}, as well as
Ilmanen~\cite[Problem 9]{IlmanenProb}, 
Colding-Minicozzi-Pedersen~\cite[Conjecture 7.1]{CMP15} and
Sun-Xue~\cite[Conjecture 1.1]{SX22} for closely related conjectures. 

\begin{conj}
  For a generic compact embedded initial hypersurface $L_0$, the mean curvature flow
  $L_t$ admits only spherical and nondegenerate cylindrical
  singularities. 
\end{conj}

The importance of nondegenerate singularities is that by Sun-Wang-Xue~\cite{SWX25}
they correspond to well understood topological surgeries (see
Theorem~\ref{thm:SWXsurgery}), so 
a flow through nondegenerate cylindrical singularities can be viewed
as a canonical flow with surgery.
While we are not able to prove this conjecture in such generality, our
main result is the following.

\begin{thm}\label{thm:main}
  For any compact, embedded, mean convex initial hypersurface
  $L_0$, there exists an \emph{almost mean curvature flow} $L_t$ for all
  $t\geq 0$ that
  has only spherical and nondegenerate cylindrical singularities. 
\end{thm}

An almost mean curvature flow is a mean curvature flow at all but
finitely many smooth times, while at these times it is slightly
adjusted by smooth isotopies. See Definition~\ref{defn:AMCF}
for the precise definition. This result yields a long sought-after surgery-type procedure 
applicable to the class of mean convex hypersurfaces, and thus opens
the way to topological applications to this wider class.

At the heart of the proof of Theorem~\ref{thm:main} is
Proposition~\ref{prop:perturbdegen10} which shows, roughly speaking, that if a mean
curvature flow $L_t$ has a degenerate cylindrical singularity at a
point $(x_0, t_0)$, then we can find arbitrarily small perturbations
of $L_t$, which have no degenerate singularities of the same type in a
neighborhood of $(x_0, t_0)$ of a definite size. This result is
obtained by a generalization of the technique introduced by the author
in \cite{Sz26} in order to deal with cylindrical singularities of the
form $S^1\times \mathbb{R}$. The advantage of having a
one-dimensional $\mathbb{R}$-factor in that work is that these
singularities satisfy a dichotomy: either they are nondegenerate, or
the corresponding rescaled mean curvature flow 
decays exponentially fast to the cylinder. This is no longer the case for
cylindrical singularities modeled on $S^{n-k}\times \mathbb{R}^k$ for
$k > 1$, because of the possibility that the behavior is nondegenerate
in certain directions in the $\mathbb{R}^k$ factor, but not
others. This is the substantial new difficulty that is addressed in this
paper.

The paper is organized as follows. In Section~\ref{sec:prelim} we
collect some basic definitions and preliminary results. In the next
two sections, \ref{sec:L2normal} and \ref{sec:geomnormal}, the main
goal is to identify 
the asymptotics of rescaled mean curvature flows converging to
cylinders.
Some of these results are stated in the work of Sun-Xue~\cite{SX22}, but we
need to obtain more precise quantitative control, and so we reprove most
of the ingredients of their normal form theorem. In Section~\ref{sec:geomnormal} we
also prove Proposition~\ref{prop:uqdecay}, which shows that the
derivatives $\partial_{y_j}u$ of the graphical function in the
degenerate directions decay exponentially. In Section~\ref{sec:perturb} we will prove the
main result, Proposition~\ref{prop:perturb20} that allows us to locally perturb away degenerate
cylindrical singularities. Finally in Section~\ref{sec:proof} we use
this local perturbation result to prove Theorem~\ref{thm:main}. 

\subsection*{Acknowledgements}
I would like to thank Nick Edelen for helpful discussions. This work was
supported in part by NSF grant DMS-2506325.

\section{Preliminaries}\label{sec:prelim}
Throughout the paper we fix the dimension of our hypersurfaces to be
$n$, and we assume a bound 
for the entropy of the initial hypersurface, introduced by
Colding-Minicozzi~\cite{CM12}. By Huisken's monotonicity
formula~\cite{Hui90}, the entropy is non-increasing along the
flow, which  implies that all
rescalings of all hypersurfaces that appear in our arguments will have
uniform entropy bounds, or equivalently, uniform bounds for the area
ratios at all scales. We will usually suppress dependence of our constants on this
entropy bound and the dimension. In estimates we will sometimes use the notation
$\Psi(\epsilon_1, \ldots, \epsilon_k\,|\, \delta_1, \ldots, \delta_l)$
to denote a function which converges to zero if the $\delta_i$ are
fixed and $\epsilon_i\to 0$. This function $\Psi$ may change depending
on the context. 

Suppose that $L_t \subset \mathbb{R}^{n+1}$ is a mean curvature flow
of mean convex hypersurfaces. This can be defined through
singularities as a Brakke flow~\cite{Brakke}, or equivalently, the level set
flow~\cite{ES91,CGG91}. Recall that by the mean convexity assumption,
these flows are
non-fattening and unique~\cite{ES91}.

The singularities of $L_t$ can be studied using 
tangent flows. Suppose that $(x_0, t_0)$ is a point in spacetime. A
tangent flow of $L_t$ at $(x_0, t_0)$ is defined as a subsequential
limit of a sequence of rescalings
\[ \lambda_k (L_{t_0 + \lambda_k^{-2}t} - x_0). \]
By Huisken~\cite{Hui90} (see also Ilmanen~\cite{Ilmanen94}), any tangent flow is given by a
self shrinking solution $\sqrt{-t}\Sigma$ of the mean curvature
flow. 
A convenient way to study tangent flows at $(x_0, t_0)$ is through the rescaled mean
curvature flow (RMCF) centered at $(x_0, t_0)$, which is defined to be the flow
\[ M_\tau = e^{\tau/2} (L_{t_0 - e^{-\tau}} - x_0). \]
This flow has normal variation given by 
\[ \frac{\partial x}{\partial \tau} = H + \frac{x^\perp}{2}, \]
and the self-shrinkers $\Sigma$ corresponding to the possible tangent
flows of $L_t$ at $(x_0, t_0)$ are given by the 
subsequential limits of $M_\tau$ as $\tau\to \infty$. 

Let us write
\[ \mathcal{C}_{n,k} = \mathbb{R}^k \times S^{n-k}(\sqrt{2(n-k)})
  \subset \mathbb{R}^{n+1} \]
for $k =  1,\ldots, n-1$. Throughout the paper we will always use
coordinates $y_1,\ldots y_k$ on the $\mathbb{R}^k$ factor, coordinates
$z_1,\ldots, z_{n+1-k}$ on the $\mathbb{R}^{n+1-k}$ factor, and $x$
for a point in $\mathbb{R}^{n+1}$.

In general the tangent flows may have
multiplicity, and singularities. However, by White~\cite{White02}, all tangent flows of the
mean convex flow $L_t$ have multiplicity one, and up to rotation are
given  by homothetically shrinking
cylinders $\sqrt{-t}\mathcal{C}_{n,k}$, or the sphere $\sqrt{-t}\mathcal{S}_n =
S^n(\sqrt{-2nt})$. We will say that $L_t$ has a
$\mathcal{C}_{n,k}$-singularity at a spacetime point $(x_0, t_0)$ if
the rescaled flow $M_\tau$ centered at $(x_0, t_0)$ converges, along a
subsequence, to $\mathcal{C}_{n,k}$ as $\tau\to\infty$. By the uniqueness result of
Colding-Minicozzi~\cite{CM15_1} in this case the whole flow $M_\tau$
converges to the same cylinder as $\tau\to\infty$. In this case we
will also say that $M_\tau$ has a $\mathcal{C}_{n,k}$-singularity at
infinity. More generally we will say that $L_t$ has an
$(n,k)$-cylindrical singularity at $(x_0, t_0)$ if the corresponding
rescaled flow converges to  a rotation $Q\mathcal{C}_{n,k}$.

If the rescaled flow $M_\tau$ converges to $\mathcal{C}_{n,k}$, then
as $\tau\to \infty$ we can write $M_\tau$ as the graph of a function
$u(x,\tau)$ over $\mathcal{C}_{n,k}$ on larger and larger regions. We
will say that $M_\tau$ is $\epsilon$-graphical over
$\mathcal{C}_{n,k}$ on the region $U$, if $M_\tau\cap U$ can be
written as the graph of a function $u$ with $\Vert u\Vert_{C^4} \leq
\epsilon$.

On several occasions we will consider a rescaled flow $M'_\tau$ that is
the normal graph of a function $u(x,\tau)$ over another rescaled flow
$M_\tau$. In this case $u$ satisfies an equation of the form
\[ \partial_\tau u = \mathcal{L}_{M_\tau} u + Q_{M_\tau}(u, \nabla u,
  \nabla^2 u), \]
where $\mathcal{L}$ is the linearized operator given by
\[ \mathcal{L}_{M_\tau}(u)  = \Delta u - \frac{1}{2} x\cdot \nabla u +
  \frac{1}{2} u + |A|^2 u, \]
and the quantities are computed with respect to $M_\tau$. For the
cylinder $\mathcal{C}_{n,k}$ we have
\[ \mathcal{L}_{\mathcal{C}} = \Delta - \frac{1}{2} x\cdot \nabla + 1, \]
and we will sometimes omit the subscript if from the context it is
clear which linearized operator we are using. We will denote by
$\sigma(\mathcal{C}_{n,k})$ the spectrum of
$\mathcal{L}_{\mathcal{C}_{n,k}}$ corresponding to the rates of
  homogeneous solutions of the linearized equation. As explained in
  \cite{SX22} we have
  $\sigma(\mathcal{C}_{n,k}) \cap [0,\infty) = \{0,1/2, 1\}$. The
  eigenfunctions corresponding to 1 are the constants; those
  corresponding to $1/2$ are the coordinate functions $y_i, z_\alpha$;
  the ones corresponding to $0$ are the quadratic functions
  $y_iz_\alpha$ corresponding to rotations, as well as $y_iy_j -
  2\delta_{ij}$. The next largest (negative) eigenvalue is
  $-\min\{\frac{1}{n-k}, \frac{1}{2}\}$ (see \cite[Table 1]{SX22}).

Sun-Xue~\cite{SX22} (see also \cite{AV97,Gang22,Gang24}) further analyzed cylindrical tangent
flows, and proved a normal form theorem for rescaled flows converging to
$\mathcal{C}_{n,k}$. They showed the following.
\begin{thm}[\cite{SX22}]\label{thm:SXnormal} Suppose that the rescaled flow $M_\tau$
  converges to the cylinder $\mathcal{C}_{n,k}$. Then there is a number
  $\ell \in \{0,\ldots, k\}$ such that up to a rotation in the
  $\mathbb{R}^k$-factor, for sufficiently large $\tau$ we have the
  following: for some $\alpha > 0$, we can write $M_\tau$ as the graph of a function
  $u(y,z,\tau)$ over the ball $B_{\tau^\alpha}$, such that
  \[\label{eq:SXnormal}  u(y,z,\tau) = \frac{\sqrt{2(n-k)}}{4\tau} \sum_{i=1}^\ell (y_i^2-2)
    + o(\tau^{-1}), \]
  as $\tau\to\infty$.
\end{thm}

\begin{definition}\label{defn:nondeg}
Following Sun-Xue~\cite{SX22}, we say that the flow $L_t$ has a
\emph{non-degenerate} $\mathcal{C}_{n,k}$ singularity at $(x_0, t_0)$,
if up to rotation, in the normal form above for the rescaled flow
centered at $(x_0,t_0)$, we have $\ell=k$.
If $\ell < k$, we will say that the flow has an
\emph{$\ell$-degenerate} $\mathcal{C}_{n,k}$
singularity. We will also
say that $M_\tau$ has a non-degenerate or $\ell$-degenerate
$\mathcal{C}_{n,k}$ singularity at infinity. 

More generally, we will say that the flow has a
non-degenerate or degenerate $(n,k)$-cylindrical singularity at a
point, if up to a rotation it has a non-degenerate or degenerate
$\mathcal{C}_{n,k}$-singularity. When it is
clear from the context, we will drop the $(n,k)$ prefix or subscript. 
\end{definition}

The importance of non-degenerate cylindrical singularities is that, as
shown by Sun-Wang-Xue~\cite{SWX25}, we can understand the topological
change performed by the flow through them.
More precisely, they proved the following (see
\cite[Theorem 1.1]{SWX25}). 

\begin{thm}[\cite{SWX25}] \label{thm:SWXsurgery}
  Suppose that the flow $L_t$ has a non-degenerate $(n,k)$-cylindrical
  singularity at $(x_0, t_0)$. Then in a spacetime neighborhood
  \[ Q_r(x_0, t_0) := \{ (x, t) \, :\, |x-x_0| +  |t-t_0|^{1/2} < r\} \]
  the singularity is isolated. In addition if $t_-
  < t_0 < t_+$ are sufficiently close, then inside $B_r(x_0)$,
  $M_{t_+}$ is obtained topologically by performing
  an $(n-k)$-surgery on $M_{t_-}$.  In other words, $M_{t_+}\cap
  B_r(x_0)$ is obtained by removing a region diffeomorphic to 
$S^{n-k}\times B^k$ from $M_{t_-}\cap B_r(x_0)$, and gluing in
$B^{n-k+1}\times S^{k-1}$ along the  boundary, where $B^k$ denotes the
$k$-dimensional ball. 
\end{thm}

An important ingredient in this paper is a more quantitative version
of the normal form result. The main goal is to show the following type
of result. Suppose that the flow
$L_t$ has an $\ell$-degenerate $\mathcal{C}_{n,k}$-singularity at a
point $(x_0, t_0)$, and the normal form asymptotics holds with good
enough error terms in the sense that the $o(\tau^{-1})$ term is
bounded by $\tau^{-19/10}$, say, on a time interval $[T_0,
T_0+L_0(T_0)]$. Suppose that a nearby flow $L_t'$ has an
$(n,k)$-cylindrical singularity at a nearby point $(x_0', t_0')$ (with
``nearby'' depending on $T_0$). We want to show that then this singularity is either
$\ell$-degenerate, with good error term estimates in its asymptotics
for $\tau > T_1$, where $T_1$ depends on $T_0$, or the singularity is
$\ell'$-degenerate for $\ell' > \ell$, or non-degenerate. 
The overall argument roughly follows the
strategy of \cite{SX22}, or earlier works such as
Velazquez~\cite{Vel93} and Filippas-Liu~\cite{FL93} on semilinear
parabolic PDE, however the details are somewhat different. This will
take up most of the next two sections. 

An important new result will be Proposition~\ref{prop:uqdecay} on the exponential decay
of derivatives in the degenerate $y$-directions. This is related to a
result of Choi-Haslhofer~\cite[Theorem 1.10]{CH24}, which studies
ancient flows in $\mathbb{R}^4$ asymptotic to a cylinder $S^1\times
\mathbb{R}^2$, and in particular deals with a situation analogous to
$1$-degenerate cylindrical singularities. 

The starting point is
\cite[Proposition 2.3]{SX22}, which is implicitly contained in the
work of Colding-Minicozzi~\cite{CM15_1}.
\begin{prop}\label{prop:CMgraphicality}
  There is an $\alpha > 0$, with the following property. Suppose that
  $M_\tau$ is a RMCF converging to $\mathcal{C}_{n,k}$, and $M_0$ is
  $\delta$-graphical 
over $\mathcal{C}_{n,k}$ on the ball $B_{\delta^{-1}}$, for
sufficiently small $\delta$. Then for all
$\tau >1$ we have the following. Defining $r(\tau) = \sqrt{\alpha
  \ln \tau}$, we have that $M_\tau$ is the graph of a
function $u(x,\tau)$ over $\mathcal{C}_{n,k}$ on $B_{r(\tau)}$,
with $\Vert u\Vert_{C^4}\leq \Psi(\delta)$, and $\Vert u(\cdot,
\tau)\Vert_{L^2(B_{r(\tau)})} \leq \tau^{-\alpha}$.
\end{prop}

Recall that we are suppressing the dependence of constants on the
assumed entropy bound and dimension. In general $\alpha, \delta$ would
depend on these too.

The next step is to extend the good graphicality
to balls of size $B_{\tau^\beta}$ for some $\beta > 0$. For this we will
use an approach based on Ecker's log-Sobolev
inequality~\cite{Ecker}. We also use
Brendle's sharp version of the log-Sobolev
inequality~\cite{Brendle20}, which simplifies some constants that
appear, although this is not strictly necessary. 

First recall the following consequence of the log-Sobolev
inequality from \cite{Ecker}, where we have an additional forcing
term.  We give the proof in the appendix. It is worth noting that while
the flow $M_\tau$ may have singularities, and might only satisfy the
rescaled mean curvature equation in a weak sense, in the following
results we will always consider functions that are supported on the
smooth part of the flow -- in practice they will be supported in a
region where the flow is close to a cylinder. 
\begin{prop}\label{prop:logSobolev}
  Suppose that $M_\tau$ is a rescaled mean curvature flow, and for
  some $\epsilon \geq 0$,  the nonnegative function
  $f(x,\tau)$ satisfies the differential inequality
  \[ \label{eq:driftheatineq}
    \partial_\tau f \leq \Delta f - \frac{1}{2} x\cdot \nabla f +
    a\cdot \nabla f + b f + \mathcal{E},  \]
  where $a(x, \tau)$ and $b(x,\tau)$ satisfy $|a| \leq \epsilon<1$ and
  $|b| \leq b_0$, and $\mathcal{E}(x,\tau) \geq 0$.
  Then we have 
   \[ \label{eq:subsol1} \Vert f(\cdot, \tau)\Vert_{L^{p(\tau)}} \leq
     e^{(\epsilon+b_0)\tau} \Vert 
     f(\cdot, 0)\Vert_{L^{2}} + \int_0^\tau e^{(\epsilon+b_0)(\tau -
       s)} \Vert \mathcal{E}(\cdot, s)\Vert_{L^{p(s)}}\, ds, \]
   where we define $p(\tau) = 1 + \epsilon + (1-\epsilon)e^\tau$ (so
   that $p(0)=2$). The $L^p$ norms here are taken using the weight
  $(4\pi)^{-n/2} e^{-|x|^2/4}$. 
\end{prop}

The following result leading to pointwise points is similar to \cite[Lemma
3.5]{LSSz22}, obtained by
applying the monotonicity formula centered at different
points. Compared to \cite{LSSz22} we have an additional error term. 

\begin{prop}\label{prop:L2pointwise}
Suppose that along the rescaled mean curvature flow $M_\tau$, the
nonnegative function $f(x,\tau)$ satisfies
\[ \partial_\tau f \leq \Delta f - \frac{1}{2} x\cdot \nabla f +
  \mathcal{E}, \]
for $\tau\in [0,1]$, where $\mathcal{E}\leq \delta$.
There is a constant $C$ (depending on the dimension and entropy
bounds) such that if $R > 1$ and
$\mathcal{E}=0$ on $B_R(x_0)$, then 
we have
\[ f(x_0, 1)\leq C e^{|x_0|^2/4} \int_{M_0} f(x, 0) e^{-|x|^2/4}\,
  d\mathcal{H}^n + C\delta e^{-R^2/15}. \]
\end{prop}
\begin{proof}
  The function $f$ satisfies the differential inequality
  \[ \partial_t f \leq \Delta f + \frac{1}{|t|}\mathcal{E}, \]
  for $t\in [-1, -e^{-1}]$ along the mean curvature flow $L_t$, where
  $L_{-1} = M_0$. We use the monotonicity formula centered at the point
  $(\tilde x_0, t_0)$, with $t_0 = -e^{-1}$, and $\tilde x_0 = e^{-1/2}x_0$:
  \[ \label{eq:mon1} \frac{d}{dt} \int_{L_t} f(x,t)\, \rho_{\tilde x_0, t_0}(x,t) \leq
    \int_{L_t} |t|^{-1}\mathcal{E}\, \rho_{\tilde x_0,t_0}(x,t), \]
  where $\rho_{\tilde x_0, t_0}$ is the backwards heat kernel 
  centered at $(\tilde x_0, t_0)$:
  \[ \rho_{\tilde x_0, t_0}(x,t) = \frac{1}{(4\pi(t_0-t))^{n/2}}
    e^{-\frac{|x-\tilde x_0|^2}{4(t_0-t)}}. \]
  Note that by assumption, on $L_t$ for $t < -e^{-1}$,
  the support of $\mathcal{E}$ is outside of the $e^{-1/2}R$-ball
  around $\tilde x_0$ (here $e^{-1/2}$ is the scaling factor between
  $L_{t_0}$ and $M_1$). It follows that for $t\in [-1, t_0]$ we have
  \[  \int_{L_t} |t|^{-1}\mathcal{E}\, \rho_{\tilde x_0,t_0}(x,t) \leq
    C \delta
    e^{-e^{-1}R^2/5}, \]
  if $R > 1$. Integrating \eqref{eq:mon1} we obtain
  \[ \label{eq:mon2}  f(\tilde x_0, t_0) \leq \int_{L_{-1}} f(x,t)
    \rho_{\tilde x_0, t_0}(x,t) + C \delta
    e^{-R^2/15}, \]
  where we are viewing $f$ as a function on $L_{t_0}$. 
  As in the proof of \cite[Lemma 3.5]{LSSz22} we have
  \[ \frac{\rho_{\tilde x_0, t_0}(x,-1)}{\rho_{0,0}(x,-1)} \leq C_n \exp\left(
    \frac{|\tilde x_0|^2}{4(-t_0)}\right) = C_n
  \exp\left(\frac{|x_0|^2}{4}\right).  \]
It follows that
  \[ \int_{L_{-1}} f(x,t) \rho_{\tilde x_0, t_0}(x,t) \leq C
    e^{|x_0|^2/4} \int_{M_0} f(x,0) e^{-|x|^2/4}. \]
   Putting this together with \eqref{eq:mon2}, we get the result. 
\end{proof}

Combining this with Proposition~\ref{prop:logSobolev} we have the
following. 
\begin{cor}\label{cor:logSobolev}
  Suppose that we have $f$ as in
  Proposition~\ref{prop:logSobolev} on the interval $[\tau_0, \tau]$,
  with $\tau \geq \tau_0+1$. Assume that $b_0 \geq 1/2$ and
  $\epsilon < 1$. In addition, suppose that for $s\in [\tau-1,\tau]$ we have
  $\mathcal{E}(x,s)\leq \delta$, and for a given $x_0$ we have
  $\mathcal{E}(x,s)=0$ for $x\in B_R(x_0)$. Then there is a constant $C=C(b_0)$
  such that 
  \[ \label{eq:cor9eq} f(x_0, \tau) &\leq C 
    e^{\frac{|x_0|^2}{4p(\tau-\tau_0-1)}}\left( e^{(\epsilon + b_0)(\tau-\tau_0)} \Vert f(\cdot,
    \tau_0)\Vert_{L^2} + \int_{\tau_0}^\tau e^{(\epsilon + b_0)(\tau
      - s)} \Vert \mathcal{E}(\cdot, s)\Vert_{L^{p(s-\tau_0)}}\,
    ds\right) \\
  &\quad + C \delta e^{-\frac{R^2}{15p(\tau-\tau_0-1)}}.  \]
\end{cor}
\begin{proof}
  First we apply Proposition~\ref{prop:logSobolev} to the flow
  translated in time by $\tau_0$, to estimate the $L^p$ norm of $f$ at
  time $\tau-1$. We find that
  \[ \label{eq:100} \Vert f(\cdot, \tau-1)\Vert_{L^{p(\tau-\tau_0-1)}} &\leq
     e^{(\epsilon+b_0)(\tau-\tau_0-1)} \Vert 
     f(\cdot, \tau_0)\Vert_{L^{2}}  \\
     &\qquad + \int_{\tau_0}^{\tau-1}
     e^{(\epsilon+b_0)(\tau-1 - 
       s)} \Vert \mathcal{E}(\cdot, s)\Vert_{L^{p(s-\tau_0)}}\, ds, \] 
   Next we consider the function $F = f^p$, where for simplicity we
   wrote $p=p(\tau-\tau_0-1) \geq 2$. The function $F$ satisfies the
   differential inequality
   \[ \partial_\tau F \leq \Delta F - \frac{1}{2}x\cdot \nabla F -
     p(p-1) f^{p-2}|\nabla f|^2 + p f^{p-1}a \cdot \nabla f + pb_0 f^p
     + p f^{p-1}\mathcal{E}. \]
   Using that $b_0\geq 1/2$ and $|a| < 1$, we can absorb the
   $pf^{p-1}a\cdot \nabla f$ term using the negative gradient squared,
   and the $f^p$ term. We also have $p f^{p-1}\mathcal{E} \leq
   (p-1)f^p + \mathcal{E}^p$. Therefore
   \[ \partial_\tau F \leq \Delta F - \frac{1}{2}x\cdot \nabla F + Cpb_0
     F + \mathcal{E}^p, \]
   and so if $\tilde{F} = e^{-Cb_0p\tau} F$, we get
   \[ \partial_\tau \tilde{F} \leq \Delta \tilde F - \frac{1}{2}x\cdot
     \nabla \tilde F + e^{-Cb_0p\tau} \mathcal{E}^p. \]
   We apply Proposition~\ref{prop:L2pointwise}, on the time interval
   $[\tau-1, \tau]$. The conclusion is that
   \[ e^{-Cb_0p} F(x_0, \tau) \leq Ce^{|x_0|^2/4} \Vert F(\cdot,
     \tau-1)\Vert_{L^1} + C \delta^p e^{-R^2/15}. \]
   We have $\Vert F(\cdot, \tau-1)\Vert_{L^1} = \Vert f(\cdot,
   \tau-1)\Vert_{L^p}^p$. Taking $p^{th}$ roots we get the required
   result. 
\end{proof}

We want to apply this result to a function $u(x,\tau)$, on the region where the flow
$M_\tau$ is the graph of $u(x,\tau)$ over $\mathcal{C}_{n,k}$. Such a
function $u$, cut off at a suitable scale, will satisfy a differential
inequality like \eqref{eq:driftheatineq} as long as $M_\tau$ has good
graphicality over $\mathcal{C}_{n,k}$. The following result is stated
in more generality than needed here, since we will also use it later.
\begin{prop}\label{prop:subsolpointwise}
Given $\epsilon, \beta_0 > 0$, there is a constant $C_0(\epsilon,
\beta_0) > 10$ with the following property. Suppose that $\beta,
\beta' \in (\beta_0, 1/2]$, with $\beta' \leq \beta$, $K\in (0,20)$, and
let $M_\tau$ be a RMCF for $\tau\in [T_0, T_1]$ such that $KT_0^{\beta'}
> C_0$. Suppose that on the $2K\tau^\beta$ ball the function $f \geq
0$ satisfies the differential inequality
\[ \partial_\tau f \leq  \Delta f - \frac{1}{2}x\cdot \nabla f +
  a \cdot \nabla f + b f+ \mathcal{E}, \]
such that $\mathcal{E}(x,\tau)\leq K^{-1}\tau^{-\beta}$ is supported on the
annulus where $|x|\in (4/3 K\tau^\beta, 3/2 K\tau^\beta)$. We assume
that $|a| < \epsilon, |b| < b_0$. Then for $\tau\in [\frac{5}{3}T_0, T_1]$, if
$|x| < 7/6 K\tau^{\beta'}$ we have
\[ f(x,\tau) \leq C\tau^{2\beta'(\epsilon+b_0)} \Vert f(\cdot,
  \tau_0)\Vert_{L^2} + C K^{-1}\tau^{-\beta} e^{-\frac{1}{10^4}
    \tau^{2(\beta-\beta')}}, \]
for a uniform constant $C$, and some $\tau_0\in [\tau/2, \tau-1]$. 
\end{prop}
\begin{proof}
  We will apply Corollary~\ref{cor:logSobolev}. We let
$\tau\in [2T_0, T_1]$, and choose $\tau_0$ so that
\[ e^{\tau-\tau_0} = K^2 \tau^{2\beta'}.\]
Since $K^2T_0^{2\beta'} > 10$, we have $\tau_0 < \tau-1$. At the same
time, we have
\[ \tau_0 = \tau -2\log K - 2\beta' \log \tau > \frac{2}{3}\tau, \]
if $T_0$ is sufficiently large since $K, \beta'$ are bounded above. So
$\tau_0\in [\frac{2}{3}\tau, \tau-1]$, and in particular we can ensure
that $\tau_0 > T_0+1$ if $\tau > \frac{5}{3} T_0$. 

For $\tau\in [T_0, T_1]$ we also have 
\[ \label{eq:pbound10} \frac{1}{4} e^{\tau-\tau_0} \leq p(\tau-\tau_0-1) = 1+\epsilon +
  (1-\epsilon) e^{\tau-\tau_0-1} \leq e^{\tau-\tau_0}, \]
if $\epsilon$ is small. It follows that as long as $|x|\leq K \tau^{\beta'}$,
we have
\[ \label{eq:101} e^{\frac{|x|^2}{4p(\tau-\tau_0-1)}} \leq e. \]

Note that for $s\in [\tau-1, \tau]$ we have that $\mathcal{E}\leq
K^{-1}\tau^{-\beta}$, and $\mathcal{E}$ is supported outside of the
$\frac{4}{3} K(\tau-1)^\beta$-ball. We apply Corollary~\ref{cor:logSobolev} at time
$\tau$, and $|x_0| < \frac{7}{6} K\tau^{\beta'}$. If $T_0$ is sufficiently large,
then for $s\in [\tau-1,\tau]$, $\mathcal{E}$ vanishes on $B_R(x_0)$
for $R = \frac{1}{8} K\tau^{\beta}$. For this $R$ we have
\[ e^{-\frac{R^2}{15 p(\tau-\tau_0-1)}} 
  \leq  e^{-\frac{K^2\tau^{2\beta}}{15\cdot 64 e^{\tau-\tau_0}}} 
  \leq  e^{-\frac{1}{10^4}\tau^{2(\beta - \beta')}}. \]

From Corollary~\ref{cor:logSobolev} we therefore get
\[ f(x_0,\tau)&\leq C e^{(b_0+ \epsilon)(\tau-\tau_0)}
  \Vert f(\cdot, \tau_0)\Vert_{L^2} + C\int_{\tau_0}^\tau e^{(b_0 + \epsilon)(\tau
    - s)} \Vert \mathcal{E}(\cdot, s)\Vert_{L^{p(s-\tau_0)}}\, ds \\
  &\quad + C K^{-1}\tau^{-\beta} e^{-\frac{1}{10^4}\tau^{2(\beta -
      \beta')}}, \]
for $|x_0| \leq \frac{7}{6} K\tau^{\beta'}$ and $\tau\in [\frac{5}{3}T_0,T_1]$.

By our bound for $\mathcal{E}$ we have
\[ \Vert \mathcal{E}(\cdot, s)\Vert_{L^{p(s-\tau_0)}} \leq C K^{-1}
  s^{-\beta} e^{-K^2 s^{2\beta} / 5p(s-\tau_0)} \]
We have $p(s-\tau_0) \leq 2e^{s-\tau_0}$, and also $s \geq \tau/2$, so
\[ \Vert \mathcal{E}(\cdot, s)\Vert_{L^{p(s-\tau_0)}} &\leq C K^{-1}
  \tau^{-\beta} e^{-\frac{K^2 \tau^{2\beta}}{40e^{\tau-\tau_0}e^{s-\tau}}}
  \\
  &= CK^{-1}\tau^{-\beta} e^{-\frac{1}{40} 
    \tau^{2(\beta-\beta')}e^{\tau-s}} \\
  &\leq CK^{-1}\tau^{-\beta} e^{-\frac{1}{40}\tau^{2(\beta-\beta')}}\exp\left( - \frac{1}{80}
    (\tau-s)^2\right),\]
where $\tau\geq s$ and we also used
$e^{\tau-s} \geq  1 + \frac{1}{2}(\tau-s)^2$.
Using this in the estimate for $f$ above, we get
\[ \label{eq:fest30} f(x,\tau) &\leq C (K^2\tau^{2\beta'})^{b_0
    +\epsilon} \Vert f(\cdot, \tau_0)\Vert_{L^2} +
  CK^{-1}\tau^{-\beta} e^{-\frac{1}{40}\tau^{2(\beta-\beta')}} \int_{\tau_0}^\tau
  e^{(b_0  +\epsilon)(\tau-s) - \frac{1}{80}(\tau-s)^2}\, ds \\
  &\qquad\qquad +  CK^{-1}\tau^{-\beta} e^{-\frac{1}{10^4}\tau^{2(\beta-\beta')}} \\
  &\leq C\tau^{2\beta' (b_0 + \epsilon)}\Vert f(\cdot,
  \tau_0)\Vert_{L^2}  + CK^{-1}\tau^{-\beta} e^{-\frac{1}{10^4}\tau^{2(\beta-\beta')}},\]
for $|x| <\frac{7}{6} K\tau^{\beta'}$, and $\tau\in [\frac{5}{3}T_0,
T_1]$.
\end{proof}

In order to apply this result to the graphicality function, we need
the following extension property, proven in
\cite[Proposition 10]{Sz26}.  It is a consequence of pseudolocality, see
Ilmanen-Neves-Schulze~\cite{INS19}, as well as the interior estimates
of Ecker-Huisken~\cite{EH91}. See also Sun-Xue~\cite[Theorem
2.4]{SX22}. 

\begin{prop}\label{prop:pseudo}
  Given $\epsilon > 0$ there exist $\delta_1(\epsilon), C_1(\epsilon)>
  0$,  satisfying the following. Suppose that $M_\tau$ is a
  rescaled mean curvature flow such that $M_0$ is a $\delta_1$-graph
  over $\mathcal{C}_{n,k}$ on the ball $B_R$ for some $R > C_1$. Then for
  $\tau\in [0,10]$, $M_\tau$ is an $\epsilon$-graph over $\mathcal{C}_{n,k}$
  on the ball $B_{e^{\tau/2}(R-C_1)}$. 
\end{prop}

Using this pseudolocality result and the estimate from
Corollary~\ref{cor:logSobolev}, we can 
obtain improved graphicality estimates for the flow $M_\tau$ in terms
of the $L^2$ bound from Proposition~\ref{prop:CMgraphicality}. We use
the following estimate, which we will also need later on, and
is stated in a more general form than needed here.

\begin{prop}\label{prop:tbetagraph10}
  Given small $\epsilon, \beta_0 > 0$, there are $\epsilon_0(\epsilon) > 0$ 
  and $C_0(\epsilon, \beta_0) > 10$ with the following properties. Suppose that
  $\beta\in (\beta_0, \alpha/2 - \beta_0)$, $K\in   (0,20)$, and let $M_\tau$ be a
  RMCF for $\tau\in [T_0, T_1]$ for some $T_0 >  C_0$, such that
  also $KT_0^\beta > C_0$. Suppose that
  $M_\tau$ is $\epsilon_0$-graphical over $\mathcal{C}_{n,k}$ on
  $B_{K\tau^\beta}$, and the graphicality function $u(x,\tau)$
  satisfies $\Vert u(\cdot, \tau)\Vert_{L^2} \leq \tau^{-\alpha}$ for
  some $\alpha\in (0,2)$. Then for $\tau\in [2T_0, T_1]$ we have
  \[\label{eq:niuest} |\nabla^i u(x, \tau)| \leq C \tau^{2\beta (2\epsilon+1) - \alpha}
    + CK^{-1} \tau^{-\beta}, \]
  for $i=0, \ldots, 4$, and $|x| \leq K\tau^\beta$.
\end{prop}
It may be helpful to keep in mind that when we apply this result later
on, $\epsilon, \beta_0$ can be thought of as fixed very small constants,
and the important part is to control the estimate as
$K\to 0$. For our application at this point we will take $K=1$. 
\begin{proof}
First we apply pseudolocality, Proposition~\ref{prop:pseudo}.
Given $\epsilon_1 > 0$, we can assume
that $\epsilon_0$ is small enough so that we have a constant $C_1 > 0$
with the following property. Since for $\tau\in [T_0, T_1]$, $M_{\tau}$ is
$\epsilon_0$-graphical over $\mathcal{C}$ on $B_{K\tau^\beta}$, we
have that for all $\tau\in [T_0+1, T_1]$, $M_\tau$ is
$\epsilon_1$-graphical over $\mathcal{C}$ on $B_{R(\tau)}$,
where
\[ R(\tau) = e^{1/2} (K(\tau-1)^\beta - C_1). \]
If $KT_0^\beta$ is sufficiently large, depending on $C_1$ (which in
turn depends on $\epsilon_0$), then we will have $R(\tau) >
\frac{3}{2} K\tau^\beta$. So for $\tau\in [T_0+1, T_1]$, $M_\tau$ is
$\epsilon_1$-graphical on the larger balls 
$B_{1.5 K\tau^\beta}$.

Define a standard cutoff function $\chi$ such that $\chi(s) = 0$ for
$s > 3/2$, and $\chi(s)=1$ for $s\leq 4/3$. Let us write $M_\tau$ as the
graph of $u(x,\tau)$ over $\mathcal{C}$ on $B_{1.5K \tau^\beta}$, and
define
\[ \tilde u(x, \tau) = \chi(|x| / K\tau^\beta) u(x, \tau). \]
We can extend $\tilde u$ by zero outside of the $1.5 K
\tau^\beta$-ball. The function $u$ satisfies the equation
\[ \partial_\tau u = \Delta u - \frac{1}{2} x \cdot \nabla u + u +
  Q(u, \nabla u, \nabla^2 u), \]
where $Q$ contains terms that are at least quadratic in its entries,
but $\nabla^2 u$ appears at most linearly. 
Since $\Vert u\Vert_{C^4} \leq \epsilon_1$, if
$\epsilon_1$ is small enough, we can write the equation of $u$ over
$\mathcal{C}$ as
\[ \partial_\tau u = \Delta u - \frac{1}{2}x\cdot \nabla u + u +
  a_1\cdot \nabla u + a_2 u. \]
Given any $\epsilon > 0$ we can assume that $\epsilon_1$ is small
enough so that $a_1(x,\tau),a_2(x,\tau)$ satisfy $|a_1|, |a_2| <
\epsilon$. This determines $\epsilon_1$, and in turn $\epsilon_0$, in
terms of $\epsilon$.  It follows that $\tilde u$ satisfies the same equation, with some
additional terms when derivatives land on  the cutoff function:
\[ \label{eq:utildeeq} \partial_\tau \tilde{u} = \Delta \tilde{u} - \frac{1}{2}x\cdot \nabla \tilde{u} + \tilde{u}+
  a_1\cdot \nabla \tilde{u} + a_2 \tilde{u} + \mathcal{E}. \]
Note that
derivatives of $\chi(|x| / K\tau^\beta)$ are at most $C K^{-1}\tau^{-\beta}$,
and so $\mathcal{E}(x,\tau)$ is supported on the annulus $|x| \in (4/3 K
\tau^\beta, 3/2 K
\tau^\beta)$ and is bounded by $K^{-1} \tau^{-\beta}$ if $\epsilon_1$
is small enough.

We apply Proposition~\ref{prop:subsolpointwise} to
$f=|\tilde{u}|$, with $b_0 = 1+\epsilon$, and with $\beta'=\beta$,
$K=1$. 
The conclusion is that
\[ f(x,\tau) \leq C\tau^{2\beta(2\epsilon+1)-\alpha} + C
  \tau^{-\beta}. \]
This implies the derivative estimates \eqref{eq:niuest} using
parabolic Schauder estimates.  
\end{proof}

This implies the following estimate for the graphicality radius.
\begin{prop}\label{prop:tbetagraph11}
  There exist $\beta, \delta > 0$ with the following property. Suppose
  that $M_\tau$ is a RMCF converging to $\mathcal{C}_{n,k}$, and $M_0$
  is $\delta$-graphical over $\mathcal{C}_{n,k}$ on
  $B_{\delta^{-1}}$. Then $M_\tau$ is $\tau^{-\beta}$-graphical over
  $\mathcal{C}_{n,k}$ on $B_{\tau^\beta}$ for all $\tau > 0$.
\end{prop}
\begin{proof}
  Let $\beta  = \alpha/4$, for the $\alpha$ in
  Proposition~\ref{prop:CMgraphicality}. Given any $T_1 > 0$, if we choose $\delta$ 
  sufficiently small, then the conclusion holds for all $\tau \leq
  T_1$. We will show that if $\delta, T_1$ are chosen suitably, then
  the result can be extended for all $\tau > T_1$ as well. 

  We will apply Proposition~\ref{prop:tbetagraph10} with $\epsilon =
  \frac{1}{20}$, $\beta_0 = \alpha/8$, $\beta=\alpha/4$, and $K=1$. We let $T_0$ be
  large enough for the proposition to apply, and then assume that the
  result we are trying to prove already holds for all $\tau \leq T_1$, with $T_1 > 3T_0$. By
  choosing $\delta$ sufficiently small, we can ensure that $M_\tau$ is
  $\epsilon_0$-graphical over $\mathcal{C}_{n,k}$ on $B_{\tau^\beta}$
  for $\tau\leq T_1$, so the
  hypotheses of Proposition~\ref{prop:tbetagraph10} hold.
  We find that the graphicality function $u$ satisfies
  \[ \label{eq:nu20} |\nabla^i u(x, \tau)| &\leq C \tau^{2\beta(2\epsilon+1)-\alpha} +
    C\tau^{-\beta} \leq C\tau^{-\beta},\]
  for $|x| \leq \tau^\beta$, for $\tau\in [2T_0, T_1]$.  For any
  $\epsilon_1 > 0$, once $\tau$
  (i.e. $T_0$) is large enough depending on $\epsilon_1$,
  then this means that for $\tau\in [T_1 -1, T_1]$, $M_\tau$ is
  $\epsilon_1$-graphical over $\mathcal{C}$ on
  $B_{\tau^{\beta}}$. Choosing $\epsilon_1$ small, depending on $\epsilon_0$, we
  can then use pseudolocality, Proposition~\ref{prop:pseudo}, to show that for such $\tau$,
  $M_{\tau+1}$ is $\epsilon_0$-graphical on the ball of radius
  $e^{1/2}(\tau^\beta - C_1)$. If $\tau$ is large, this will contain
  the $(\tau+1)^\beta$ ball. Now we are again in the position to apply
  Proposition~\ref{prop:tbetagraph10}, on the larger time interval
  $[T_0, T_1+1]$, and obtain \eqref{eq:nu20}. Repeating this argument,
  we obtain that $M_\tau$ is $C\tau^{-\beta}$-graphical on
  $B_{\tau^\beta}$ for all $\tau\geq 1$. Letting $\beta' < \beta$, we
  have $C\tau^{-\beta} < \tau^{-\beta'}$ for large $\tau$, so by
  replacing $\beta$ with a smaller constant we can remove the $C$
  factor in the conclusion. 
\end{proof}

\section{The $L^2$ normal form}\label{sec:L2normal}
In the previous section we saw that if the rescaled flow $M_\tau$
has a cylindrical singularity at infinity, then the $L^2$-norm of the
graphicality function $u$ over the cylinder decays at a rate
$\tau^{-\alpha}$ for some $\alpha > 0$, which in turn is related to
good graphicality on balls of size $\tau^\beta$ for $\beta < \alpha /
2$. In this section our goal is to improve the $L^2$-asymptotics of the graphicality
function to obtain an error of order $o(\tau^{-1})$, as done in
\cite[Section 5]{SX22}. As can be seen from
the normal form, Theorem~\ref{thm:SXnormal}, for this we need to
identify the term of order $\tau^{-1}$ giving the leading order
behavior. It is worth noting that even if $M_\tau$ has good
graphicality over $\mathcal{C}_{n,k}$ on a ball of radius larger than
$\tau^\beta$, the contribution to the $L^2$-norm of $u$
from outside of $B_{\tau^\beta}$ is bounded
by $e^{-\tau^{2\beta}/9}$ for large
$\tau$, and this is of much lower order than $\tau^{-1}$. So at this
point we can still work on balls of radius $\tau^\beta$ for small
$\beta > 0$. Once we have the $L^2$-asymptotics identified up to an
$o(\tau^{-1})$ error, in the next section we will be able to obtain
graphicality on balls of radius $\tau^{1/2}$, over suitable models. 

To obtain an effective $o(\tau^{-1})$ estimate for the $L^2$-asymptotics, our
hypothesis can no longer be simply closeness to the cylinder
$\mathcal{C}_{n,k}$ at time zero. This is because as $\tau\to\infty$
the flow could become ``less degenerate'', with a larger value of
$\ell$, than what it appears to be
initially. In this case additional terms of size $(T_1 + \tau)^{-1}$
could appear, which would not be detectable at times $\tau \ll
T_1$. In addition we need to allow for rotations in the $y$-directions
to correctly identify the degenerate and nondegenerate directions. 

We have the following.
\begin{prop} \label{prop:L2normalform}
  Suppose that $M_\tau$ is a rescaled flow converging to
  $\mathcal{C}_{n,k}$ at infinity, and we can
  write $M_\tau$ as the graph of $u(\cdot, \tau)$
  over $\mathcal{C}_{n,k}$ on $B_{\tau^\beta}$, for all $\tau\geq 0$.
  For any $T_0 > 1$ there is an $L_0(T_0) > 0$ with the
  following property. For some $\ell \geq 0$ let us define
\[ \label{eq:L2decay10} E(\tau) := \left\Vert u(\cdot, \tau) -
      \frac{\sqrt{2(n-k)}}{4\tau}\sum_{i=1}^\ell
      (y_i^2-2)\right\Vert_{L^2(B_{\tau^\beta})}, \]
  and suppose that for all $\tau\in [T_0, T_0 + L_0]$
  we have $E(\tau) < \tau^{-3/2}$.   Then we have the following.
  \begin{itemize}
    \item[(i)] If $M_\tau$ has an $\ell'$-degenerate
      $(n,k)$-cylindrical singularity at infinity, then $\ell' \geq
      \ell$.
    \item[(ii)] If $M_\tau$ has an $\ell$-degenerate
      $(n,k)$-cylindrical singularity  at  infinity, then after a
      rotation of size $\Psi(T_0^{-1})$ in the $y$-coordinates, we
      have $E(\tau) < \tau^{-19/10}$ for all $\tau \geq T_0 + L_0$. In
      particular the rotated flow $QM_\tau$ has an $\ell$-degenerate
      $\mathcal{C}_{n,k}$-singularity at infinity. 
   \end{itemize}
\end{prop}

In other words, if we have $\ell$ directions that look sufficiently
nondegenerate on a large range of scales, then actually for all
sufficiently large $\tau$ we must have at least $\ell$ nondegenerate
directions. And if no further nondegenerate directions appear, then we
have a good estimate in the nondegenerate directions. Note that when
$\ell = 0$, in case (ii) we actually have that $E(\tau)$ decays
exponentially fast (see \cite[Proposition 5]{Sz26}), but we do not
need this here. 

The proof of this will require a number of preliminary steps. We first recall the
discrete frequency monotonicity results in \cite{SWX25}, and in
particular the non-concentration estimate \cite[Corollary
3.3]{SWX25}. Let us define the $L^2$-distance from $\mathcal{C}_{n,k}$ by
\[ \mathbf{d}_{\mathcal{C}_{n,k}}(M)^2 = \int_M
  \overline{\mathrm{dist}}_{n,k}(x)^2\, e^{-|x|^2/4}\,
  d\mathcal{H}^n, \]
where $\overline{\mathrm{dist}}_{n,k}$ is a regularized distance
function from $\mathcal{C}_{n,k}$. For our purposes we can take
$\overline{\mathrm{dist}}_{n,k}(x) = \min\{ d(x, \mathcal{C}_{n,k}), 1\}$,
which is uniformly equivalent to the definition used in \cite{SWX25}. 
This $L^2$-distance satisfies the following
non-concentration estimate.

\begin{prop}[See Corollary 3.3 in \cite{SWX25}] \label{prop:SWXnonconc}
  There is a dimensional constant $C_n$ with the following
  property. If $M_\tau$ is a rescaled mean curvature flow then for $\tau\in [0,2]$ we have
  \[ \int_{M_\tau}\overline{\mathrm{dist}}_{n,k}(x)^2 ( 1+ \tau |x|^2)\, e^{-|x|^2/4}\,
    d\mathcal{H}^n \leq C_n \int_{M_0} \overline{\mathrm{dist}}_{n,k}(x)^2\,
    e^{-|x|^2/4}\, d\mathcal{H}^n. \]
\end{prop}

As a consequence we have the following, which is essentially
\cite[Remark 3.8]{SWX25}. Note that in \cite[Remark 3.8]{SWX25} this result is stated for
  $L=1$, which is a stronger result, but may need a more careful
  definition of the distance function. 
\begin{prop}\label{prop:3ann2}
  Let $\epsilon\in (0,1/2)$. There exist $L_0, \delta > 0$, depending on
  $\epsilon$, such
  that if $L > L_0$ and  $M_\tau$ is a rescaled flow that is $\delta$-graphical over
  $\mathcal{C}_{n,k}$ on $B_{\delta^{-1}}$ for $\tau\in [0,2L]$, then
  \[ \mathbf{d}_{\mathcal{C}_{n,k}}(M_L) \geq e^{L\epsilon}
    \mathbf{d}_{\mathcal{C}_{n,k}}(M_0)\quad \text{ implies }\quad
    \mathbf{d}_{\mathcal{C}_{n,k}}(M_{2L}) \geq e^{L\epsilon} 
    \mathbf{d}_{\mathcal{C}_{n,k}}(M_L).  \]
\end{prop}
\begin{proof}
  Let $L > 0$, and suppose for a contradiction that we have a sequence of
  flows $M^i_\tau$ converging to $\mathcal{C}_{n,k}$ locally smoothly
  on compact subsets of $[0,2L]\times \mathbb{R}^{n+1}$ such that
  \[ \mathbf{d}_{\mathcal{C}_{n,k}}(M^i_L) \geq e^{L\epsilon}
    \mathbf{d}_{\mathcal{C}_{n,k}}(M^i_0), \text{ and }
    \mathbf{d}_{\mathcal{C}_{n,k}}(M^i_{2L}) \leq e^{L\epsilon}
    \mathbf{d}_{\mathcal{C}_{n,k}}(M^i_L). \]
   we will show that if $L$ is large enough, then this is a
   contradiction. We can write the $M^i_\tau$ as graphs of $u^i(x,
   \tau)$ over larger and larger subsets of $\mathcal{C}_{n,k}$, and
   if we let $\mathbf{d}_i = \mathbf{d}_{\mathcal{C}_{n,k}}(M^i_L)$,
   then $\mathbf{d}_i^{-1}u^i$ converges locally smoothly on compact
   subsets of $(0,2L]\times\mathbb{R}^{n+1}$ to a solution $u^\infty$
   of the linearized equation $\partial_\tau u^\infty = \mathcal{L}_{\mathcal{C}}
   u^\infty$ on $\mathcal{C}_{n,k}$. Note that we may not have
   convergence at $\tau=0$. However from
   Proposition~\ref{prop:SWXnonconc} we do have
   $\mathbf{d}_{\mathcal{C}_{n,k}}(M^i_1) \leq C
   \mathbf{d}_{\mathcal{C}_{n,k}}(M^i_0)$, and so from our hypothesis
   we have
   \[ \mathbf{d}_i = \mathbf{d}_{\mathcal{C}_{n,k}}(M^i_L) \geq e^{L\epsilon}C^{-1}
     \mathbf{d}_{\mathcal{C}_{n,k}}(M^i_1). \]
   Using Proposition~\ref{prop:SWXnonconc} repeatedly, we also have
   $\mathbf{d}_{\mathcal{C}_{n,k}}(M^i_\tau) \leq C_L
   \mathbf{d}_{\mathcal{C}_{n,k}}(M^i_0)$ for all $\tau\in [0, 2L]$,
   where the constant depends on $L$.  Applying
   Proposition~\ref{prop:SWXnonconc} on the time interval $[L-1, L]$
   we also find that for every $R > 1$
   \[ \int_{M^i_L\setminus B_R} \dist(x)^2
     e^{-|x|^2/4}\,d\mathcal{H}^n \leq \frac{C_n}{R^2}  \dd(M_{L-1})^2
     \leq \frac{C_n C_L}{R^2} \mathbf{d}_i^2. \]
   It follows from this, by choosing larger and larger $R$ (we can
   view $L$ as fixed here), that $\Vert u^\infty(\cdot, L)\Vert_{L^2}
   = 1$, and at the same time $\Vert u^\infty(\cdot, 1)\Vert_{L^2}
   \leq \lim \mathbf{d}_i^{-1} \dd(M^i_1)$ and $\Vert u^\infty(\cdot,
   2L)\Vert_{L^2} \leq \lim \mathbf{d}_i^{-1} \dd(M^i_{2L})$. It
   follows that
   \[ \Vert u^\infty(\cdot, L)\Vert_{L^2} \geq e^{L\epsilon} C^{-1}
     \Vert u^\infty(\cdot, 1)\Vert_{L^2}, \text{ and } \Vert
     u^\infty(\cdot, 2L)\Vert_{L^2} \leq e^{L\epsilon} 
     \Vert u^\infty(\cdot, L)\Vert_{L^2}. \]
  Note, however, that since there are no homogeneous solutions of the
  linearized equation with growth rate $\epsilon$, there exists some
  $\epsilon' < \epsilon$ such that if
  $\Vert
     u^\infty(\cdot, L)\Vert_{L^2} \geq e^{-L\epsilon} 
     \Vert u^\infty(\cdot, 2L)\Vert_{L^2}$, then $\Vert
     u^\infty(\cdot, 1)\Vert_{L^2} \geq e^{-(L-1)\epsilon'} 
     \Vert u^\infty(\cdot, L)\Vert_{L^2}$. If we choose $L$ large
     enough so that $e^{(L-1)\epsilon'}< e^{L\epsilon}C^{-1}$, then we
     get the required contradiction. 
   \end{proof}

   An immediate corollary of this is the following (see
   \cite[Proposition 6]{Sz26}). 
   \begin{cor}\label{cor:slowgrowth}
     Let $\epsilon > 0$. There are $L_0, \delta > 0$ depending on
     $\epsilon$ such that if $M_0$ is $\delta$-graphical over
     $\mathcal{C}_{n,k}$ on $B_{\delta^{-1}}$, and $M_\tau$ converges
     to a rotation $QC_{n,k}$ at infinity, then
     \[ \dd(M_{\tau+L}) \leq e^{L\epsilon} \dd(M_\tau) \]
     for all $\tau \geq 0$ and $L \geq L_0$. 
   \end{cor}

Next we need to quotient out by the action of rotations. Suppose
that $M_\tau$ is a rescaled mean curvature flow for $\tau\geq 0$ which
is $\delta$-graphical over $\mathcal{C}_{n,k}$ on $B_{\delta^{-1}}$ for all
$\tau$, for some small $\delta > 0$. For each $\tau\geq 0$ we define
$\tilde{M}_\tau$ to be given by a rotation $S_\tau M_\tau$ in such a
way that the corresponding graphicality functions $v(x,
\tau)$ defined on $B_{\delta^{-1}}$ are $L^2$-orthogonal to the
eigenfunctions $y_i z_\alpha$ 
corresponding to rotations (which do not fix the cylinder). Note that we have
$\mathbf{d}_{\mathcal{C}_{n,k}}(\tilde{M}_\tau) \leq C
\mathbf{d}_{\mathcal{C}_{n,k}}(M_\tau)$ for a uniform $C$. It is worth
noting that the definition of $S_\tau$ depends on the choice of small
$\delta$, and we may choose this smaller as needed. 

We need the following variant of Proposition~\ref{prop:3ann2} that applies to
this modified flow. 
\begin{prop}\label{prop:3ann3}
  Let $-\epsilon \not\in \sigma(\mathcal{C}_{n,k})$. There exist $L_0, \delta > 0$
  depending on $\epsilon$ with the following property. Suppose that $M_\tau$ is
  $\delta$-graphical over $\mathcal{C}$ on $B_{\delta^{-1}}$ for
  $\tau\in [0,2L]$ and $L > L_0$. Let $\tilde{M}_\tau = S_\tau M_\tau$ as
  above. Then 
  \[ \mathbf{d}_{\mathcal{C}_{n,k}}(\tilde{M}_L) \geq e^{-\epsilon L}
    \mathbf{d}_{\mathcal{C}_{n,k}}(\tilde{M}_0). \]
  implies
  \[ \mathbf{d}_{\mathcal{C}_{n,k}}(\tilde{M}_{2L}) \geq
    e^{-\epsilon L} \mathbf{d}_{\mathcal{C}_{n,k}}(\tilde{M}_L). \]
\end{prop}
\begin{proof}
  We argue as above, supposing that we have a sequence $M^i_\tau$
  converging to $\mathcal{C}$, and we fix $L > 0$, to be chosen below. 
  Without loss of generality we can assume that $\tilde{M}^i_0 =
  M^i_0$, i.e. the corresponding rotations satisfy $S^i_0 = Id$. Let us define
  $\mathbf{d}_i = \mathbf{d}_{\mathcal{C}_{n,k}}(M^i_L)$, and suppose
  that 
 \[ \mathbf{d}_{\mathcal{C}_{n,k}}(\tilde{M}^i_L) \geq e^{-L\epsilon}
    \mathbf{d}_{\mathcal{C}_{n,k}}(M^i_0), \text{ and }
    \mathbf{d}_{\mathcal{C}_{n,k}}(\tilde{M}^i_{2L}) \leq e^{-L\epsilon}
    \mathbf{d}_{\mathcal{C}_{n,k}}(\tilde{M}^i_L). \]
   We have
  $|S^i_L - Id| \leq C\mathbf{d}_i$ and also, by using
  Proposition~\ref{prop:SWXnonconc} repeatedly as in the previous
  proof, we have 
  $\mathbf{d}_{\mathcal{C}_{n,k}}(M^i_\tau) \leq C_L \mathbf{d}_i$ for
  all $\tau\in [0,2L]$, with a constant depending on $L$. As a consequence 
  $|S^i_\tau- Id| \leq C_L \mathbf{d}_i$.
  As in the previous proof we also have
  \[ \mathbf{d}_{\mathcal{C}_{n,k}}(\tilde{M}^i_L) \geq e^{-L\epsilon}C^{-1}
    \mathbf{d}_{\mathcal{C}_{n,k}}(M^i_1). \]
  Write the $M^i_\tau$ and $\tilde{M}^i_\tau$ as graphs of $u^i$ and
  $v^i$ over $\mathcal{C}$ on larger and larger
  subsets of $(0, 2] \times \mathbb{R}^{n+1}$. Note that we can write
  $v^i_\tau = u^i_\tau + \xi^i_\tau$, where for each $\tau$, up to
  choosing a subsequence, 
  $\mathbf{d}_i^{-1}\xi^i_\tau$ converges to a Jacobi field
  $\xi^\infty_\tau$ in the
  span of the functions $z_\alpha y_i$. Up to choosing a subsequence,
  $\mathbf{d}_i^{-1}u^i_\tau$ converges to a solution $u^\infty_\tau$ of the
  linearized equation on compact subsets of $(0, 2L] \times
  \mathcal{C}$, and we have that $u^\infty_L + \xi^\infty_L$ is
  orthogonal to the $z_\alpha y_i$. But this implies that
  $u^\infty_{2L} + \xi^\infty_L$ and $u^\infty_1 + \xi^\infty_L$ are
  also orthogonal to the $z_\alpha 
  y_i$, and so $\xi^\infty_1 = \xi^\infty_L = \xi^\infty_{2L}$. From the distance
  estimates above together with the non-concentration estimate, we
  find that
  \[ \Vert u^\infty_{L} + \xi^\infty_L\Vert_{L^2} \geq e^{-L\epsilon} C^{-1}
    \Vert u^\infty_1\Vert_{L^2}, \text{ and } \Vert u^\infty_{2L} +
    \xi^\infty_L\Vert_{L^2} \leq e^{-L\epsilon}  \Vert u^\infty_L +
    \xi^\infty_L\Vert_{L^2}. \]
  Since $u^\infty_1 + \xi^\infty_L$ is
  orthogonal to the $z_\alpha y_i$, we have $\Vert u^\infty_1\Vert_{L^2}
  \geq \Vert u^\infty_1+ \xi^\infty_L \Vert_{L^2}$. We are now in the
  same situation as in the previous proof. For large $L$ this is a
  contradiction since there are no homogeneous Jacobi fields with rate
  of growth $-\epsilon$. 
\end{proof}

Finally, along these lines, we need the following, related to \cite[Corollary
3.7]{SWX25}. 

\begin{prop}\label{prop:growthdominant}
  There are $\epsilon_1, L_1 > 0$, and for any $\lambda >
  0$, there exist $\delta_1(\lambda), T_1(\lambda)$ with the following
  property. Suppose that 
  for some $\delta < \delta_1$, 
  $M_\tau$ is $\delta$-graphical over $\mathcal{C}$ on
  $B_{\delta^{-1}}$ for all $\tau \geq 0$, it converges to
  $\mathcal{C}_{n,k}$, and in addition, using the
  definition $\tilde{M}_\tau = S_\tau M_\tau$ above, we have
  \[\label{eq:slowgrow50}
    \mathbf{d}_{\mathcal{C}_{n,k}}(\tilde{M}_{L_1}) \geq
    e^{-\epsilon_1 L_1} 
    \mathbf{d}_{\mathcal{C}_{n,k}}(\tilde{M}_0). \]
  Then for all $\tau\geq T_1$ the graphicality
  functions $v$ of the $\tilde{M}_\tau$ satisfy
  \[ \label{eq:projh} \Vert \Pi_\mathfrak{h} v(\cdot, \tau)\Vert_{L^2(B_{\delta^{-1}})} \geq
    (1-\lambda) \Vert v(\cdot, \tau)\Vert_{L^2(B_{\delta^{-1}})}, \]
  where $\mathfrak{h}$ is the space spanned by $y_i^2-2$, and $y_iy_j$
  for $i\not=j$. 
\end{prop}
\begin{proof}
  We first choose $\epsilon_1, L_1, \delta_1$ so that 
  Proposition~\ref{prop:3ann3} can be applied to deduce that
 \eqref{eq:slowgrow50}  is preserved for later times.
 Similarly to Corollary~\ref{cor:slowgrowth}, 
 with suitable choices of $L_1, \delta_1$, we can also assume that 
 \[ \label{eq:103} \mathbf{d}_{\mathcal{C}_{n,k}}(\tilde{M}_{\tau+L_1}) \leq e^{\epsilon L}
   \mathbf{d}_{\mathcal{C}_{n,k}}(\tilde{M}_\tau), \]
 for all $\tau$, since otherwise by Proposition~\ref{prop:3ann3} this
 growth would persist, but then $M_\tau$ could not converge to
 $\mathcal{C}_{n,k}$ at infinity. We also choose $\epsilon_1$ small
 enough so that $\sigma(\mathcal{C}_{n,k})\cap [-\epsilon_1,
 \epsilon_1] = \{0\}$. 

 For a given $\lambda > 0$, suppose that no suitable $T_1$ and
 (smaller) $\delta_1$ exist. Then we can find a sequence flows $\tilde{M}^i_\tau$
 satisfying the above conditions, but \eqref{eq:projh} failing at
 later and later times. We argue similarly to the previous
 Proposition.  We can 
 write $\tilde{M}^i_\tau$ as graphs of $v^i$ on larger and
 larger sets, and normalizing by $\mathbf{d}_i =
 \dd(\tilde{M}^i_{L_1})$, we find that along a subsequence the
 $\mathbf{d}_i^{-1} v^i$ converge to a solution $v^\infty(x,\tau)$ of the
 linearized equation on $\mathcal{C}_{n,k}$ for all $\tau \geq 0$,
 which is orthogonal to the functions $z_\alpha y_i$, and still satisfies
 \[ \label{eq:201} \Vert v^\infty(\cdot, (k+1)L_1)\Vert_{L^2} \geq e^{-\epsilon_1 L_1}
   \Vert v^\infty(\cdot, kL_1)\Vert_{L^2}, \]
 and
 \[ \label{eq:202} \Vert v^\infty(\cdot, (k+1)L_1)\Vert_{L^2} \leq e^{\epsilon_1 L_1}
   \Vert v^\infty(\cdot, kL_1)\Vert_{L^2}, \]
 for all $k > 0$. Using the convexity of $\log \Vert v^\infty(\cdot,
 \tau)\Vert_{L^2}$ in $\tau$ (see Colding-Minicozzi~\cite{CM21}), it
 follows that the translated solutions $\tau\mapsto v^\infty(\cdot,
 \tau+T)$ converge to a homogeneous solution of the linearized
 equation as $T\to\infty$. Because of \eqref{eq:201} and
 \eqref{eq:202}, and the choice of $\epsilon_1$, this limit solution
 must be homogeneous of degree zero, and since it is orthogonal to the
 $z_\alpha y_i$, it lies in $\mathfrak{h}$. It follows from this that
 for the given $\lambda > 0$ there exists some $T_1$ such that
 \[ \label{eq:projh2} \Vert \Pi_{\mathfrak{h}} v^\infty(\cdot, \tau)\Vert_{L^2} \geq
   (1-\lambda/2) \Vert v^\infty(\cdot, \tau)\Vert_{L^2}, \]
 for all $\tau \in [T_1, T_1 + L_1]$. 

 Using the non-concentration estimate, we can see that for all
 $\tau \geq 1$ the 
 distance $\mathbf{d}_{\mathcal{C}_{n,k}}(\tilde{M}_\tau)$ is well approximated by the
 $L^2$-norm on $B_{\delta^{-1}}$, if $\delta$ is small enough. Indeed,
 note that using the slow decay property, and
 Proposition~\ref{prop:SWXnonconc},
 there is a constant $C_L$ such that $\dd(\tilde{M}_{\tau-1}) \leq C_{L_1}
 \dd(\tilde{M}_{\tau})$. Then, from Proposition~\ref{prop:SWXnonconc}
 we have
 \[\label{eq:deltaintest} \delta^{-2}\int_{M_\tau\setminus B_{\delta^{-1}}}
   \overline{\mathrm{dist}}_{n,k}(x)^2\, e^{-|x|^2/4}\, d\mathcal{H}^n
   \leq C_n \dd(\tilde{M}_{\tau-1}) \leq C_nC_{L_1}
   \dd(\tilde{M}_{\tau}). \]
  It follows that given any $\kappa > 0$, if $\delta$ is small enough
  (depending on $\kappa, L$), we have
  \[  \label{eq:ddelta} \int_{M_\tau \cap B_{\delta^{-1}}}
   \overline{\mathrm{dist}}_{n,k}(x)^2\, e^{-|x|^2/4}\, d\mathcal{H}^n
   \geq (1-\kappa) \dd(\tilde{M}_{\tau}). \]
   It follows that once $\delta_1$ is small enough, \eqref{eq:projh2}
   implies \eqref{eq:projh}, on the interval $[T_1, T_1+L_1]$.  

   To extend the result to all $\tau \geq T_1$, we can use that the
   slow decay condition \eqref{eq:slowgrow50} is preserved for later
   times if $\delta_1, \epsilon_1,L_1$ are chosen suitably. 
\end{proof}

Under the slow decay condition \eqref{eq:slowgrow50} we can find good
pointwise estimates for $v(\cdot, \tau)$ in terms of the corresponding
$L^2$-norm. 
\begin{prop}\label{prop:L2ptwise}
  Given $\epsilon_1, L_1 > 0$, if $\delta_1$ is sufficiently small, we
  have the following. Suppose that for $\delta < \delta_1$, $M_\tau$ is $\delta$-graphical over
  $\mathcal{C}_{n,k}$ on $B_{\delta^{-1}}$ for all $\tau \geq 0$, it
  converges to $\mathcal{C}_{n,k}$, and we have
  \[\label{eq:slowgrow51}
    \mathbf{d}_{\mathcal{C}_{n,k}}(\tilde{M}_{L_1}) \geq
    e^{-\epsilon_1 L_1} 
    \mathbf{d}_{\mathcal{C}_{n,k}}(\tilde{M}_0). \]
  Then for all $\tau \geq 10$, and a small $\gamma > 0$, we have the
  estimate
  \[\label{eq:tvbounds10} |\tilde{v}(x, \tau)|, |\nabla \tilde{v}(x, \tau)|, |\nabla^2
  \tilde{v}(x, \tau)| \leq C_{L_1} e^{\frac{|x|^2}{40} }\Vert \tilde{v}(\cdot,
  \tau)\Vert_{L^2(B_{\delta^{-1}})} + e^{-\tau^\gamma}, \]
  on $B_{\tau^\gamma}$. 
\end{prop}
\begin{proof}
First note that by Proposition~\ref{prop:3ann3}, the condition
\eqref{eq:slowgrow51} implies that
\[ \mathbf{d}_{\mathcal{C}_{n,k}}(\tilde{M}_{(i+1)L_1}) \geq
    e^{-\epsilon_1 L_1} 
    \mathbf{d}_{\mathcal{C}_{n,k}}(\tilde{M}_{iL_1}),\]
for all $i > 0$. For any $\tau > 0$ we let $i_0$ be the smallest
integer so that $i_0 L_1 \geq \tau$. Then using
Proposition~\ref{prop:SWXnonconc} we have a bound $\dd(\tilde{M}_{i_0
  L_1}) \leq C_{L_1} \dd(\tilde{M}_\tau)$. The slow growth condition
above then ensures that $\dd(\tilde{M}_{(i_0-2)L_1}) \leq C_{L_1}
\dd(\tilde{M}_\tau)$, and using Proposition~\ref{prop:SWXnonconc}
again we get $\dd(\tilde{M}_{\tau-10}) \leq C_{L_1}
\dd(\tilde{M}_\tau)$ for a possibly larger constant $C_{L_1}$,
depending on $L_1$. 
We can then apply Corollary~\ref{cor:logSobolev} on the interval $[\tau-10,
\tau]$ to obtain pointwise 
estimates for $\nabla^i v$ on the $\frac{1}{2}\tau^\gamma$ ball.
In  the notation of the corollary, we can ensure that $\epsilon <
1/2$, and then  we have $p(9) > 10$. First we can apply the corollary
to the ``ungauged'' flow starting with $\tilde{M}_{\tau-10}$. This
leads to a bound for the graphicality function of 
$S_{\tau-10}M_\tau$, and then from that we get bounds for $S_\tau
M_\tau$.  
The conclusion is that  up to
decreasing $\gamma$ slightly, we have
\[ |\tilde{v}(x, \tau)|, |\nabla \tilde{v}(x, \tau)|, |\nabla^2
  \tilde{v}(x, \tau)| \leq C_{L_1} e^{\frac{|x|^2}{40} }\Vert \tilde{v}(\cdot,
  \tau)\Vert_{L^2(B_{\delta^{-1}})} + e^{-\tau^\gamma}, \]
on the ball $B_{\tau^\gamma}$. 
\end{proof}

We are now ready to prove Proposition~\ref{prop:L2normalform}.
\begin{proof}[Proof of Proposition~\ref{prop:L2normalform}]
  By Proposition~\ref{prop:tbetagraph11} we can assume that for a small $\beta
  > 0$, $M_\tau$ is $\tau^{-\beta}$-graphical over
  $\mathcal{C}_{n,k}$ on $B_{\tau^\beta}$, for all $\tau \geq 1$. 
Let $T_0 > 1$, and suppose that we have 
\[ \label{eq:L2normal1} \left\Vert u(\cdot, \tau) -
      \frac{\sqrt{2(n-k)}}{4\tau}\sum_{i=1}^\ell
      (y_i^2-2)\right\Vert_{L^2(B_{\tau^\beta})} < \tau^{-3/2}, \]
  for all $\tau\in [T_0, T_0 + L_0]$, with $L_0$ to be chosen. Without loss of generality
  we can assume that $T_0$ is large. Let $\lambda > 0$ be small, to be
  chosen below, and $T_1(\lambda)$ the corresponding constant from
  Proposition~\ref{prop:growthdominant}. We can assume that for the
  $L_1$ in Proposition~\ref{prop:growthdominant} we have $L_0/2 > L_1$,
  and also $L_0/2 > T_1$. Note that $L_1$ is independent of the choice of
  $\lambda$, while $T_1$ also depends on $\lambda$, so for small
  $\lambda$ we will typically have $T_1 \gg L_1$. In particular, some constants
  below depend on $L_1$, but this can be ignored since $L_1$ is fixed
  from now. 

  We claim that if $T_0, \delta^{-1}$ are sufficiently large and
  $\ell > 0$, 
  then \eqref{eq:L2normal1} implies the hypothesis
  \eqref{eq:slowgrow50}, for the flow translated in time by $T_0$. 
 For this, note that for $\gamma > 0$ the contribution from outside of the
$\tau^\gamma$-ball is less than $e^{-\tau^{2\gamma}/9}$, which is much
smaller than any inverse power of $\tau$. As a result, in
\eqref{eq:L2normal1} we can replace the $\tau^\beta$-ball with the
$\tau^\gamma$-ball for any $\gamma < \beta$.  In particular for $\tau \in
[T_0, T_0 + L_0]$ we have, say
\[ \Big|\mathbf{d}_{\mathcal{C}_{n,k}}(M_\tau) - \Vert u(\cdot,
  \tau)\Vert_{L^2(B_{\tau^\gamma})} \Big| < \tau^{-10}, \]
if $T_0$ is large.

We define the modified flow $\tilde{M}_\tau = S_\tau M_\tau$ as above, for
a small $\delta$ to be chosen, so the corresponding
graphical functions $v(x,\tau)$ are orthogonal on $B_{\delta^{-1}}$ to the space
$\mathfrak{h}_{rot}$ of functions spanned by the $y_i
z_\alpha$. Using the bound \eqref{eq:L2normal1}, and the fact that the
$y_i^2-2$ are orthogonal to elements in $\mathfrak{h}_{rot}$, we find
that for $\tau\in [T_0, T_0 + L_0]$ we have $|S_\tau - Id| <
C\tau^{-3/2}$. Using this, we find that to leading order
$\dd(\tilde{M}_\tau)$ is equal to $\alpha_{n,k} \ell \tau^{-1}$, for a
constant $\alpha_{n,k}$. In particular, for any $\epsilon > 0$ and
$L_1 < L_0$ we
can choose $\delta$ small and $T_0$ large so that
\[ \label{eq:slowdec51} \dd(\tilde{M}_{T_0 + L_1}) \geq e^{-\epsilon L_1}
  \dd(\tilde{M}_{T_0}). \]
Since $L_0/2 > T_1$, Proposition~\ref{prop:growthdominant} implies
that for all $\tau \geq T_0+L_0/2$ we have
\[ \label{eq:mostlyh} \Vert \Pi_{\mathfrak{h}} v(\cdot, \tau)\Vert_{L^2(B_{\delta^{-1}})}
  \geq (1-\lambda) \Vert v(\cdot, \tau)\Vert_{L^2(B_{\delta^{-1}})}. \]
Note that using the $\tau^{-\beta}$ graphicality of $M_\tau$ we also
find that $|S_\tau - Id| < C\tau^{-\beta}$ for all $\tau > 1$,
although  this is a much worse estimate than the $C\tau^{-3/2}$ we have on
the interval $[T_0, T_0 + L_0]$.

The case when $\ell=0$ is slightly different. In that case it may
happen that for the same $\epsilon, L_1$ as above we have
\[ \label{eq:expdecay53} \dd(\tilde{M}_{T_0 + (j+1)L_1}) \leq e^{-\epsilon L_1}
  \dd(\tilde{M}_{T_0 + jL_1}),\]
for $j <  j_0$, and
\[ \dd(\tilde{M}_{T_0 + (j_0+1)L_1}) \geq e^{-\epsilon L_1}
  \dd(\tilde{M}_{T_0 + j_0 L_1}),\]
where this growth persists for all $j > j_0$ as well. Here
possibly $j_0=\infty$, so that we have decay for all $j$. The upshot
is that we have \eqref{eq:mostlyh} for all $\tau \geq T_1 = T_0 +
j_0L_1 + L_0/2$. So in the setting $\ell=0$, we could have exponential decay
for a certain (possibly for all) time, after which the
$\mathfrak{h}$-component dominates as in \eqref{eq:mostlyh}. We will
focus on the $\ell > 0$ case below, but the same results hold when
$\ell=0$, once we are in the slow decaying regime.

We can use the slow decay property together with non-concentration to
relate the $L^2$ norms of $v$ on different balls to each other. First note that arguing
similarly to the proof of Proposition~\ref{prop:growthdominant}, there
is a constant $C_{L_1}$ such that for $\tau\geq T_0+L_1$ we have
$\dd(\tilde{M}_{\tau-1}) \leq C_{L_1} \dd(\tilde{M}_\tau)$.  Then
similarly to \eqref{eq:deltaintest} and \eqref{eq:ddelta} we have
the following: for any $\lambda > 0$, if $\delta$ is sufficiently
small (depending on $\lambda, L_1$),
then for all $\tau$ such that $\tau^\gamma > \delta^{-1}$,  we have
\[ \label{eq:normcompare} \int_{B_{\delta^{-1}}} |v(\cdot, \tau)|^2\, e^{-|x|^2/4}\,
    d\mathcal{H}^n \geq (1-\lambda) \dd(\tilde{M}_\tau) \geq
    (1-\lambda) \int_{B_{\tau^\gamma}} |v(\cdot, \tau)|^2\,
    e^{-|x|^2/4}\, d\mathcal{H}, \]
so it does not matter whether we take the $L^2$-norm on $B_{\delta^{-1}}$
or on $B_{\tau^\gamma}$.

Let us now define the cut off graphical functions $\tilde{v}(x,\tau) =
\chi(|x| \tau^{-\gamma}) v(x, \tau)$. Similarly to
\eqref{eq:utildeeq}, $\tilde{v}$ satisfies the equation
\[ \label{eq:vtildeeq} \partial_\tau \tilde{v} = \Delta \tilde{v} - \frac{1}{2} x\cdot
  \nabla\tilde{v} + \tilde{v} + a_1\cdot \nabla\tilde{v} + a_2
  \tilde{v} + b_1(\xi_\tau)\cdot \nabla\tilde{v} + b_2(\xi_\tau)
  \tilde{v} + \chi\xi_\tau + \mathcal{E}, \]
where $\xi_\tau \in \mathfrak{h}_{rot}$ corresponds to the
infinitesimal rotation $S_\tau^{-1}\dot{S_\tau}$, and the $b_1,b_2$
terms account for the fact that the function $\xi_\tau$ may not
exactly equal the normal component of the corresponding vector
field along the graph of $\tilde{v}$. We will also write this as
\[ \label{eq:vtildeeq2} \partial_\tau \tilde{v} = \mathcal{L}\tilde{v} + Q(\tilde v,
  \nabla\tilde v, \nabla^2\tilde v) + \mathcal{R}_{\xi_\tau}(\tilde v,
  \nabla\tilde v) + \chi\xi_\tau + \mathcal{E}, \]
where $Q$ is the nonlinear part of the mean curvature of a graph as
before, and $\mathcal{R}_{\xi_\tau}$ is at least linear in $\tilde v,
\nabla\tilde v$ with coefficients
bounded by $|\xi_\tau|$.
Note that by adjusting the $\xi_\tau \in \mathfrak{h}_{rot}$
slightly, we can ensure that 
\[ \mathcal{R}_{\xi_\tau}(\tilde v, \nabla\tilde v) \perp \mathfrak{h}_{rot}, \]
for all $\tau$, and we will assume we chose the $\xi_\tau$ in this way. 
Then, taking the inner product with $\xi_\tau$, and noting that the
image of $\mathcal{L}$ is orthogonal to $\mathfrak{h}_{rot}$, we find that
\[ \label{eq:xibound} |\xi_\tau| \leq \Vert Q(\tilde v,
  \nabla\tilde v, \nabla^2\tilde v)\Vert_{L^2} + \Vert
  \mathcal{E}\Vert_{L^2} + |\xi_\tau|^{-1} \langle \xi_\tau,
  \partial_\tau \tilde{v}\rangle. \]
We have $\Vert \mathcal{E}\Vert_{L^2} \leq e^{-\tau^\gamma}$. By
construction we also have that the $L^2$-product of
$\partial_\tau\tilde{v}$ with $\xi_\tau$ on $B_{\delta^{-1}}$
vanishes, and using \eqref{eq:vtildeeq2} together with
\eqref{eq:tvbounds10} to bound the terms on the right hand side of
\eqref{eq:vtildeeq2} outside of the $\delta^{-1}$-ball,  
we can choose $\delta$ small, depending on $\lambda$, so that 
\[ \left|\int_{\mathcal{C}\setminus B_{\delta^{-1}}} \xi_\tau
    \partial_\tau \tilde{v} \, d\mu \right| \leq \lambda |\xi_\tau| \Vert
  \tilde{v}(\cdot, \tau)\Vert_{L^2} +
  e^{-\tau^\gamma} |\xi_\tau|. \]
Since each term in $Q(\tilde{v}, \nabla \tilde{v}, \nabla^2
\tilde{v})$ is at least quadratic in the entries,
Proposition~\ref{prop:L2ptwise}  implies that
\[ \int_{\mathcal{C}} Q(\tilde{v}, \nabla \tilde{v}, \nabla^2
\tilde{v})^2\, e^{-|x|^2/4}\, d\mathcal{H}^n \leq C_{L_1} \Vert
\tilde{v}(\cdot, \tau)\Vert_{L^2(B_{\delta^{-1}})}^4 +
e^{-\tau^\gamma}. \]
Using these estimates in \eqref{eq:xibound} we find that
\[\label{eq:xitaubound} |\xi_\tau| \leq C \Vert \tilde{v}(\cdot,
  \tau)\Vert_{L^2(B_{\tau^\delta})}^2 + \lambda \Vert \tilde{v}(\cdot,
  \tau)\Vert_{L^2(B_{\tau^\gamma})}+ e^{-\tau^\gamma}, \]
further decreasing $\gamma$ slightly. 
Note also that by \eqref{eq:normcompare} the $L^2$-norm here is
comparable to the $L^2$-norm on 
$B_{\delta^{-1}}$, or also to $\dd(\tilde{M}_\tau)$. 

Now we consider the evolution of the component of $\tilde{v}$ lying in
$\mathfrak{h}$. Let us therefore write
\[ \tilde{v} = \tilde{v}_{\mathfrak{h}} + \tilde{v}_{\perp}, \]
where $\tilde{v}_\perp$ is orthogonal to $\mathfrak{h}$. 
Further, following \cite[Section 2]{FL93} (see also
\cite[Section 3]{Vel93} or \cite[Lemma 5.6]{SX22}) we write  
\[ \label{eq:Av} \tilde{v}_{\mathfrak{h}} = \frac{1}{\sqrt{2}}
  \sum_{i=1}^k A_{ii}(\tau) h_{ii}(y) + \frac{1}{2}\sum_{i\not=j}
  A_{ij}(\tau) h_{ij}(y), \]
for a symmetric matrix $A(\tau)$. Here $h_{ii}(y) = c_1 (y_i^2-2)$ and
$h_{ij} = c_2 y_iy_j$ for suitable constants $c_1,c_2$ so that these
functions have unit $L^2$ norm on $\mathcal{C}_{n,k}$. 
Note that $\tilde{v}_{\perp}$ also satisfies pointwise bounds like in
\eqref{eq:tvbounds10} (with the norm of either $\tilde{v}$ or
$\tilde{v}_{\perp}$ on the
right), since the elements in the finite dimensional space
$\mathfrak{h}$ are explicit functions that satisfy such estimates.

Denote by $f\in\mathfrak{h}$ an element with unit norm, and consider
the evolution of the corresponding component of $\tilde{v}$. We will write $d\mu
= e^{-|x|^2 / 4} d\mathcal{H}^n$.  We have
\[ \label{eq:hevolve}\partial_\tau \int_{\mathcal{C}} f \tilde{v}\, d\mu &= \int_{\mathcal{C}} f
  (\mathcal{L}\tilde{v} + Q + \mathcal{R}_{\xi_\tau} + \xi_{\tau} +
  \mathcal{E})\, d\mu \\
  &= \int_{\mathcal{C}} f(Q + \mathcal{R}_{\xi_\tau} + \mathcal{E})\,
  d\mu, 
\]
where we used that $f$ is orthogonal to the image of $\mathcal{L}$,
and also to $\xi_\tau$. Let us further decompose $Q = Q_2 + Q_{>2}$
into the quadratic and higher order terms. We have (see
\cite[Proposition A.1]{SX22})
that
\[ \label{eq:Q2formula}
  Q_2 &= -\frac{1}{2\sqrt{2(n-k)}}\left(\tilde{v}^2 +
    2|\nabla_z\tilde{v}|^2 +4\tilde{v} \Delta_z
  \tilde{v}\right) \\
&= -\frac{1}{2\sqrt{2(n-k)}}\left(\tilde{v}^2 + 2|\nabla_z \tilde{v}_{\perp}|^2+ 4\tilde{v} \Delta_z \tilde{v}_{\perp}\right), \]
where the gradient and Laplacian are only taken in the sphere
direction of the cylinder, and these vanish on the
$\mathfrak{h}$-component. In order to deal with the norm squared
component, we integrate by parts, and note that the measure $d\mu$,
and $f$, are
invariant in the $z$-directions. So we have
\[ \int_{\mathcal{C}} f |\nabla_z \tilde{v}_{\perp}|^2\, d\mu =
  -\int_{\mathcal{C}} f \tilde{v}_{\perp} \Delta_z \tilde{v}_{\perp}\,
  d\mu. \]
We can also integrate the Laplacian term by parts twice to get that
$\tilde{v}\Delta_z \tilde{v}_{\perp}$ can be replaced by
$\tilde{v}_{\perp}\Delta_z \tilde{v}_{\perp}$. Using the pointwise
bounds \eqref{eq:tvbounds10} and integrating, we find that
\[ \label{eq:fv40} \left| \int_{\mathcal{C}} f \tilde{v}_{\perp} \Delta_z \tilde{v}_{\perp}\,
  d\mu\right| &\leq C \Vert \tilde{v}_\perp\Vert_{L^2}
(\Vert \tilde{v}\Vert_{L^2} + e^{-\tau^\gamma}) \\
&\leq C(\lambda \Vert \tilde{v}\Vert_{L^2}^2 + e^{-\tau^\gamma}),  \]
for $\tau\geq T_0 + L_0/2$, 
where we can take $\lambda > 0$ arbitrarily small by taking $L_0$
large. This deals with the
terms in $Q_2$ other than $\tilde{v}^2$. We also have
\[ \int_{\mathcal{C}} f \tilde{v}^2 \, d\mu = \int_{\mathcal{C}} f
  (\tilde{v}_{\mathfrak{h}}^2 +
  2\tilde{v}_{\mathfrak{h}}\tilde{v}_\perp + \tilde{v}_\perp^2)\,
  d\mu, \]
so
\[ \label{eq:fv41}  \left| \int_{\mathcal{C}} f \tilde{v}^2\, d\mu - \int_{\mathcal{C}}
    f \tilde{v}_{\mathfrak{h}}^2\, d\mu \right| \leq C\lambda \Vert
  \tilde{v}\Vert_{L^2}^2 e^{-\tau^\gamma}. \]

The terms in $Q_{>2}$ contain at least 3 factors of $\tilde{v}$ or its
derivatives, which we estimate using the  pointwise bounds \eqref{eq:tvbounds10}, and that $\Vert
\tilde{v}\Vert_{C^4}\leq \tau^{-\gamma}$ on
$B_{\tau^\gamma}$. Similarly the $\mathcal{R}_{\xi_\tau}$ term can be
estimated using \eqref{eq:xitaubound} and \eqref{eq:tvbounds10}. We get
\[ \left|\int_{\mathcal{C}} f (Q_{>2} + \mathcal{R}_{\xi_\tau}) \,d\mu \right| \leq C
  \tau^{-\gamma} \Vert \tilde{v}\Vert_{L^2}^2 +C\lambda
\Vert \tilde{v}\Vert_{L^2}^2+  e^{-\tau^\gamma}. \]
Finally, the cutoff term $\mathcal{E}$ is bounded, and supported where $|x|
\sim \tau^\gamma$. Combining everything, we find that for any $\lambda
> 0$ we can arrange that  for all $\tau\geq T_0 + L_0/2$, if $L_0$ is
chosen large enough, we have
\[ \left| \partial_\tau \int_{\mathcal{C}} f \tilde{v}_{\mathfrak{h}}\, d\mu +
    \frac{1}{2\sqrt{2(n-k)}} \int_{\mathcal{C}}f \tilde{v}_{\mathfrak{h}}^2\, d\mu\right| \leq
  C\lambda \Vert \tilde{v}_{\mathfrak{h}} \Vert^2_{L^2}+ e^{-\tau^\gamma}. \]
Note that we replaced $\tilde{v}$ by $\tilde{v}_{\mathfrak{h}}$ on the
left since $f\in \mathfrak{h}$, and on the right by the fact that the
$L^2$-norm of $\tilde{v}_{\mathfrak{h}}$ dominates the $L^2$-norm of
$\tilde{v}$.

This is now a
finite dimensional problem. Following the calculation in
\cite[Section 3]{Vel93} or \cite[Lemma 5.6]{SX22}, 
we can write it as the following ODE for the matrix $A$:
\[ \label{eq:AODE} \left\Vert \frac{d}{d\tau} A(\tau)  + \nu_{n,k} A(\tau)^2\right\Vert \leq
  C\lambda \Vert A(\tau)\Vert^2 + e^{-\tau^\gamma}, \]
where $\nu_{n,k} > 0$ is a fixed constant. 
We can arrange that $\lambda$ is small enough, so that $C\lambda <
\nu/4$. Moreover, note that on the interval $[T_0, T_0 + L_0]$ we have,
by \eqref{eq:L2normal1}, that 
\[ \label{eq:Aform} A_{ii} &= \frac{\nu_{n,k}^{-1}}{\tau} + O(\tau^{-3/2}), \text{ for } i
  \leq \ell, \\
  A_{ii} &= O(\tau^{-3/2}), \text{ for } i > \ell, \\
  A_{ij} &= O(\tau^{-3/2}), \text{ for } i\not=j. \]
If $T_0$ is chosen sufficiently large, then we also have
$e^{-\tau^\gamma} < \tau^{-3/2}$. Using this, together with the ODE
for $A$, we can use \cite[Lemma 3.1]{FL93} to
find that if $\ell > 0$, then  for $\tau \geq T_0$ we have
\[ C^{-1} \tau^{-1} \leq \Vert A(\tau)\Vert \leq C \tau^{-1} \]
for a uniform constant $C$. 
In particular this means that $\Vert \tilde{v}_{\mathfrak{h}}\Vert
\sim \Vert \tilde{v}\Vert_{L^2} \sim
\tau^{-1}$ for all $\tau \geq T_0$. When $\ell=0$, then by \eqref{eq:expdecay53} we have
exponential decay of $A(\tau)$ until eventually \eqref{eq:mostlyh}
begins to hold, and then the argument above still leads to $\Vert
A(\tau)\Vert \leq C\tau^{-1}$ for all $\tau \geq T_0$. 

To get more precise
asymptotics we need to return to the original flow $M_\tau$, not
modified by rotations. From \eqref{eq:xitaubound} we now have
$|\xi_\tau| \leq C\tau^{-2}$, which implies that $|S_\tau - Id| \leq
C\tau^{-1}$. Writing $M_\tau$ as the graph of $u(x,\tau)$ on
$B_{\tau^\gamma}$, we now have the improved $L^2$-estimate  $\Vert u\Vert_{L^2} \leq
C\tau^{-1}$. From Proposition~\ref{prop:tbetagraph10} we find that for small
$\gamma < \gamma'$ we have
$\Vert \tilde{u}\Vert_{C^4(B_{\tau^\gamma})} \leq \tau^{2\gamma'-1}$,
for $\tau \geq T_0+1$, if $T_0$ is large enough.   

Let $\tilde{u} = \chi u$ as before, cutting off
on $B_{\tau^\gamma}$. This satisfies the equation
\[ \partial_\tau \tilde{u} = \Delta \tilde{u} - \frac{1}{2}x\cdot \nabla \tilde{u} + \tilde{u}+
  a_1\cdot \nabla \tilde{u} + a_2 \tilde{u} + \mathcal{E}, \]
as before. 
Consider the decomposition
\[ \label{eq:tudecomp}
  \tilde{u} = \tilde{u}_0 + \tilde{u}_{>0} + \tilde{u}_{<0} =  \tilde{u}_{\mathfrak{h}} + \tilde{u}_{rot} +
  \tilde{u}_{>0} + \tilde{u}_{<0}, \]
according to the eigenspaces of $\mathcal{L}$ corresponding to
$\mathfrak{h}$, rotations, positive, and negative eigenvalues
respectively. 
We first show that $\tilde{u}_{<0}, \tilde{u}_{>0}$ are of lower
order. We have
\[ \partial_\tau \int_{\mathcal{C}} \tilde{u}_{>0}^2\,
  d\mu &= \int_{\mathcal{C}} 2 \tilde{u}_{> 0} (\mathcal{L}\tilde{u}_{>0}
  + Q(\tilde{u}, \nabla\tilde{u}, \nabla^2\tilde{u}) + \mathcal{E})\,
  d\mu \\
&\geq  \frac{1}{2} \int_{\mathcal{C}} \tilde{u}_{> 0}^2 \,d\mu -
C\tau^{8\gamma'-4} - e^{-\tau^\gamma},\]
using that the smallest positive eigenvalue of $\mathcal{L}$ is
$1/2$, and the term $\tilde{u}_{>0}Q$ is bounded by
\[ C|\tilde{u}_{>0}\tau^{4\gamma'-2}| \leq \frac{1}{2} \tilde{u}_{>0}^2
  + C^2 \tau^{8\gamma'-4}. \]
Letting $p(t) = \Vert \tilde{u}_{>0}\Vert_{L^2}^2$, we get
\[ p'(\tau) \geq \frac{1}{2}p(\tau) - C \tau^{8\gamma' - 4}, \]
for large $\tau$. Integrating this inequality from $\tau$ to $\infty$,
and using that $p(\tau)\to 0$, we find that
\[ p(\tau) &\leq C\int_\tau^\infty e^{\frac{1}{2}(\tau - s)}
  s^{8\gamma'-4}\, ds \leq C \tau^{8\gamma'-4}. 
\]
It follows that for any $\gamma' > 0$ we can arrange that 
$\Vert \tilde{u}_{>0} \Vert_{L^2} \leq C\tau^{-2 +
  4\gamma'}$ for large $\tau$.

Similarly we have
\[ \partial_\tau \int_{\mathcal{C}} \tilde{u}_{<0}^2\,
  d\mu &\leq - c_1 \int_{\mathcal{C}} \tilde{u}_{<  0}^2 \,d\mu +
C\tau^{8\gamma'-4} - e^{-\tau^\gamma},\]
where $c_1$ is the smallest negative eigenvalue of $\mathcal{L}$, so $q(\tau) =
\Vert \tilde{u}_{<0}\Vert_{L^2}$ satisfies
\[ q'(\tau) \leq -c_1 q(\tau) + C \tau^{8\gamma' - 4}. \]
Now we integrate from $T_0$ to $\tau$, and use that by
\eqref{eq:L2normal1} we have $\Vert \tilde{u}_{<0}(T_0)\Vert_{L^2} <
2T_0^{-3/2}$. We find that
\[ q(\tau) \leq 2T_0^{-3/2} e^{c_1(\tau-T_0)} + \int_{T_0}^\tau
  e^{c_1(s-t)} s^{-4+8\gamma'}\, ds. \]
If $L_0$ is chosen sufficiently large, depending on $T_0$, then this
implies that $\Vert \tilde{u}_{<0}\Vert_{L^2} \leq
C\tau^{-2+4\gamma'}$ for all $\tau \geq T_0 + L_0/2$. 

Now we consider the component $\tilde{u}_{rot}$ in
$\mathfrak{h}_{rot}$. Letting $h\in \mathfrak{h}_{rot}$ be an element
with norm 1, we have
\[ \label{eq:hcomp2} \partial_\tau \int_{\mathcal{C}} h \tilde{u}\, d\mu =
  \int_{\mathcal{C}} h (Q_2 + Q_{>2} + \mathcal{E})\, d\mu, \]
decomposing the nonlinear terms into quadratic and higher degree
pieces. The higher order terms and the cutoff error can be estimated
as before, using the bound
$\Vert\tilde{u}\Vert_{C^4(B_{r^\gamma})}\leq \tau^{2\gamma'-1}$:
\[ \label{eq:hcomp3} \left| \int_{\mathcal{C}} h (Q_{>2} + \mathcal{E})\, d\mu\right|
  \leq C\tau^{6\gamma' - 3}. \]
Let us substitute the decomposition \eqref{eq:tudecomp} into the formula
\eqref{eq:Q2formula} for $Q_2$. The terms arising purely from the
$\tilde{u}_{rot}$ component are orthogonal to $h$ -- this follows from
the fact that the Jacobi fields in $\mathfrak{h}_{rot}$ are
integrable. One can check explicitly that the other terms arising from
$\tilde{u}_{rot}$ and $\tilde{u}_{\mathfrak{h}}$ are also
orthogonal to $h$. The remaining terms in \eqref{eq:hcomp2} due to
$Q_2$ can all be bounded using $\Vert
\tilde{u}_{>0}\Vert_{L^2}$ and $\Vert \tilde{u}_{<0}\Vert_{L^2}$, which for
$\tau \geq T_0 + L_0/2$ are 
of order $\tau^{-2+4\gamma'}$, multiplied
by $\Vert \tilde{u}\Vert_{C^4(B_{\tau^\gamma})} \leq
\tau^{2\gamma'-1}$. Combined with \eqref{eq:hcomp3} we this implies that
\[ \left| \partial_\tau \int_{\mathcal{C}} h \tilde{u}\, d\mu \right|
  \leq C\tau^{6\gamma' - 3}. \]
Integrating this, using that as $\tau\to\infty$
the $h$ component decays to zero, we get $\Vert \tilde{u}_{rot}
\Vert_{L^2} \leq C \tau^{6\gamma'- 2}$, for $\tau \geq T_0+L_0/2$. 

Finally we return to the component $\tilde{u}_{\mathfrak{h}}$, and
define the matrix $A(\tau)$ as in \eqref{eq:Av}, but using
$\tilde{u}_{\mathfrak{h}}$. For $f\in \mathfrak{h}$ the inner product
of $f\tilde{u}$ evolves like \eqref{eq:hevolve}, except we no longer
have the $\mathcal{R}_{\xi_\tau}, \xi_\tau$ terms:
\[ \partial_\tau \int_{\mathcal{C}} f \tilde{u}\, d\mu =
  \int_{\mathcal{C}} f(Q_2 + Q_{>2} + \mathcal{E})\, d\mu. \]
We write $\tilde{u} = \tilde{u}_{\mathfrak{h}} + \tilde{u}_{\perp}$,
and use that $\Vert \tilde{u}_\perp\Vert_{L^2}\leq
C\tau^{6\gamma'-2}$. 
The terms arising from $Q_{>2}$ and $\mathcal{E}$ can be bounded by
$\tau^{6\gamma'-3}$ as in \eqref{eq:hcomp3}. We can estimate the $Q_2$
term as before, but now in \eqref{eq:fv40} and \eqref{eq:fv41} the
errors on the right hand side become
\[ C( \tau^{6\gamma'-2} \Vert \tilde{u}\Vert_{L^2} + e^{-\tau^\gamma})
  \leq C \tau^{6\gamma'-3}. \]
As a consequence, instead of \eqref{eq:AODE},
the equation satisfied by the matrix $A$ is now
\[ \label{eq:Aevolve2} \left\Vert \frac{d}{d\tau} A(\tau) + \nu_{n,k}
    A(\tau)^2 \right\Vert \leq 
  \tau^{6\gamma'-3}, \]
once $\tau\geq T_0+L_0/2$ is large enough. Note that we can replace
the constant $C$ by 1, by adjusting $\gamma'$ and taking $\tau$ large
enough. 

We can now argue similarly to \cite[Proposition 5.1]{FL93}
to analyze the ODE \eqref{eq:Aevolve2}. In their case the error is
$O(\tau^{-3})$, while in our case it is only $O(\tau^{6\gamma'-3})$
for any $\gamma' > 0$, and also we want better control of the error
terms, but the arguments are similar.

First we can show that for any $\kappa >
0$, once $T_0, L_0$ are large enough the eigenvalues
$\lambda_i(\tau)$ of $A(\tau)$ in decreasing order satisfy
\[ \label{eq:lambda30} \lambda_i(\tau) &= \frac{\nu_{n,k}^{-1}}{\tau} +
  O(\tau^{\kappa-2}), \text{ for } i\leq \ell, \\
  \lim_{\tau\to \infty} \tau \lambda_i(\tau) &> 0, \text{ for } \ell <
  i \leq \ell', \\
  \lambda_i(\tau) &= O(\tau^{\kappa-2}), \text{ for } i > \ell', \]
for some $\ell' \geq \ell$, and
all $\tau \geq T_0 + L_0$. To see this, note that for $\tau \geq
T_0 + L_0/2$  the eigenvalues satisfy the ODEs
\[\label{eq:lambdaode} |\lambda'(\tau)  +\nu_{n,k} \lambda(\tau)^2|
  \leq  \tau^{6\gamma'-3}, \]
where $\gamma' > 0$ can be as close to zero as we like (if $T_0, L_0$
are large).  Instead of arguing like \cite[Lemma 5.1]{FL93}, we
use barriers. The functions
\[ g_{\pm}(\tau) =
  \nu_{n,k}^{-1}\tau^{-1} \pm (\tau - L_0/2)^{7\gamma'-2}\]
are super- and
subsolutions of this equation once $\tau \geq T_0 + L_0/2$ and $T_0,
L_0$ are sufficiently large. At the same time, by \eqref{eq:L2normal1}
we have
\[ \left|\lambda_i(T_0 + L_0/2) - \frac{\nu_{n,k}^{-1}}{T_0 + L_0/2}
  \right| < (T_0 + L_0/2)^{-3/2}, \text{ for } i\leq\ell, \]
so we have
\[ \lambda_i(T_0 + L_0/2) \leq g_+(T_0+ L_0/2), \]
if $L_0$ is sufficiently large, depending on $T_0$ and $\gamma'$. By the
supersolution property of $g_+$, this inequality then holds for all
$\tau\geq T_0 + L_0/2$ as well. If $\tau \geq T_0 + L_0$, then 
\[ g_+(\tau) \leq \nu_{n,k}^{-1} \tau^{-1} + \tau^{8\gamma'-2}, \]
if $T_0+L_0$ is sufficiently large, and in particular the upper bound
in \eqref{eq:lambda30} holds for $i\leq \ell$ if $8\gamma' <
\kappa$. The lower bound is similar, using $g_-$.

To estimate the $\lambda_i$ for $i > \ell$ we can argue as
follows. First note that if $\lambda(\tau)$ satisfies
\eqref{eq:lambdaode}, then 
\[ (\tau\lambda)' \leq \tau^{-1}( \tau\lambda - 
  \tau^2\nu_{n,k}\lambda^2 + \tau^{6\gamma'-1}). \]
It follows that if $\tau_0\lambda(\tau_0) \leq -\tau_0^{6\gamma'-1}$
for some $\tau_0$, then
$|\tau\lambda|$ is increasing, and this inequality persists for all
$\tau \geq \tau_0$. Eventually this leads to $\lambda(\tau) \to
-\infty$ as $\tau\to\infty$, which contradicts that our flow
$M_\tau$ converges to $\mathcal{C}_{n,k}$. It follows that we have
$\lambda(\tau) \geq -\tau^{6\gamma'-2}$ for all $\tau \geq T_0 +
L_0/2$.

To bound the $\lambda_i$ from above for $i > \ell$, consider the
differential inequality
\[ \lambda'(\tau) \geq - \nu \lambda(\tau)^2 - \tau^{6\gamma'-3}. \]
Then
\[ (\tau\lambda)' \geq \tau^{-1}(\tau\lambda - \nu (\tau\lambda)^2 -
  \tau^{6\gamma'-1}). \]
If at some time $\tau_0$ we have the two conditions $\tau_0\lambda(\tau_0) >
2\tau_0^{6\gamma'-1}$ and $2\nu \tau_0\lambda(\tau_0) < 1$, then
$\tau\lambda(\tau)$ is increasing, and it keeps increasing until $2\nu
\tau \lambda(\tau) \geq 1$, and this inequality then persists for all
larger $\tau$. If this happens for $\lambda_i$ with some $i > \ell$,
then according to the normal form \eqref{eq:SXnormal} the singularity
of $M_\tau$ at infinity is $\ell'$-degenerate for some $\ell' >
\ell$. This is the alternative (i) in
Proposition~\ref{prop:L2normalform}.

The other possibility is that $\tau\lambda(\tau) \leq
2\tau^{6\gamma'-1}$ for all $\tau\geq T_0+L_0/2$, which applied to the
$\lambda_i$ for $i > \ell$ gives the
required estimate in \eqref{eq:lambda30}. Suppose from now that we are in this
setting. 

Once the bounds for the eigenvalues are established, one can consider
the evolution of the
different principal minors of $A$ as in \cite[Proposition
5.1]{FL93}, to find that in fact
\[ QA(\tau)Q^{-1} = \frac{\nu_{n,k}^{-1}}{\tau} \begin{pmatrix}
    I_{\ell} & 0 \\ 0 & 0_{k-\ell}\end{pmatrix}  +  O(\tau^{\kappa-2}), \]
for an orthogonal matrix $Q$, where $0_{k-\ell}$ is the $k-\ell$
dimensional zero matrix. Since for $\tau\in [T_0,
T_0+L_0]$ the matrix $A$ has the form \eqref{eq:Aform}, by choosing
$L_0$ large we can ensure that $Q$ is close to the identity when $\ell
> 0$. When $\ell=0$, we can take $Q=I$.  
\end{proof}

\section{Geometric normal form}\label{sec:geomnormal}
In the previous section we identified the second term in the normal
form for a rescaled flow $M_\tau$ with a cylindrical singularity at
infinity, in an $L^2$ sense. Recall that for
any $\beta > 0$ the
contribution of the region outside of the $\tau^\beta$-ball to the
$L^2$-norm of the graphicality function $u$ in
Proposition~\ref{prop:L2normalform} is negligible compared to
$\tau^{-1}$, so the $L^2$ normal form does not immediately say anything
about the geometry of $M_\tau$ outside of such balls. 
Our goal in this section is to obtain good graphicality of $M_\tau$
over suitable models over balls  of size
$B_{K_0\sqrt{\tau}}$ for large $K_0 > 0$, where in practice we can take $K_0
= 20$. The corresponding asymptotic result was obtained in the
$C^1$-normal form result of Sun-Xue~\cite[Theorem 1.3]{SX22}, but we
need a more quantitative version. We will also prove the exponential
decay of the derivatives $\partial_{y_q} u$ in the degenerate
$y_q$-directions in Proposition~\ref{prop:uqdecay}, which will play an
important role later. 

As explained in \cite{AV97}, \cite{SX22},
near the boundary $|x| = K_0\tau^{1/2}$ of the ball $B_{K_0\tau^{1/2}}$ the flow can
no longer be expected to have good graphicality over
$\mathcal{C}_{n,k}$, and we must use a corrected model. For the moment
consider the setting of a nondegenerate
$\mathcal{C}_{n,1}$-singularity, i.e. a nondegenerate neck-pinch of
hypersurfaces modeled on the cylinder $\mathbb{R}\times S^{n-1}$.  It was
observed by Angenent-Vel\'azquez~\cite{AV97} (see also
Sun-Xue~\cite[Theorem 1.3]{SX22}) that to leading order,
near $|y| = K_0\sqrt{\tau}$ as $\tau\to\infty$, the correct asymptotic
behavior is to consider the graph of
\[ \sqrt{2(n-1)}\left(\sqrt{1 + \frac{y^2}{2\tau}} - 1\right) \]
over $\mathcal{C}_{n,1}$. In \cite{AV97} it is briefly discussed how
one can construct a rotationally symmetric flow with a nondegenerate
neck-pinch by further refining this approximate solution, in analogy
to their construction of degenerate neck-pinches. As far as we know
this construction has not been carried out in detail,
however. Moreover we also need analogous ``model flows'' for
nondegenerate $\mathcal{C}_{n,k}$-singularities with $k > 1$.

In this paper we will use an indirect inductive argument to show that
flows developing non-degenerate $\mathcal{C}_{n,k}$-singularities
exist for all $n,k$, see Proposition~\ref{prop:nondegexist}. For each
$n,k$ we fix one such flow as our model, and use that asymptotically
as $\tau\to \infty$ the geometry of such non-degenerate singularities
is described by the $C^1$-normal form of \cite{SX22}. Taking a product
with a Euclidean factor we obtain suitable $\ell$-degenerate models as
well. We will then show that arbitrary $\ell$-degenerate flows have
good graphical behavior over these models.

Eventually we will work by induction on the dimension
$n$, and for now we assume that for all $n' < n$ and $0 < k' < n'$ we
have fixed a mean curvature flow $\mathcal{N}^{n', k'}_t$
which develops a nondegenerate $\mathcal{C}_{n', k'}$-singularity at
$(0,0)$. We write $\mathcal{M}^{n', k'}_\tau$ for the corresponding
rescaled flow. Let us recall the $H^1$ and $C^1$ normal form theorems
due to Sun-Xue~\cite[Theorem 1.3]{SX22}. These say that for any $K_0 > 0$
and $\gamma >0$ there exists a $T_0(K_0, \gamma)$, such that
for $\tau \geq T_0$ we can write $\mathcal{N}^{n', k'}_\tau$ as the
graph of $U(x, \tau)$ over $\mathcal{C}_{n',k'}$ on the ball
$B_{K_0\sqrt{\tau}}$, satisfying
\[ \label{eq:SXH1} \left\Vert U(x, \tau) - \frac{\sqrt{2(n'-k')}}{4\tau}
    \sum_{i=1}^{k'} (y_i^2-2) \right\Vert_{H^1} \leq \tau^{\gamma-2}, \]
and 
\[ \label{eq:SXC1} \left\Vert U(x,\tau) - \sqrt{2(n'-k')} \left(\sqrt{1 +
        \frac{1}{2\tau} \sum_{i=1}^{k'}
        (y_i^2-2)}-1\right)\right\Vert_{C^1} \leq \tau^{\gamma-1}. \]
Note that $|U(x,\tau)|$ is not small, since it is of order $K_0$
near the boundary of $B_{K_0\sqrt{\tau}}$, but $|\nabla_x U(x,\tau)|
\lesssim K_0 \tau^{-1/2}$, so for large $\tau$ the model is very close
to a cylinder of radius $\sqrt{2(n'-k')} + U((y_0,z_0), \tau)$ on the region
where $|y - y_0| \ll \tau^{1/2}$.  
The following statements result from this description. Let us define the function
\[ F(\xi) = \sqrt{1 + \frac{\xi^2}{2}}. \]
We emphasize that the following result refers to a specific model flow
$\mathcal{M}^{n,k,\ell}_\tau$, however in
Proposition~\ref{prop:ldegendecay} we will show that any rescaled flow
with an $\ell$-degenerate cylindrical singularity will have similar
behavior, if we have good
control of the $L^2$-norm of the graphicality function. 
\begin{prop}\label{prop:models}
  Let $K_0, \kappa_0 > 0$. There is a $T_0(K_0, \kappa_0) > 0$ such that for $\tau > T_0$
  we have the following properties of $\mathcal{M}^{n', k'}_\tau$ on
  the ball $B_{K_0\sqrt{\tau}}$:
  \begin{enumerate}
  \item For any $x_0 = (y_0, z_0)$, on the ball $B_{\tau^{1/3}}(x_0)$,
    $\mathcal{M}^{n',k'}_\tau$ is $\kappa_0$-graphical over the
    cylinder $S^{n'-k'}(r)\times \mathbb{R}^{k'}$ with $r =
    \sqrt{2(n'-k')} F(|y_0| \tau^{-1/2})$. In addition the second
    fundamental form satisfies $|A|^2\leq \frac{1}{2}  +\kappa_0$.
    Let $\mathbf{n}$ be the unit normal vector of
    $\mathcal{M}^{n', k'}_\tau$, and $\mathbf{n}_0$ be the unit
    radial vector in the $\mathbb{R}^{n-k+1}$-factor, i.e. the normal
    vector to $\mathcal{C}_{n,k}$, extended as a constant vector field
    radially. Then $\Vert\mathbf{n} -
    \mathbf{n}_0\Vert_{C^2} \leq C(K_0) \tau^{-1/2}$. 
  \item For small $\gamma > 0$, on the ball $B_{\tau^\gamma}$,
    $\mathcal{M}^{n',k'}_\tau$ is the graph of $w(x,\tau)$ over
    $\mathcal{C}_{n',k'}$, where
    \[ \label{eq:Btaugammaw} w(x,\tau) = \frac{\sqrt{2(n'-k')}}{4\tau}
      \sum_{i=1}^{k'} (y_i^2-2) + O(\tau^{-19/10}), \]
    with the error estimate being in $C^2$. 
    \item We identify functions on $\mathcal{M}^{n',k'}_\tau\cap
      B_{\tau^\gamma}$ with functions on $\mathcal{C}_{n',k'}$
      through the graphicality. Then we have the following comparison
      between the linearized operators on
      $\mathcal{M}^{n',k'}_\tau$ and $\mathcal{C}_{n',k'}$:
   \[ \Delta -
      \frac{1}{2}x\cdot \nabla + |A|^2 =
    \left(\mathcal{L}_{\mathcal{C}} - \frac{1}{2}\right) - \frac{1}{4\tau} \left(\sum_{i=1}^{k'}
   (y_i^2-2)\right) (2\Delta_z  + 1) + O(\tau^{-19/10}). \]
      On the left the operators are computed on $\mathcal{M}$, while
      on the right they are in terms of $\mathcal{C}_{n',k'}$. The Laplacian
      $\Delta_z$ is on the $S^{n'-k'}$-factors. 
  \end{enumerate}
\end{prop}
\begin{proof}
The statement (1) follows directly from the $C^1$-normal form
\eqref{eq:SXC1}, together with pseudolocality to obtain higher
derivative estimates. The bound $|A|^2 \leq \frac{1}{2} + \kappa_0$
follows from the fact that $F(\xi)\geq 1$, so at each point in
$B_{K_0\tau^{1/2}}$, $\mathcal{M}^{n', k'}_\tau$ has good graphicality
over a cylinder with radius at least $\sqrt{2(n'-k')}$.

To see the more precise formula \eqref{eq:Btaugammaw} on the smaller ball
$B_{\tau^\gamma}$ for small $\gamma > 0$, we show that on such a ball
the integral bound \eqref{eq:SXH1} can be improved to a $C^2$-bound,
using Proposition~\ref{prop:subsolpointwise}. We consider the
function
\[ V(x,\tau) = U(x,\tau) - \frac{\sqrt{2(n'-k')}}{4\tau}
  \sum_{i=1}^{k'} (y_i^2-2). \]
Note that the graph of $U$ satisfies the mean curvature flow, so we
have
\[ \partial_\tau( V +  E) = \mathcal{L}_{\mathcal{C}} (V + E) +
    Q(V+E), \]
  where we write
  \[ E = \frac{\sqrt{2(n'-k')}}{4\tau}  \sum_{i=1}^{k'} (y_i^2-2). \]
  We have $\partial_\tau E = -\tau^{-1}E$, and $\mathcal{L}(E) = 0$, so
\[ \label{eq:dtaue30} \partial_\tau V = \mathcal{L}_{\mathcal{C}}(V) + Q(V+E) +
  \tau^{-1}E. \]
 On $B_{\tau^\gamma}$ we have $\Vert E(\cdot, \tau)\Vert_{C^2} \leq
 \tau^{2\gamma-1}$. Using that the terms in $Q$ are all at least
 quadratic, and the second derivative $\nabla^2(V+E)$ appears at most
 linearly,  we can write
 \[ Q(V + E) = a_1\cdot \nabla V + a_2 V + O(\tau^{4\gamma -2}), \]
 where for any $\kappa_0 > 0$ we can ensure that $|a_1|, |a_2| <
 \kappa_0$ if $\tau$ is large enough. From \eqref{eq:dtaue30} we then
 have
 \[ \partial_\tau |V| \leq \Delta |V| - \frac{1}{2}x\cdot \nabla |V| +
   |V| + |a_1| |\nabla |V|| + |a_2| |V| + C\tau^{4\gamma -2}. \]
 In terms of $f = |V| + C\tau^{4\gamma -2}$ this implies
 \[ \partial_\tau f \leq \Delta f - \frac{1}{2} x\cdot \nabla f +
   (1+\kappa_0) f + |a_1| |\nabla f|. \]
 We let $\tilde{f}(x, \tau) = \chi(2|x| \tau^{-\gamma}) f(x, \tau)$
 for a cutoff function $\chi$ as before, so that $\tilde{f}$ is
 supported on $B_{\tau^\gamma}$. Using that we have the $L^2$-bound
 $\Vert f\Vert_{L^2} \leq C \tau^{4\gamma-2}$ from \eqref{eq:SXH1}, we
 can apply Proposition~\ref{prop:subsolpointwise} with $\beta=\gamma$,
 $\beta' = \gamma' < \gamma$, leading to the pointwise bound
 \[ f(x,\tau) \leq C \tau^{2\gamma'(1+\kappa_0) + 4\gamma-2}, \]
 on $B_{\tau^{\gamma'}}$, 
 for a $\gamma' < \gamma$. By choosing $\gamma, \gamma'$ sufficiently
 small, we can ensure that $|V|(x, \tau) \leq f(x,\tau) \leq \tau^{-19/10}$ for large
 $\tau$, and using Schauder estimates we find that the derivatives of
 $V$ are also $O(\tau^{-19/10})$.

 In particular this means that on $B_{\tau^\gamma}$, for large $\tau$,
 we can write $\mathcal{M}^{n',k'}_\tau$ as the graph of
 \[  w = \frac{\sqrt{2(n'-k')}}{4\tau}
   \sum_{i=1}^{k'} (y_i^2-2) + O(\tau^{-19/10}) \]
 over $\mathcal{C}_{n',k'}$, as required by (2).

 Using the formula for $w$ above, by direct calculation we can compare
 various operators on $\mathcal{M}^{n',k'}_\tau$ and
 $\mathcal{C}_{n',k'}$: 
    \[ \Delta  &= \Delta_{\mathcal{C}} -
        \frac{2w}{\sqrt{2(n'-k')}} \Delta_z +
        \frac{\sqrt{2(n'-k')}}{2}\nabla w\cdot \nabla 
            + O(\tau^{-19/10})\\
        x\cdot\nabla &= x\cdot \nabla_{\mathcal{C}} +
        \sqrt{2(n'-k')} \nabla w\cdot \nabla + O(\tau^{-19/10}) \\
        |A|^2 &= |A_{\mathcal{C}}|^2 - \frac{1}{\sqrt{2(n'-k')}} w +
        O(\tau^{-19/10}),
      \]
    where we view both sides as acting on functions on
    $\mathcal{C}_{n',k'}$, identifying this with
    $\mathcal{M}^{n',k'}_\tau$ by the graphicality. The comparison of
    the linearized operators in (3) follows from this. 
\end{proof}
    
We use these nondegenerate model flows to define $\ell$-degenerate
models converging to $\mathcal{C}_{n,k}$. Specifically, for $0 < k < n$
and $0< \ell < k$ we define
\[ \mathcal{M}^{n, k, \ell}_\tau = \mathbb{R}^{k-\ell}\times \mathcal{M}^{n-(k-\ell),
    \ell}_\tau. \]
Note that this rescaled flow converges to $\mathbb{R}^{k-\ell}\times
\mathbb{R}^\ell \times S^{n-k} = \mathcal{C}_{n,k}$ as $\tau\to
\infty$, and it has an $\ell$-degenerate
$\mathcal{C}_{n,k}$-singularity at infinity. Note also that since we take
$\ell < k$, in order to write down this $n$-dimensional flow, we only
need to have defined the nondegenerate models in lower
dimensions. Proposition~\ref{prop:models} implies an
analogous result for $\mathcal{M}^{n,k,\ell}_\tau$. For $\ell=0$ we
can define $\mathcal{M}^{n,k,0}_\tau = \mathcal{C}_{n,k}$. 

The fact that for
large $\tau$ the model is close to a cylinder on balls of radius
$o(\tau^{1/2})$ implies the following, using the pseudolocality
result, Proposition~\ref{prop:pseudo} (see also the proof of
\cite[Theorem 6.2]{SX22}). 
\begin{prop}\label{prop:modelpseudo}
  Given $\epsilon, K_0 > 0$ there exist $\delta_2(\epsilon), C_2(\epsilon),
  T_2(\epsilon, K_0)$ with the following property. Suppose that $M_\tau$ is
  a rescaled mean curvature flow, and for some $\tau_0 > T_2$, $M_{\tau_0}$
  is $\delta_2$-graphical over $\mathcal{M}^{n,k,\ell}_{\tau_0}$ on
  the ball $B_R$ for some $R < K_0\tau_0^{1/2}$. 
  Then for $\tau\in [\tau_0, \tau_0+10]$, and
  \[ R(\tau) = \min\{ K_0 \sqrt{\tau}, e^{(\tau-\tau_0)/2} (R - C_2)\}, \]
  we have
  that $M_\tau$ is $\epsilon$-graphical over
  $\mathcal{M}^{n,k,\ell}_{\tau}$ on the ball $B_{R(\tau)}$.
\end{prop}
\begin{proof}
From Proposition~\ref{prop:models} we know that given $K_0, \kappa_0 >
0$, if $\tau$ is large enough, then on balls of radius $R <
\tau^{1/3}$ in $B_{K_0\tau^{1/2}}$, the model flow $\mathcal{M}^{n',k'}_\tau$ is
$\kappa_0$-graphical over a cylinder $r\mathcal{C}_{n',k'}$ for some $r\in [1,
(1+K_0^2/2)^{1/2}]$. We can use Proposition~\ref{prop:pseudo},
rescaled, and translated, to
obtain a pseudolocality result for such cylinders of radius $r
\sqrt{2(n'-k')}$: Given $\epsilon > 0$, there exist $\delta_1(\epsilon, r),
C_1(\epsilon, r) > 0$ with the following properties. Let $N_{r,\tau}$
denote the rescaled mean curvature flow with initial condition $N_{r,0} = 
r\mathcal{C}_{n',k'}$. If $M_0$ is $\delta_1$-graphical over $N_{r,0}$ on
the ball $B_{R_0}(x_0)$ for some $R_0 > C_1$, then for $\tau\in [0,10]$,
$M_\tau$ is $\epsilon$-graphical over $N_{r,\tau}$ on the ball
$B_{e^{\tau/2}(R_0-C_1)}(e^{\tau/2}x_0)$. The constants $\delta_1, C_1$ are uniform
in $r$ as long as $r$ is bounded away from $0$ and $\infty$, so for a
fixed $K_0$ we can apply this result uniformly for the cylinders that
model our flow. For a given $\epsilon > 0$ we choose $\delta_1, C_1$ for this
pseudolocality result to hold for such cylinders. We let $R_0 =
C_1+1$.

For $R > 0$ as in the proposition, if $|x_0| < R-R_0$, then by
assumption $M_{\tau_0}$ is $\delta_2$-graphical over
$\mathcal{M}^{n',k'}_{\tau_0}$ on the ball $B_{R_0}(x_0)$. At the same
time, if $\tau_0$ is sufficiently large, then on this ball
$\mathcal{M}^{n',k'}_{\tau_0}$ is $\kappa_0$-graphical over
$r_0\mathcal{C}_{n',k'}$ for $r_0 = F(|y_0|\tau_0^{-1/2})$. By the
pseudolocality result above, if $\delta_2, \kappa_0$ are sufficiently
small, we find that for $\tau\in [\tau_0, \tau_0+10]$, $M_\tau$ is
$\epsilon/2$-graphical over $N_{r_0, \tau-\tau_0}$ on
$e^{(\tau-\tau_0)/2}B_{1}(x_0)$. We have
\[ N_{r_0, \tau-\tau_0} = (1 + e^{\tau-\tau_0}(r_0^2-1))^{1/2}
  \mathcal{C}_{n',k'}, \]
and
\[ (1 + e^{\tau-\tau_0} (r_0^2-1))^{1/2} = \left(1 + \frac{1}{2\tau}
    e^{\tau-\tau_0}|y_0|^2 \right)^{1/2}.\]
At the same time, on $e^{(\tau-\tau_0)/2} B_1(x_0)$, we know from
Proposition~\ref{prop:models} that
$\mathcal{M}^{n', k'}_\tau$ is $\kappa_0$-graphical over $r
\mathcal{C}_{n',k'}$, where
\[ r = F( e^{(\tau-\tau_0)/2}|y_0| \tau^{-1/2}) = \left( 1 +
    \frac{1}{2\tau} e^{\tau-\tau_0} |y_0|^2\right)^{1/2}. \]
If $\kappa_0$ is small enough, then we find that on
$e^{(\tau-\tau_0)/2} B_1(x_0)$, $M_\tau$ is $\epsilon$-graphical over
$\mathcal{M}^{n',k'}_\tau$.

The ball $B_{e^{(\tau-\tau_0)/2}(R-R_0)}$ can be covered using
balls $e^{(\tau-\tau_0)/2}B_1(x_0)$, where $|x_0| < R-R_0$, so we can set
$C_2 = R_0$. 
\end{proof}

Suppose now that $M_\tau$ is a RMCF with an $\ell$-degenerate
$\mathcal{C}_{n,k}$-singularity at infinity. By the normal form
theorem of Sun-Xue~\cite{SX22},
we know that for large $\tau$ we can write $M_\tau$ as a graph
over $\mathcal{M}^{n,k,\ell}_\tau$, however we want to obtain a more
quantitative estimate, in terms of the $L^2$ bounds given by
Proposition~\ref{prop:L2normalform}. We will assume that we are in the
setting where the conclusion of this proposition holds, and so,
replacing $T_0$ by the larger value $T_0+L_0$ from the proposition, we will
suppose that on the $\tau^\gamma$-ball, for some $\gamma > 0$,
$M_\tau$ is the graph of $u_0$ over $\mathcal{C}_{n,k}$, satisfying
$E(\tau) < \frac{1}{2}\tau^{-19/10}$ for all $\tau \geq T_0$. Here
$E(\tau)$ is as in \eqref{eq:L2decay10}. 
Choosing $T_0$
larger if necessary, we can assume that the model flow $\mathcal{M}^{n,k,\ell}_\tau$
also satisfies the same $L^2$-estimate.  This implies that we can
write $M_\tau$ as the graph 
of a function $u$ over $\mathcal{M}^{n,k,\ell}_\tau$ on the
$\tau^\gamma$-balls for $\tau\geq T_0$, satisfying $\Vert u(\cdot,
\tau)\Vert_{L^2(B_{\tau^\gamma})} <  \tau^{-19/10}$.

For later use, we will prove a more general result, roughly
speaking showing that if the $L^2$-norm decays at the rate of
$\tau^{-\alpha}$, then we have good graphicality on the $\tau^\beta$
balls for $\beta < \alpha/2$. First we improve graphicality from
$K\tau^\beta$-balls for small $K$, to graphicality on the $19\tau^\beta$-balls.
\begin{prop}\label{prop:Claim1}
  Suppose that for $\tau\in [T_0, 3T_0]$, $M_\tau$ is
$\epsilon_0$-graphical over $\mathcal{M}^{n,k,\ell}_\tau$ on
$B_{K\tau^\beta}$, with $K\in (0,20)$, $\beta\in (\beta_0, \alpha/2-\beta_0)$, and
$\Vert u\Vert_{L^2} < \tau^{-\alpha}$ for all $\tau\geq T_0$, with
$\alpha\in (0,2)$. If $\epsilon_0$ is sufficiently small (depending on
$\beta_0$), and 
$KT_0^\beta$ is sufficiently large, then there exists a $T_1$,
depending on $\beta_0, T_0, K$, such that $M_\tau$
is $\epsilon_0$-graphical over $\mathcal{M}^{n,k,\ell}_\tau$ on the
possibly larger $19\tau^\beta$-balls, for all $\tau \geq T_1$.
\end{prop}
\begin{proof}First we extend the
$\epsilon_0$-graphicality on the $K\tau^\beta$-balls for all $\tau
\geq T_0$. We can apply the same argument as the proof of
Proposition~\ref{prop:tbetagraph10}. For this note that
the graphical function $u(x,\tau)$ of $M_\tau$ over
$\mathcal{M}^{n,k,\ell}_\tau$ satisfies an equation of the form
\[ \label{eq:ptu50} \partial_\tau u = \Delta u - \frac{1}{2} x\cdot \nabla u +
  \left(\frac{1}{2} + |A|^2\right) u + Q(u, \nabla u, \nabla^2 u). \]
The argument of Proposition~\ref{prop:tbetagraph10} can be applied in
this setting too, with the model flow $\mathcal{M}^{n,k,\ell}_\tau$
replacing the static flow $\mathcal{C}_{n,k}$, for sufficiently large
$\tau$, since the only properties of the cylinder that we used were
the pseudolocality, and the fact that for any $\epsilon > 0$  the second fundamental form
satisfies $|A|^2 < \frac{1}{2} + \epsilon$. This holds for our model
flow as well, once $\tau$ is large, depending on $\epsilon$.

The argument of Proposition~\ref{prop:tbetagraph10} 
implies that on the $K\tau^\beta$-balls we have
\[ |\nabla^i u| \leq C \tau^{2\beta(2\epsilon+1) - \alpha} +
  CK^{-1}\tau^{-\beta}. \]
For any $\epsilon_1 > 0$ we can arrange that this quantity is less
than $\epsilon_1$ if $KT_0^\beta$ is sufficiently large. So we can
assume that for $\tau \in [2T_0, 3T_0]$ we have
$\epsilon_1$-graphicality on the $K\tau^\beta$-balls. If we choose
$\epsilon_1$ sufficiently small, Proposition~\ref{prop:modelpseudo}
implies that we have $\epsilon_0$-graphicality for $\tau\in[2T_0+1,
3T_0+1]$ on the balls of radius $e^{1/2}(K(\tau-1)^\beta-C_1)$.
This is larger than $1.1K\tau^\beta$, and in particular than
$K\tau^\beta$, if $K\tau^\beta$ is large 
enough. We can keep repeating this argument to get
$\epsilon_0$-graphicality on the $K\tau^\beta$-balls for all
$\tau\geq T_0$.

After this, we can apply the same argument again, to find that for all
$\tau \geq 2T_0 +1$ we have $\epsilon_0$-graphicality on the larger
balls of radius $1.1K\tau^\beta$. If $1.1K > 19$, we are
done. Otherwise we can keep repeating this argument
several more times to obtain $\epsilon_0$-graphicality on the $1.1^j
K\tau^\beta$-balls, for $\tau \geq T_{0,j}$, with suitable
$T_{0,j}$. After finitely many steps we find that we have
$\epsilon_0$-graphicality on the $19\tau^\beta$-balls, for all $\tau
\geq T_1$, for some $T_1$ depending on $T_0, K$. Note that as $K$ and
$\tau$ increase, the condition that $K\tau^\beta$ is sufficiently
large persists.  \end{proof}

The next result will be used to increase the power $\gamma$ when we
have good graphicality on the $\tau^\gamma$-balls. 

\begin{prop}\label{prop:Claim2}
Suppose that for $\tau\in [T_0, 3T_0]$, $M_\tau$
is $\epsilon_0$-graphical over $\mathcal{M}^{n,k,\ell}_\tau$
on $B_{\tau^\gamma}$, and in addition $\Vert
u\Vert_{L^2} < \tau^{-19/10}$, for all $\tau \geq T_0$. Then there
exists a $T_1$, depending on $T_0$, such that $M_\tau$ is
$\epsilon_0$-graphical on $B_{\tau^{\gamma'}}$ for all $\tau \geq
T_1$, where
\[ \gamma' = \min\{ \frac{3}{2}\gamma, \frac{1}{2}\}. \]
\end{prop}

\begin{proof}
For simplicity we assume $\gamma'
= \frac{3}{2}\gamma$, since the case $\gamma'=1/2$ is similar. We
choose $K$ so that for $\tau\in[T_0, 3T_0]$ the $\tau^\gamma$-ball
contains the $K\tau^{3\gamma/2}$-ball. For this we can take $K =
(3T_0)^{-\frac{1}{2}\gamma}$, since
\[ (3T_0)^{-\frac{1}{2}\gamma} \tau^{\frac{3}{2}\gamma} \leq
  \tau^\gamma, \]
for $\tau \leq 3T_0$. We now apply Proposition~\ref{prop:Claim1}, with
$\beta=\gamma'$. Note that
\[ K T_0^{\gamma'} = 3^{-\frac{1}{2}\gamma} T_0^{-\frac{1}{2}\gamma +
    \frac{3}{2}\gamma} \]
will be large enough to apply that result, if $T_0$ is sufficiently
large. The conclusion is that $M_\tau$ is $\epsilon_0$-graphical over
$\mathcal{M}^{n,k,\ell}_\tau$ on the $\tau^{\frac{3}{2}\gamma}$-balls
for all $\tau \geq T_1$, with $T_1$ depending on $T_0$. 
\end{proof}

Using these two results, we can show good graphicality of the flow
$M_\tau$ over $\mathcal{M}^{n,k,\ell}_\tau$ on the larger
$19\tau^{1/2}$-balls.
\begin{prop}\label{prop:ldegendecay}
  Suppose that $M_\tau$ has an $\ell$-degenerate
  $\mathcal{C}_{n,k}$-singularity at infinity and let $\epsilon_0 >
  0$. Suppose that for some 
  $T_0 > 1$ we have that $M_\tau$ is $\tau^{-\gamma}$-graphical over
  $\mathcal{M}^{n,k,\ell}_\tau$ on $B_{\tau^\gamma}$, and the
  graphicality function $u$ satisfies $\Vert u\Vert_{L^2} \leq
  \tau^{-19/10}$ for all $\tau \geq T_0$.

  There exists a $T_1$, depending on $T_0, \gamma, \epsilon_0$, such
  that $M_\tau$ is $\epsilon_0$-graphical over
  $\mathcal{M}^{n,k,\ell}_\tau$ on the ball $B_{19\tau^{1/2}}$ for all
  $\tau \geq T_1$.   In addition on the ball $B_{\tau^\gamma}$ we have the improved
  estimate $\Vert u\Vert_{C^4} \leq \tau^{-3/2}$. 
\end{prop}
\begin{proof}
We apply Proposition~\ref{prop:Claim2} multiple times until we reach good graphicality
on the  $\tau^{1/2}$-balls, for sufficiently large $\tau$. Then we can apply
Proposition~\ref{prop:Claim1} to improve graphicality to the
$19\tau^{1/2}$-balls, for even larger $\tau$. 

To obtain the better graphicality estimate on the smaller
$\tau^\gamma$-balls, we apply Proposition~\ref{prop:subsolpointwise}. We
set $\beta = 1/2$, $\beta' = \gamma$, $\alpha=19/10$, and $b_0 =
1+\epsilon$. The conclusion is that on the $\tau^\gamma$-balls we have
\[ |u| \leq C\tau^{2\gamma(2\epsilon+1)-19/10} <
  \tau^{-3/2} \]
if $\tau$ is large, and $\gamma< 1/5$, say. The derivative estimates
follow from parabolic Schauder estimates. 
\end{proof}

Note that the model flows $\mathcal{M}^{n,k,\ell}$ are invariant under
translations in the degenerate $y_i$ directions, i.e. for $i > \ell$. 
A crucial new ingredient in this paper is the following result, showing
that up to an exponentially small error, the same holds for any
rescaled flow $M_\tau$ that has an 
$\ell$-degenerate $\mathcal{C}_{n,k}$-singularity at infinity. In
other words, the derivatives of the graphicality function of $M_\tau$ in the degenerate
directions decay exponentially fast in $\tau$. A related result in the setting
of ancient flows was shown by Choi-Haslhofer~\cite[Theorem
1.10]{CH24}. While the main focus is the case $\ell > 0$, note that
the result holds for $\ell=0$ as well, and in fact $u$ itself decays
exponentially fast in that case (see \cite[Proposition 5]{Sz26}). 

\begin{prop}\label{prop:uqdecay}
  There exists $\epsilon_0, \kappa_0 > 0$, and for every $T_0 > 0$
  there is $T_1 > 0$ with the following properties. 
  Suppose that $M_\tau$ has an $\ell$-degenerate
  $\mathcal{C}_{n,k}$-singularity at infinity, and suppose that
  for all $\tau \geq T_0$ we have:
  \begin{itemize}
    \item The flow $M_\tau$ can be written
  as the graph of $u(\cdot, \tau)$ over the model
  $\mathcal{M}^{n,k,\ell}_\tau$ on the ball $B_{19\tau^{1/2}}$,
  with   $\Vert u\Vert_{C^4} < \epsilon_0$, and
$\Vert u(\cdot,  \tau)\Vert_{L^2(B_{10\tau^{1/2}})} <
\tau^{-19/10}$.
\end{itemize}
  Then for all $\tau > T_1$ and $q=\ell+1, \ldots,k$ we have the estimate
  \[ \Vert \partial_{y_q} u(\cdot, \tau)\Vert_{L^2(B_{10\tau^{1/2}})}
    \leq  e^{-\kappa_0 \tau}, \]
  for the $y_q$-derivative. 
\end{prop}

To prove this we will first show some preliminary results, the overall
strategy being somewhat similar to the proof of
Proposition~\ref{prop:L2normalform} on the $L^2$ normal form.

Instead of computing the $L^2$-norm of $\partial_{y_q}u$
on the $10\tau^{1/2}$-ball, we can work on
the smaller $\kappa_0^{1/4}\tau^{1/2}$-ball, since the
contribution to the $L^2$-norm from outside of this ball is bounded by
$e^{-\kappa_0^{1/2}\tau/5}$, and if $\kappa_0$ is chosen small enough,
then this is of much lower order than $e^{-\kappa_0\tau}$. The
advantage of this is that for any given $\epsilon_0>0$,
if $\kappa_0$ is sufficiently small, then on the
$\kappa_0^{1/4}\tau^{1/2}$-ball the model flow
$\mathcal{M}^{n,k,\ell}_\tau$ is $\epsilon_0$-graphical over the
cylinder $\mathcal{C}_{n,k}$. Therefore the flow $M_\tau$ is also
$2\epsilon_0$-graphical over $\mathcal{C}_{n,k}$ on these balls. Note
that it will still be important to view $M_\tau$ as a graph over
$\mathcal{M}^{n,k,\ell}_\tau$, since relative to $\mathcal{C}_{n,k}$
we would not have the estimate $\Vert u\Vert_{L^2} < \tau^{-19/10}$, but
rather only $\Vert u\Vert_{L^2} = O(\tau^{-1})$ instead.

We differentiate \eqref{eq:ptu50} with respect to $y_q$, and note that
$\mathcal{M}^{n,k,\ell}_\tau$ is invariant in this direction, since $q
> \ell$. We find that $u_q = \partial_{y_q} u$ satisfies
\[ \label{eq:uqtildeeq} \partial_\tau u_q = \Delta u_q - \frac{1}{2} x \cdot \nabla u_q +
  |A|^2 u_q + \sum_{i=0}^2 Q_i \ast \nabla^i u_q, \]
where the $Q_i$ are derivatives of $Q$ with respect to its
entries. Note that $Q$ depends on $x$ too, but is invariant in the
$y_q$-direction -- otherwise we would have additional terms involving
$u$ or derivatives of $u$ other than the $y_q$-derivative. 
Let us write $\tilde{u}_q(x,\tau) = \chi(\kappa_0^{-1/4}\tau^{-1/2} |x|)
u_q(x,\tau)$, where $\chi$ is the cutoff function satisfying
$\chi(x)=0$ for $x >1$, and $\chi(x)=1$ for $x < 1/2$. Then 
$\tilde{u}$ is supported on $B_{\kappa_0^{1/4}\tau^{1/2}}$. Similarly to
\eqref{eq:utildeeq}, this function satisfies
\[\label{eq:uqtildeeq2}
  \partial_\tau \tilde{u}_q = \Delta \tilde{u}_q - \frac{1}{2} x \cdot \nabla \tilde{u}_q +
  |A|^2 \tilde{u}_q + \sum_{i=0}^2 Q_i \ast \nabla^i \tilde{u}_q +
  \mathcal{E}, \]
where $\mathcal{E}$ is supported outside of
$B_{\frac{1}{2}\kappa_0^{1/4}\tau^{1/2}}$. We have a uniform bound $|A|^2 <
\frac{1}{2}+\epsilon_0$, and the terms $Q_i$ are also of order
$\epsilon_0$, so we can apply Proposition~\ref{prop:logSobolev} and
Proposition~\ref{prop:L2pointwise} to obtain the following. We use
$e^{-2\kappa_0\tau}$ in the error terms. Note that this is of lower
order than $e^{-\kappa_0\tau}$, but it is larger than error terms of size
$e^{-\kappa_0^{1/2}\tau/20}$ that show up from the cutting off the
Gaussian weight outside the
$\frac{1}{2}\kappa_0^{1/4}\tau^{1/2}$-ball. 
\begin{prop}\label{prop:uqnonconc}
  There is a $C > 0$ satisfying the following.
  Suppose that $\tau_0 \geq T_0$ and $\tau_0 \leq \tau \leq \tau_0 +
  10$.  We have 
  \[ \Vert u_q(\cdot, \tau) \Vert_{L^2} \leq \Vert u_q(\cdot, \tau)
    \Vert_{L^{p(\tau-\tau_0)}} \leq e^{C(\tau-\tau_0)} \Vert 
    u_q(\cdot, \tau_0)\Vert_{L^2} + e^{-2\kappa_0\tau_0}. \]
  In addition, there is a $p > 1$, such that for $\tau\in [\tau_0+1,
  \tau_0+10]$, on the ball $B_{\kappa_0^{1/4}\tau^{1/2}}$ we have the
  pointwise estimates
  \[ |\nabla^i u_q| (x,\tau) &\leq  C e^{|x|^2 / 8p} (\Vert u_q(\cdot,
    \tau_0)\Vert_{L^2} + e^{-2\kappa_0 \tau_0}),  \]
  for $i\leq 3$. For the function $u$ itself we have
  \[\label{eq:410} |\nabla^i u|(x, \tau) \leq C\tau^{-19/10}
    e^{|x|^2/8p},  \]
  using the $L^2$-bound $\Vert u\Vert_{L^2} \leq \tau^{-19/10}$. 
\end{prop}

Next, we define a modified flow $\tilde{M}_\tau = S_\tau M_\tau$ using
suitable rotations $S_\tau$ in the $\mathbb{R}^k$-factor, in order to
remove the components $y_i$ of $u_q$ for the non-degenerate
directions $i\leq \ell$. Note that when $\ell=0$, then there are no
nondegenerate directions, so $\tilde{M}_\tau = M_\tau$.  For this we will use the following.
\begin{prop}\label{prop:rotated3}
  Let $\delta > 0$. We can let $T_0$ be large, depending on $\delta$,
  such that for $\tau \geq T_0$, we have the following.
Let $R_\theta$ denote the rotation by angle $\theta$
  in the $y_jy_q$-plane, for some $j\leq \ell$ and $q > \ell$. 
Then if $\theta \leq \tau^{-9/10}$,
  $R_\theta M_\tau$ is the graph of a function $v(x,\tau)$
  over $\mathcal{M}^{n,k,\ell}_\tau$ on the $\kappa_0^{1/4}\tau^{1/2}$-ball, where
  \[ \label{eq:420}|\nabla^i (v(x,\tau) - u(x,\tau))| \leq C \tau^{-1} |\theta| e^{|x|^2/8p},  \]
  for the $p > 1$ from Proposition~\ref{prop:uqnonconc}.
  On the smaller ball
  $B_{\tau^\gamma}$ for small $\gamma > 0$, we have the improved estimate
  \[\label{eq:421}  |\nabla^i (v(x,\tau) - u(x,\tau) - c_{n,k}
    \tau^{-1}\theta y_jy_q)| \leq C_\delta  \tau^{\gamma-3/2} |\theta|,  \]
  where $c_{n,k}$ is a fixed constant.
\end{prop}
\begin{proof}
     First we examine the effect of the rotation on the model
     $\mathcal{M}^{n,k,\ell}_\tau$. Note that in the $H^1$ normal form
     \eqref{eq:SXH1}, the effect of rotating in the $y_jy_q$-plane by
     angle $\theta$ is to replace $y_j$ by $y_j + \theta y_q$ up to
     terms of order $|\theta|^2 |y|$. If $|\theta| <
     \tau^{-9/10}$, then $|\theta| |y| \ll 1$ for large $\tau$ in the
     ball $B_{\kappa_0^{1/4}\tau^{1/2}}$. It follows that the rotation
     $R_\theta\mathcal{M}^{n,k,\ell}_\tau$ is the graph of a
     function $W$ over $\mathcal{M}^{n,k,\ell}_\tau$, where
     \[ \left\Vert W(y,\tau) - \frac{\sqrt{2(n'-k')}}{2\tau} \theta y_j y_q
         \right\Vert_{L^2} \leq C \tau^{1/2} |\theta|
         \tau^{\gamma-2}. \]
     Note that $W$ is essentially $\theta y_q$ times the
     $y_j$-derivative of the graphicality function $U$ of
     $\mathcal{M}^{n,k,\ell}_\tau$. 

     Next we use that $M_\tau$ is the graph of $u$ over
     $\mathcal{M}^{n,k,\ell}_\tau$, where $u$ satisfies
     \eqref{eq:410}. The effect of rotating this function by
     $\theta$ can be bounded by $|y||\theta| \tau^{-19/10} e^{|x|^2/8p}
     \leq \tau^{-19/10} |\theta||x|e^{|x|^2/8p}$. Combining these two
     results it follows that we can write $R_\theta M_\tau$
     as the graph of a function $w$ over $M_\tau$, where 
 \[ \label{eq:vL251} \left\Vert w(y,\tau) - \frac{\sqrt{2(n'-k')}}{2\tau} \theta y_j y_q
   \right\Vert_{L^2} \leq |\theta| \tau^{\gamma -3/2}, \]
  for sufficiently large $\tau$, for any $\gamma > 0$. 

     We can now use the same techniques as before to upgrade this
     $L^2$-estimate at a time $\tau-1$ to a pointwise estimate at time
     $\tau$.  We find that
     \[ |\nabla^i w|(x,\tau) \leq C|\theta| \tau^{-1}
       e^{|x|^2/8p}, \]
     on the $\kappa_0^{1/4} \tau^{1/2}$-ball.
   \end{proof}

Using this result, we can construct rotations $S_\tau$ in the
$\mathbb{R}^k$ factor, so that the derivative $\partial_{y_q}v$ of
the graphicality function $v$ of
$S_\tau M_\tau$ over $\mathcal{M}^{n,k,\ell}_\tau$ is orthogonal to
the nondegenerate $y_j$ directions on the
$\kappa_0^{1/4}\tau^{1/2}$-ball. To simplify the later argument, we
also incorporate a cutoff function at this stage. 
\begin{prop}\label{prop:rotated2}
  Let $\delta > 0$. If $\tau$ is sufficiently large, depending on
  $\delta$, then we can find a rotation $S_\tau$ of $\mathbb{R}^k$,
  with $|S_\tau - Id| < C\tau^{-9/10}$, such that we have
  \begin{itemize}
    \item[(i)] $\tilde{M}_\tau = S_\tau M_\tau$ is
      the graph of $v(x, \tau)$ over $\mathcal{M}^{n,k,\ell}_\tau$ on
      $B_{10\sqrt{\tau}}$,
    \item[(ii)] With $\chi$ the cutoff function as before, define
      $\tilde{v}(x,\tau) =
      \chi(|x|\kappa_0^{-1/4}\tau^{-1/2})v(x,\tau)$. Then $\partial_{y_q}
      \tilde{v}(x, \tau)$ is $L^2$-orthogonal to each 
      $y_i$ for $i=1,\ldots, \ell$, when viewed
      as functions on $\mathcal{C}_{n,k}$.
    \item[(iii)] On the $\kappa_0^{1/4}\tau^{1/2}$-ball we have
      \[ \label{eq:407} |\nabla^i \tilde v(x,\tau)| \leq C\tau^{-19/10} e^{|x|^2/8p}, \]
      for $i=0,1,2$ and some fixed $p > 1$. 
    \end{itemize}
  \end{prop}
  \begin{proof}
    Note that from Proposition~\ref{prop:uqnonconc} we have $\Vert
    \partial_{y_q} u\Vert_{L^2} \leq C\tau^{-19/10}$, and so,
    according to \eqref{eq:421} in
    Proposition~\ref{prop:rotated3}, we can use rotations by angles
    $|\theta| \leq C \tau^{1-19/10} = C\tau^{-9/10}$ to eliminate the 
    $y_i$-component of $\partial_{y_q}u$. This way we can define the
    rotations $S_\tau$, satisfying $|S_\tau - Id| < C
    \tau^{-9/10}$. Writing $\tilde{M}_\tau = S_\tau M_\tau$ as the
    graph of $v$ over $\mathcal{M}^{n,k,\ell}_\tau$, the estimates for
    $v$ follow from \eqref{eq:420} in Proposition~\ref{prop:rotated3}. 
  \end{proof}

Using
\eqref{eq:uqtildeeq}, the function $\tilde{v}(x,\tau) = \chi(\kappa_0^{-1/4}\tau^{-1/2} |x|)
v(x,\tau)$ above satisfies the following equation, similarly to 
\eqref{eq:vtildeeq2}: 
\[  \label{eq:vtildeeq10} \partial_\tau \tilde{v}_q = \left(\Delta -
    \frac{1}{2}x\cdot \nabla + |A|^2\right) \tilde{v}_q +
  \mathcal{Q}_{v}(\tilde{v}_q) + \mathcal{R}_{\xi_\tau}(\tilde{v}) +
    \chi \xi_\tau  + \mathcal{E}.  \]
Here, and below, we write $\tilde{v}_q = \partial_{y_q}\tilde{v}$, 
$|A|$ is the second fundamental form of
$\mathcal{M}^{n,k,\ell}_\tau$, and for each $\tau$, $\xi_\tau$ is a function in the
span of $y_i$ for $i\leq \ell$, corresponding to the derivative
$S_\tau^{-1}\dot S_\tau$. Note that the norm $|\xi_\tau|$ corresponds
to the infinitesimal change in the graphicality function due to our
rotations, and as in Proposition~\ref{prop:rotated3}, this is of order
$\tau^{-1} |S_\tau^{-1}\dot S_\tau|$. 
The cutoff error term $\mathcal{E}$ is
supported outside of the $\frac{1}{2}\kappa_0^{1/4} \tau^{1/2}$-ball,
and is bounded by $1$. The term $\mathcal{R}_{\xi_\tau}$ arises from the
lower order terms corresponding to the infinitesimal rotation
$S_\tau^{-1}\dot S_\tau$, and as a result we have
\[ \label{eq:Rxbound} \Vert \mathcal{R}_{\xi_\tau}\Vert_{L^2} \leq C \tau^{-9/10}
  |\xi_\tau|. \]
By modifying $\xi_\tau$ slightly, we can ensure that
$\mathcal{R}_{\xi_\tau}(\tilde{v})$ is orthogonal to
$\xi_\tau$. Throughout we view the functions that appear as functions on
  $\mathcal{C}$, and integrate using the measure $d\mu = e^{-|x|^2/4}
  d\mathcal{H}^n$ with the area form of $\mathcal{C}$.

The following is similar to Proposition~\ref{prop:L2ptwise}. 
\begin{prop}\label{prop:xitest30} 
  Suppose that for some $\tau \geq T_0+1$, and large $T_0$, we have 
  \[ \Vert v_q(\cdot, \tau-1)\Vert_{L^2} \leq 2 \Vert v_q(\cdot,
    \tau)\Vert_{L^2}. \]
  Then for the $p$ in Proposition~\ref{prop:uqnonconc} we have the estimate
 \[ \label{eq:401} | \nabla^i v_q|(x, \tau) \leq Ce^{|x|^2/8p} (\Vert
   v_q(\cdot, \tau)\Vert_{L^2} + e^{-2\kappa_0 \tau}),\]
 on the $\kappa_0^{1/4} \tau^{1/2}$-ball. In addition we have 
   \[ \label{eq:xitaubound5} |\xi_\tau| \leq C( \tau^{-1} \Vert v_q(\cdot, \tau)\Vert_{L^2} +
      e^{-2\kappa_0 \tau}). \] 
\end{prop}
\begin{proof}
  To prove \eqref{eq:401}, we can assume that $S_{\tau-1}=Id$. Then
  Proposition~\ref{prop:uqnonconc} implies that the estimate
  \eqref{eq:401} holds with $\nabla^i u_q$ on the left hand side. It
  follows from Proposition~\ref{prop:rotated3} that to get $S_\tau
  M_\tau$ from $M_\tau$, it is enough to consider a rotation $S_\tau$
  with $|S_\tau - Id| \leq C\tau \Vert u_q(\cdot, \tau)\Vert_{L^2}$,
  and the new graphicality function $v(x, \tau)$ satisfies
  \[ |\nabla^i(v(x, \tau) - u(x,\tau))|&\leq C \tau \Vert u_q(\cdot,
    \tau)\Vert_{L^2} \tau^{-1} e^{|x|^2/8p} \\
    &= C \Vert u_q(\cdot, \tau)\Vert_{L^2}  e^{|x|^2/8p},  \]
  on the $\kappa_0^{1/4}\tau^{1/2}$-ball. 
  From this it follows that if we have the estimate \eqref{eq:401}
  with $u_q$ on the left hand side, then it also holds with $v_q$. 

  To get the estimate for $\xi_\tau$, we multiply \eqref{eq:vtildeeq10} by $\xi_\tau$, and
  integrate. Note that by construction $\partial_\tau \tilde{v}_q$ and
  $\mathcal{R}_{\xi_\tau}(\tilde{v})$ are orthogonal to $\xi_\tau$. To
  estimate the linearized operator acting on $\tilde{v}_q$, we use
Proposition~\ref{prop:models} to find that
\[ \label{eq:41} \left(\Delta -
      \frac{1}{2}x\cdot \nabla + |A|^2\right) \tilde{v}_q =
    (\mathcal{L}_{\mathcal{C}} - \frac{1}{2}) \tilde{v}_q - \frac{1}{4\tau} \left(\sum_{i=1}^\ell
   (y_i^2-2)\right) (2\Delta_z \tilde{v}_q + \tilde{v}_q) + \mathcal{V}, \]
  where using \eqref{eq:401} from Proposition~\ref{prop:xitest30} we have 
  \[  \label{eq:Vxtaubound}|\mathcal{V}(x,\tau)| \leq C\tau^{-19/10} e^{|x|^2/8p} (\Vert v_q(\cdot,
    \tau)\Vert_{L^2} + e^{-2\kappa_0 \tau}). \]
  The inner product of $(\mathcal{L}_{\mathcal{C}} - \frac{1}{2})
  \tilde{v}_q$ with $\xi_\tau$ vanishes, because $\xi_\tau$ is in the
  kernel of $\mathcal{L}_{\mathcal{C}} - \frac{1}{2}$. Using
  \eqref{eq:401} we find that
  \[ \left| \int_{\mathcal{C}} \xi_\tau \left(\Delta - \frac{1}{2}
        x\cdot \nabla + |A|^2\right)\, d\mu \right| \leq C\tau^{-1}
    |\xi_\tau| \Vert v_q(\cdot, \tau)\Vert_{L^2} + C |\xi_\tau|
    e^{-2\kappa_0 \tau}. \]
  The term $\mathcal{Q}_v(\tilde{v}_q)$ is linear
  in $\tilde{v}_q, \nabla \tilde{v}_q, \nabla^2\tilde{v}_q$, and the
  coefficients are controlled by $\Vert v\Vert_{C^2}$. Using also that
  $\mathcal{E}$ is bounded and supported outside of the
  $\frac{1}{2}\kappa_0^{1/4} \tau^{1/2}$-ball, we get that
  \[ \int_{\mathcal{C}} \left|\mathcal{Q}_v(\tilde{v}_q) +
      \mathcal{E}\right| |\xi_\tau|\, d\mu  \leq C \tau^{-19/10}
    |\xi_\tau| \Vert v_q(\cdot, \tau)\Vert_{L^2} + 
    C |\xi_\tau| e^{-2\kappa_0 \tau}.  \]
  Finally the integral of $\chi \xi_\tau^2$ is $(1 - \Psi(\tau^{-1}))
  |\xi_\tau|^2$. Combining all of these results and taking the inner
  product of the equation \eqref{eq:vtildeeq10} with $\xi_\tau$, we
  find that
  \[ (1 - \Psi(\tau^{-1})) |\xi_\tau|^2 \leq C \tau^{-1} |\xi_\tau|
    \Vert v_q(\cdot, \tau)\Vert_{L^2} + C |\xi_\tau| e^{-2\kappa_0
      \tau}, \]
  and from this the estimate for $|\xi_\tau|$ follows once $\tau$ is
  large. 
\end{proof}

In a similar way to Proposition~\ref{prop:growthdominant} we also have
the following, which shows that under a slow decay condition the $z_\alpha$ and degenerate
$y_j$ components dominate $v_q$. 
\begin{prop}\label{prop:uqdominant2}
Let $\lambda > 0$. There exist $\epsilon > 0$,
depending on $\lambda$, such that if $T_0$ is sufficiently large,
then we have the following. Suppose that for some
$T_1 \geq T_0$ the
graphical function $v$ of the rotated flow $\tilde{M}_\tau$ defined
in Proposition~\ref{prop:rotated2} satisfies
\[ \label{eq:vqslow} \Vert v_q (\cdot, T_1 + 1)\Vert_{L^2} \geq
  e^{-\epsilon } \Vert v_q (\cdot,
  T_1)\Vert_{L^2}, \]
and $\Vert v_q(\cdot, T_1)\Vert_{L^2} \geq
e^{-2\kappa_0 T_1}$. Then we have
\[ \label{eq:vqslow2}\Vert v_q (\cdot, T_1 + 2)\Vert_{L^2} \geq
  e^{-\epsilon } \Vert v_q (\cdot,
  T_1+1)\Vert_{L^2}, \]
and in addition  for all $\tau \geq T_1+1$ we have
\[ \label{eq:vqdom20} \Vert \Pi_{\mathfrak{h}_1} v_q(\cdot,
  \tau)\Vert_{L^2(B_{\delta^{-1}})} \geq (1-\lambda) \Vert v_q(\cdot,
  \tau)\Vert_{L^2(B_{\delta^{-1}})}, \]
where $\mathfrak{h}_1$ is spanned by the coordinate functions
$z_\alpha$, and $y_j$ for $j > \ell$. 
\end{prop}
\begin{proof}
  First we show that for a given small $\epsilon > 0$, if $T_0$ is
  sufficiently large, then the condition 
  \eqref{eq:vqslow} together with $\Vert v_q(\cdot, T_1)\Vert_{L^2}
  \geq e^{-2\kappa_0 T_1}$ implies \eqref{eq:vqslow2}. The proof is
  very similar to Proposition~\ref{prop:3ann3}, except we need to
  account for the error arising from the cutoff function. We argue by
  contradiction, assuming that we have a sequence of flows $\tilde{M}^i_\tau =
  S^i_\tau M^i_\tau$, with corresponding  $T^i_1 \to \infty$, such
  that the hypotheses \eqref{eq:vqslow} 
  and $\Vert v^i_q(\cdot, T^i_1)\Vert_{L^2} \geq e^{-2\kappa_0 T^i_1}$
  hold, but the conclusion \eqref{eq:vqslow2} fails, so
  \[ \label{eq:vqslow10} \Vert v^i_q(\cdot, T^i_1+2)\Vert_{L^2} \leq e^{-\epsilon} \Vert
    v^i_q(\cdot, T^i_1+1)\Vert_{L^2}. \]

  As before, by rotating the whole flow by a rotation of order
  $(T^i_1)^{-9/10}$  we can assume that
  $S^i_{T^i_1} = Id$, so $\tilde{M}^i_{T^i_1} = M^i_{T^i_1}$. Let us
  write $u^i$ and $u^i_q$ for the graphicality
  functions corresponding to $M^i_\tau$ and their $y_q$-derivatives. 
  We define $\mathbf{d}_i = \Vert v_q(\cdot,
  T^i_1+1)\Vert_{L^2}$. Then by our assumption we have
  \[ \Vert u_q^i(\cdot, T^i_1)\Vert = \Vert v_q^i(\cdot, T^i_1)\Vert_{L^2}
    \leq e^\epsilon \mathbf{d}_i. \]
  From Proposition~\ref{prop:uqnonconc}, for $\tau\in [T^i_1,
  T^i_1+2]$ we have
  \[ \label{eq:301} \Vert u_q^i(\cdot, \tau)\Vert_{L^2} &\leq e^{C(\tau - T^i_1)}
    e^{\epsilon}\mathbf{d}_i + C e^{-\kappa_0^{1/2} T^i_1/10} \\
    &\leq ( e^{C(\tau - T^i_1)}
    e^{\epsilon} + 2C e^{-\kappa_0^{1/2} T^i_1/10}) \mathbf{d}_i, \]
  for large $i$, where we used that $\mathbf{d}_i \geq e^{-2\kappa_0 T^i_1}$
 In addition, for some $p > 1$ and $j=0,1,2$, 
  \[ \label{eq:302} |\nabla^j u_q^i|(x,T^i_1+1) \leq C e^{|x|^2/8p}  (
    1+ e^{-\kappa_0^{1/2} T^i_1/10}) \mathbf{d}_i. \]

  After translating in time by $T^i_1$, we can
  extract a limit $\mathbf{d}_i^{-1}u^i_q \to u_q^\infty$ on $\mathcal{C}_{n,k}$ on the time
  interval $(0,2]$, converging smoothly on compact sets, satisfying the equation
  \[ \label{eq:lineq20} \partial_\tau u_q^\infty = \Delta u^\infty_q - \frac{1}{2} x\cdot
    \nabla u_q^\infty + \frac{1}{2} u^\infty_q. \]
  Note that we do not need to know that there is 
  a function $u^\infty$ whose $y_q$-derivative is $u^\infty_q$
  (although we could integrate $u^\infty_q$ to obtain one). 
  Arguing as
  in Proposition~\ref{prop:3ann3} we have that $v^i_q = u^i_q +
  \xi^i_\tau$, where for each $\tau$ (after translating in time by
  $T^i_1$), up to choosing a subsequence, the
  $\mathbf{d}_i^{-1}\xi^i_\tau$ converge to a function
  $\xi^\infty_\tau$ in the span of $y_j$ for $j\leq 
  \ell$. In addition we find that the limit $\xi^\infty_\tau =
  \xi^\infty$ is independent of $\tau$,
  and so the $\mathbf{d}_i^{-1}v^i_q$ converge to $v^\infty_q = u^\infty_q +
  \xi^\infty$, which are orthogonal to $y_j$ for $j\leq \ell$ and also
  solve the equation \eqref{eq:lineq20}.
  Because  of \eqref{eq:301} we have
  \[ \liminf_{\tau\to 0} \Vert v^\infty_q(\cdot, \tau)\Vert_{L^2} \leq
    \liminf_{\tau\to 0} \Vert u^\infty_q(\cdot, \tau)\Vert_{L^2} \leq
    e^{\epsilon}, \]
  and from \eqref{eq:vqslow10} we also have
  \[ \Vert v^\infty_q(\cdot, 2)\Vert_{L^2} \leq e^{-\epsilon}. \]
  From \eqref{eq:302}, by integrating on larger and larger balls, we
  find that
  \[ \Vert v^\infty_q(\cdot, 1)\Vert_{L^2} = 1. \]
  These three results contradict that there are no homogeneous
  solutions of \eqref{eq:lineq20} with growth rate $\epsilon$. It
  follows that for a given $\epsilon > 0$ we can choose $T_0$
  sufficiently large, so that the slow decay property
  \eqref{eq:vqslow} can be propagated to later times (if the
  $L^2$-norm is not too small). 

  To obtain the estimate \eqref{eq:vqdom20}, choose a $\lambda >
  0$. Then, for a given $\epsilon > 0$ we can argue as above, to show
  that for large enough  $T_0$ we must have either $\Vert
  v_q(\cdot, T_1)\Vert_{L^2} < e^{-2\kappa_0 T_1}$, or 
  \[ \Vert v_q(\cdot, T_1+1)\Vert_{L^2} \leq e^\epsilon \Vert
    v_q(\cdot, T_1)\Vert_{L^2}, \]
  since otherwise this growth could be propagated for all later times,
  contradicting that we have $v_q\to 0$. Consider a
  sequence of such flows as above, corresponding to larger and larger
  $T_0$. Normalizing, and taking the limit as above, after
  translating in time by $T^i_1$, we now obtain
  $v^\infty_q$, satisfying
  \[ \liminf_{\tau\to 0} \Vert v^\infty_q(\cdot, \tau)\Vert_{L^2} &\leq
    e^\epsilon \Vert v^\infty_q(\cdot, 1)\Vert_{L^2}, \\
    \Vert v^\infty_q(\cdot, 2)\Vert_{L^2} &\leq e^\epsilon \Vert
    v^\infty_q(\cdot, 1)\Vert_{L^2}. \]
  If this holds for sufficiently small $\epsilon$, depending on
  $\lambda$, then we have
  \[ \Vert \Pi_{\mathfrak{h}_1} v^\infty_q(\cdot, \tau)\Vert_{L^2}
    \geq (1-\lambda/2) \Vert v^\infty_q(\cdot, \tau)\Vert_{L^2} \]
  for $\tau\in [1,2]$, since in the limit as $\epsilon\to 0$,
  $v^\infty_q$ would be homogeneous of degree 0, but also orthogonal
  to $y_j$ for $j\leq \ell$, and $\mathfrak{h}_1$ consists of exactly
  such functions. The estimate \eqref{eq:302} implies that the for
  sufficiently large $i$ we have \eqref{eq:vqdom20} for $\tau\in
  [T_1+1, T_1+2]$. The estimate also follows for larger $\tau$ because
  \eqref{eq:vqslow} implies the same slow decay bound for larger
  $\tau$ as well. 
\end{proof}

We can now prove Proposition~\ref{prop:uqdecay}.
\begin{proof}[Proof of Proposition~\ref{prop:uqdecay}]
  Let $\lambda > 0$ be small, to be chosen, let $\epsilon > 0$ be
  the constant obtained from   Proposition~\ref{prop:uqdominant2}, and
  suppose that $T_0$ is sufficiently large to apply that
  result. 

  Let us suppose first that for all $T_1\geq
  T_0$ either the condition \eqref{eq:vqslow} fails, or we have $\Vert
  v_q(\cdot, T_1)\Vert_{L^2} < e^{-2\kappa_0 T_1}$. It follows from
  this, together with the estimate $\Vert v_q\Vert\leq
  C\tau^{-19/10}$, that for all $\tau \geq 1$ we have
  \[ \Vert v_q (\cdot, T_0 + \tau)\Vert_{L^2} \leq C
  e^{-\epsilon \tau} T_0^{-19/10} + e^{-2\kappa_0 (T_0+\tau)}, \]
  so $v_q$ decays exponentially fast. To relate this back to $u_q$, we
  can take the $L^2$ inner product of \eqref{eq:vtildeeq10} with
  $\xi_\tau$ to find that $\xi_\tau$ also decays exponentially fast, and as
  a consequence $u_q$ decays exponentially fast as well. In
  particular, if we choose $\kappa_0 < \epsilon$, then
  for some $T_1 > T_0$, depending on $T_0, \epsilon,
  \kappa_0$, we will have $\Vert u_q\Vert_{L^2} \leq e^{-\kappa_0 \tau}$ for
  $\tau > T_1$. 

  We will therefore assume that actually both hypotheses of
  Proposition~\ref{prop:uqdominant2} hold at some time $T_1\geq
  T_0$. Then \eqref{eq:vqslow2} holds, which ensures that the
  hypotheses continue to hold at times $T_0+m$ for $m > 0$, and the
  conclusion can be iterated for all $m$.  If we choose $\kappa_0 < \epsilon$, then
  we can find a subsequent time $T_2$ which satisfies
  \[ \label{eq:vqslow21} \Vert
  v_q\Vert_{L^2} \geq e^{-\kappa_0 \tau}, \text{ for all } \tau\geq
  T_2. \]
  We will derive a contradiction from this, using the
  decay $\Vert v_q\Vert\leq C\tau^{-19/10}$, and by analyzing
  the equation satisfied by $\tilde{v}_q$.

  The conclusion \eqref{eq:vqdom20} of Proposition~\ref{prop:uqdominant2}
  says that for all $\tau \geq T_2 +1$, the function $v_q$ is
  dominated by its $z_\alpha$ and $y_j$ components, for $j >
  \ell$. Let us consider the evolution of these components, and for
  simplicity we take the $y_j$ component for some $j > \ell$ since the
  $z_\alpha$ components are dealt with in the same way.

  In the calculation,  as in the proof of
  Proposition~\ref{prop:xitest30},
  we will view the  $\tilde{v}_q$ as functions on the cylinder $\mathcal{C}_{n,k}$, and
  we will take integrals using the measure $d\mu = e^{-|x|^2/4}
  d\mathcal{H}^n$.
  Recall that we are choosing $\kappa_0$ very small, so that
  on the $\kappa_0^{1/4}\tau^{1/2}$-ball the flow $M_\tau$ has good
  graphicality over $\mathcal{C}_{n,k}$.

  Using \eqref{eq:vtildeeq10} we have
  \[\label{eq:dtyjvtilde} \partial_\tau \int_{\mathcal{C}} y_j \tilde{v}_q d\mu =
    \int_{\mathcal{C}} y_j \left(\left(\Delta - \frac{1}{2} x\cdot \nabla +
      |A|^2\right) \tilde{v}_q + \mathcal{Q}_v(\tilde{v}_q) +
    \mathcal{R}_{\xi_\tau}(\tilde{v}) +  \chi\xi_\tau
    + \mathcal{E}\right)\, d\mu.
\]
It is worth noting that here, as before, $\Delta, x\cdot \nabla, |A|$ still denote
the corresponding objects on $\mathcal{M}$, and they are related to
the corresponding linear operator on $\mathcal{C}$ by
\eqref{eq:41}.  We claim that for any $\kappa > 0$, if $\tau$ is sufficiently large,
and $\lambda > 0$ above was chosen sufficiently small, then we have the estimate
\[ \label{eq:dtyv} \partial_\tau \int_{\mathcal{C}} y_j \tilde{v}_q\, d\mu \geq
  -\frac{\kappa}{\tau} \Vert \tilde{v}_q(\cdot, \tau)\Vert_{L^2} -
  Ce^{-2\kappa_0\tau}. \]
To see this, we consider each term in \eqref{eq:dtyjvtilde}. Since
$(\mathcal{L}_{\mathcal{C}}-\frac{1}{2})\tilde{v}_q$ is orthogonal to
$y_j$, from \eqref{eq:41} we get
\[ \label{eq:601} \int_{\mathcal{C}} y_j \left(\Delta - \frac{1}{2} x\cdot \nabla +
      |A|^2\right) \tilde{v}_q \, d\mu &\geq
    -\frac{1}{4\tau}\int_{\mathcal{C}} y_j \left(\sum_{i=1}^\ell (y_i^2-2)
      \right) (2\Delta_z\tilde{v}_q + \tilde{v}_q) \, d\mu \\
    &\qquad - C\tau^{-19/10} \Vert v_q(\cdot,
    \tau)\Vert_{L^2} - C e^{-2\kappa_0\tau} \]
  We will deal with the term involving the sum below. The other two
  terms satisfy \eqref{eq:dtyv}.

  Next consider the term involving $\mathcal{Q}_v$ in
  \eqref{eq:dtyjvtilde}. 
  The coefficients of the linear map $\mathcal{Q}_v$ are all
of order at most $\tau^{-19/10}$ by \eqref{eq:407}, and so using also
\eqref{eq:401}, we have
\[ \left|\int_{\mathcal{C}} y_j
  \mathcal{Q}_u(\tilde{v}_q)\, d\mu\right| \leq C\tau^{-19/10} \Vert
\tilde{v}_q\Vert_{L^2} + e^{-2\kappa_0\tau}. \]
Using \eqref{eq:Rxbound} and \eqref{eq:xitaubound5}, the term involving
$\mathcal{R}_{\xi_\tau}(\tilde{v})$ in \eqref{eq:dtyjvtilde} is bounded by
\[ \left| \int_{\mathcal{C}} y_j \mathcal{R}_{\xi_\tau}(\tilde{v}) \,
    d\mu \right| \leq C \tau^{-9/10} (\tau^{-1} \Vert v_q(\cdot,
  \tau)\Vert_{L^2} + e^{-2\kappa_0 \tau}) \leq C\tau^{-19/10} \Vert
  v_q(\cdot, \tau)\Vert_{L^2} + Ce^{-2\kappa_0 \tau}.  \]

The function $\xi_\tau$ is orthogonal to $y_j$ since $j > \ell$
and $\xi_\tau$ is in the span of $y_i$ for $i\leq \ell$. So we have
\[ \int_{\mathcal{C}} y_j \chi\xi_\tau \, d\mu \geq - C \Vert
  (1-\chi)\xi_\tau\Vert_{L^2}. \]
Using that $(1-\chi)$ is supported outside of the $\frac{1}{2}
\kappa_0^{1/4}\tau^{1/2}$-ball it follows that $\Vert
(1-\chi)\xi_\tau\Vert_{L^2} \leq e^{-2\kappa_0 \tau}$ for large
$\tau$. The integral of $y_j \mathcal{E}$ also satisfies the same
estimate.

Finally, we deal with the term involving the sum in \eqref{eq:601}. 
Consider one of the terms in the sum, given by 
\[ \label{eq:42}\frac{1}{\tau} \int_{\mathcal{C}} y_j (y_i^2-2)
  (2\Delta_z \tilde{v}_q + 
    \tilde{v}_q)\, d\mu, \]
where $i\leq \ell$ and $j > \ell$.  Note that if we integrate the
$\Delta_z$ term by parts, it vanishes.
If we put $f=z_\alpha$ or $f = y_k$, then
\[\label{eq:405} \int_{\mathcal{C}}
  y_j(y_i^2-2) f \, e^{-|x|^2/4}\, d\mathcal{H}^n = 0, \]
using again that $i\not=j$, so these components of $\tilde{v}_q$ do
not contribute. Using \eqref{eq:vqdom20} the other components of
$\tilde{v}_q$ are much smaller, and we obtain
\[ \label{eq:406} \left|\frac{1}{\tau} \int_{\mathcal{C}} y_j (y_i^2-2)
  (2\Delta_z \tilde{v}_q + 
    \tilde{v}_q)\, d\mu\right| \leq \frac{C}{\tau}\lambda \Vert v_q
\Vert_{L^2}. \]
We remark that in \eqref{eq:405} if we had $i=j$ and $f=y_i$, then the
integral would be nonzero. This is why it is important that we are
working with $\tilde{v}_q$, obtained by removing the $y_i$ components of
$u_q$.

Combining all these estimates, we get \eqref{eq:dtyv} if $\tau$ is
sufficiently large and $\lambda$ is chosen small (depending on
$\kappa$). The same argument leading to \eqref{eq:406}
also works if we replace $y_j$ with $z_\alpha$. For this, instead of
\eqref{eq:42}, we have the integral
\[ \frac{1}{\tau} \int_{\mathcal{C}} z_\alpha (y_i^2-2)
  (2\Delta_z \tilde{v}_q + 
  \tilde{v}_q)\, d\mu, \]
and integrating the $\Delta_z$ term by parts, using that $z_\alpha$ is
an eigenfunction, we need only consider
\[ \frac{1}{\tau} \int_{\mathcal{C}} z_\alpha (y_i^2-2)
  \tilde{v}_q\, d\mu. \]
If we replace $y_j$ in \eqref{eq:405} by $z_\alpha$, then the integral
still vanishes, so we get the estimate \eqref{eq:406} with $y_j$
replaced by $z_\alpha$.

Using \eqref{eq:vqdom20} again, it
follows that if we write $g(\tau) = \Vert \Pi_{\mathfrak{h}_1}
v_{q} \Vert_{L^2}$, then
\[ g'(\tau) &\geq -C\kappa \tau^{-1} g(\tau) - C e^{-2\kappa_0\tau}. \]
  
Integrating this from $\tau_0$ to $\tau$, for $\tau_0 > T_2$, we get
\[ g(\tau) &\geq \tau_0^{C\kappa} g(\tau_0) \tau^{-C\kappa} - C
  \tau^{-C\kappa} \int_{\tau_0}^\tau s^{C\kappa} e^{-2\kappa_0 s}\, ds \\
  &\geq \frac{1}{2} \tau_0^{C\kappa} \tau^{-C\kappa} e^{-\kappa_0
    \tau_0} - C'  \tau^{-C\kappa} \tau_0^{C\kappa} e^{-2\kappa_0
    \tau_0}, \]
where we also used \eqref{eq:vqslow21} to bound $g(\tau_0) \geq \frac{1}{2}
e^{-\kappa_0 \tau_0}$. If we choose $\tau_0$ sufficiently large, then
this implies
\[ g(\tau) \geq \frac{1}{4} \tau_0^{C\kappa} e^{-\kappa_0
    \tau_0 } \tau^{-C\kappa}, \]
for all $\tau > \tau_0$. If $\kappa$ is chosen sufficiently small, then
this contradicts that $g(\tau) \leq C\tau^{-19/10}$. 
\end{proof}

We will also need the following result expressing that under taking
limits, cylindrical singularities can not become less degenerate if
they remain the same type of cylinder in the limit. And if the type of
cylinder changes in the limit, then the spherical factor must decrease
in dimension. 
\begin{prop}\label{prop:singlimits}
  Let $L_t^i$ be a sequence of mean convex mean curvature flows for $t
  \geq 0$, converging in the sense of Brakke flows to a limit flow
  $L_t^\infty$. Suppose that for all $i$, $L_t^i$ has an
  $\ell_i$-degenerate $(n,k_i)$-cylindrical singularity at $(x_i,
  t_i)$, and that $(x_i, t_i)\to (x_\infty, t_\infty)$, with
  $t_\infty > 0$. Then $L_t^\infty$ has an $(n, k_\infty)$-cylindrical
  singularity at $(x_\infty, t_\infty)$, where
  \[ k_\infty \geq \limsup_{i\to\infty} k_i. \]
  Moreover, if after replacing the sequence with a
  subsequence we have $k_i = k_\infty$, then
  $L^\infty_t$ has an $\ell_\infty$-degenerate
  $(n,k_\infty)$-cylindrical singularity at $(x_\infty, t_\infty)$,
  where
  \[ \ell_\infty \leq \liminf_{i\to\infty} \ell_i. \]
\end{prop}
\begin{proof}
  The claim about $k_\infty$ follows from considering 
 the entropy $\lambda(\mathcal{C}_{n,k})$ of the different cylindrical
 shrinkers. We have $\lambda(\mathcal{C}_{n,k}) =
 \lambda(S^{n-k}(\sqrt{2(n-k)}))$, and by the calculation of
 Stone~\cite{Stone94} we have
 \[ \lambda(S^n) < \lambda(S^{n-1}) < \ldots < \lambda(S^1). \]
 Using Huisken's monotonicity formula~\cite{Hui90}
 the entropies of the tangent flows at $(x_i, t_i)$ cannot
 decrease in the limit. It follows that a limit of $(n,k)$-cylindrical
 singularities must be an $(n,k')$-cylindrical singularity with
 $k'\geq k$.

 Suppose now that the $(x_i, t_i)$ are all $\ell$-degenerate
 $(n,k)$-cylindrical singularities, and they converge to an
 $\ell'$-degenerate $(n,k)$-cylindrical singularity $(x_\infty,
 t_\infty)$. We want to show that $\ell' \leq \ell$. To see this, note
 first that by translating and rotating the flows, we can assume that
 they are all $\mathcal{C}_{n,k}$-singularities at $(0,0)$. Let us
 write $M^i_\tau$ for the corresponding rescaled flows,
 $M^\infty_\tau$ for the limit as $i\to \infty$. We are assuming that
 $M^\infty_\tau$ has an $\ell'$-degenerate
 $\mathcal{C}_{n,k}$-singularity at infinity, and we want to show that
 $\ell \geq \ell'$. We can assume that $\ell' > 0$. Applying the
 normal form result, Theorem~\ref{thm:SXnormal} we
 have that for some $T_0$, and all $\tau \geq T_0$, $M^\infty_\tau$ is the graph
 of $u(y,z,\tau)$ over $\mathcal{C}_{n,k}$ on $B_{\tau^\beta}$ for
 some $\beta > 0$, such that
 \[ \label{eq:51} \left\Vert u(\cdot, \tau) - \frac{\sqrt{2(n-k)}}{4\tau}
     \sum_{i=1}^{\ell'} (y_i^2-2) \right\Vert_{L^2(\tau^\beta)} <
   \frac{1}{2} \tau^{-3/2}. \]
 From the convergence $M^i_\tau \to M^\infty_\tau$ it follows that for
 any $L_0 > 0$, we will have \eqref{eq:51} for the graphical function
 of $M^i_\tau$ for sufficiently large $i$, for all $\tau\in [T_0, T_0
 + L_0]$.  We can now apply Proposition~\ref{prop:L2normalform}. The
 conclusion is that for sufficiently large $i$ the flow $M^i_\tau$
 cannot have an $\ell$-degenerate $\mathcal{C}_{n,k}$-singularity at
 infinity, with $\ell < \ell'$, as required. 
\end{proof}

Let us introduce the following notion, in order to summarize the
results of this section.
\begin{definition}\label{defn:singscale}
  Suppose that $M_\tau$ is a rescaled flow, with an $\ell$-degenerate
  $\mathcal{C}_{n,k}$-singularity at infinity. Let $\kappa_0, \epsilon_0$ be as in
  Proposition~\ref{prop:uqdecay}. We say that $M_\tau$ has \emph{good
    asymptotics at time $T_1$}, if the hypotheses and conclusion of
  Proposition~\ref{prop:uqdecay} are satisfied for $\tau \geq T_1$. In
  other words, for all $\tau \geq T_1$,
  $M_\tau$ can be written as the graph of $u(\cdot, \tau)$
      over $\mathcal{M}^{n,k,\ell}_\tau$ on
      $B_{10\tau^{1/2}}$ and the
      function $u$ satisfies:
      \begin{itemize}
        \item    $\Vert u(\cdot, \tau) \Vert_{C^4} < \epsilon_0$ on the
          $10\tau^{1/2}$-ball, while $\Vert u(\cdot, \tau)\Vert_{C^4}
          < \tau^{-3/2}$ on the $\tau^{\gamma}$-ball for small $\gamma
          > 0$,
          \item $\Vert u(\cdot,
            \tau)\Vert_{L^2} < \tau^{-19/10}$,
         \item $\Vert \partial_{y_q} u(\cdot,
        \tau)\Vert_{L^2(B_{10\tau^{1/2}})} \leq e^{-\kappa_0 \tau}$
      for the degenerate directions $q=\ell+1, \ldots, k$.
     \end{itemize}

  Similarly, if $L_t$ is a mean curvature flow developing an
  $\ell$-degenerate $(n,k)$-cylindrical singularity at a point $(x_0,
  t_0)$, then we say that the singularity is \emph{$\ell$-degenerate
    at scale $r > 0$}, if after a suitable rotation (so that the
  tangent flow is $\mathcal{C}_{n,k}$), the rescaled flow centered at
  $(x_0, t_0)$, defined by
  \[ M_\tau = e^{\tau/2} (M_{t_0 - e^{-\tau}} - x_0), \]
  has good asymptotics at time $-2\ln r$. 
\end{definition}

The main results of this section can be summarized by the following.
\begin{prop}\label{prop:sec2main}
  For any $T_1 > 0$ there exist $L_1(T_1), T_2(T_1) > 0$ (also
  depending on the dimension and area ratio bounds), with the
  following property. Suppose that $M_\tau$ has an $\ell$-degenerate
  $\mathcal{C}_{n,k}$-singularity at infinity, and for the $\beta$ in
  Proposition~\ref{prop:tbetagraph11} we write $M_\tau$
  as the graph of $u(x,\tau)$ over $\mathcal{C}_{n,k}$ on the balls
  $B_{\tau^\beta}$ for $\tau\in [T_1, T_1+L_1]$. Suppose that for
  this range of $\tau$, $u$
  satisfies the $L^2$ bound $E(\tau) < \tau^{-3/2}$ for $E(\tau)$ as
  in \eqref{eq:L2decay10}. Then up to a rotation of size
  $\Psi(T_1^{-1})$, $M_\tau$ has good asymptotics at
  $\tau=T_2$.
\end{prop}

Note that by definition if $M_\tau$ has good asymptotics
at time $T_0$, then this holds for all larger times as well, and so if a
singularity is $\ell$-degenerate at scale $r_0 > 0$, then it is also
$\ell$-degenerate at all scales $r < r_0$. It is also worth noting
that this notion does not scale if we parabolically rescale our flow,
or equivalently,  if $M_\tau$ has good asymptotics at time $T_0$, then it does not
follow that the translated flow $M_{\tau+T_0}$ has good asymptotics at time $0$,
since the normal form does not behave well under translations in
$\tau$. On the other hand, using Proposition~\ref{prop:sec2main},
if a singularity is $\ell$-degenerate at
scale $r$, then all nearby $\ell$-degenerate singularities of sufficiently
close flows are also $\ell$-degenerate at some scale $r' > 0$,
depending on $r$. This is the content of the following.

\begin{prop}\label{prop:nearbydegen}
For sufficiently small $r > 0$ there exist $r', \delta > 0$, depending on $r$, with the
following property. Suppose that $L_t$ has an $\ell$-degenerate
$(n,k)$-cylindrical singularity at $(x_0, t_0)$, which is
$\ell$-degenerate at scale $r$. Suppose in addition that $L'_t$ is
another mean curvature flow, which is $\delta$-close to $L_t$ for a
metrization of Brakke convergence, and suppose that $L'_t$ has an
$\ell$-degenerate $(n,k)$-cylindrical singularity at a point $(x_0',
t_0')$ with $|x_0'-x_0| + |t_0' - t_0|^{1/2} < \delta$. Then the
singularity at $(x_0', t_0')$ is $\ell$-degenerate at scale $r'> 0$.
\end{prop}
\begin{proof}
  Up to translation and rotation, we can assume that $L_t$ has an
  $\ell$-degenerate $\mathcal{C}_{n,k}$-singularity at $(0,0)$, which
  is $\ell$-degenerate at scale $r$. Let $M_\tau$ be the corresponding
  rescaled flow, and let $T_1=-2\ln r$. By definition, for all
  $\tau\geq T_1$ the graphicality function $v$ of $M_\tau$ over the
  model flow $\mathcal{M}^{n,k,\ell}_\tau$ on the $10\tau^{1/2}$-balls
  satisfies $\Vert v\Vert_{L^2} < \tau^{-19/10}$. Using the
  asymptotics of $\mathcal{M}^{n,k,\ell}_\tau$, we can assume that
  $T_1$ is large enough so that on the smaller
  balls $B_{\tau^\beta}$ the graphicality function $u$ of $M_\tau$
  over $\mathcal{C}_{n,k}$ satisfies the $L^2$-bound $E(\tau) <
  \frac{1}{10} \tau^{-3/2}$ for $E$ defined in \eqref{eq:L2decay10}, for
  all $\tau \geq T_1$.

  Suppose now that $L_t'$ is another flow that is $\delta$-close to
  $L_t$, and has an $\ell$-degenerate $(n,k)$-cylindrical singularity
  at a point $(x_0', t_0')$ with $|x_0'| + |t_0'|^{1/2} <
  \delta$. Translating $L_t'$ we obtain a flow $L_t''$ that has an
  $\ell$-degenerate $(n,k)$-cylindrical singularity at $(0,0)$, and is
  still $\Psi(\delta)$-close to $L_t$ (where $\Psi(\delta)\to 0$ as
  $\delta\to 0$). Let $M''_\tau$ be the corresponding rescaled
  flow. If $\delta$ is sufficiently small (depending on $T_1$), then
  using the graphicality property of $M_\tau$ over $\mathcal{C}_{n,k}$
  we will have that $M''_\tau$ is the graph of a function $w$ over
  $\mathcal{C}_{n,k}$ on $B_{\tau^\beta}$, satisfying the $L^2$-bound
  $E(\tau) < \tau^{-3/2}$ for $\tau\in [T_1, T_1 + L_1]$, where $L_1$
  is determined by $T_1$ in Proposition~\ref{prop:sec2main}. Then
  Proposition~\ref{prop:sec2main} implies that up to a small rotation
  $M''_\tau$ has an $\ell$-degenerate $\mathcal{C}_{n,k}$-singularity
  at infinity, with good asymptotics at $\tau = T_2$. We can let $r' =
  e^{-T_2/2}$, and then the singularity of $L'_t$ at $(x_0, t_0)$ will
  be $\ell$-degenerate at scale $r'$ as required. 
\end{proof}

\section{Perturbing away degenerate singularities}\label{sec:perturb}
This section contains the heart of our perturbation argument,
which will be used in the next section to prove our main
results. Here we will focus on a mean convex flow $L_t$, which in practice will
already be a small perturbation of our original flow, translated in
time and space, and rotated, so that it has an
$\ell$-degenerate $\mathcal{C}_{n,k}$-singularity at $(0,0)$. Note
that we will not apply scalings to the original flow, only translations and rotations. Let us
assume furthermore that this singularity is $\ell$-degenerate at scale
$r_0 > 0$, as in Definition~\ref{defn:singscale}. Recall that then
the singularity is also $\ell$-degenerate at any scale $r < r_0$, and
so we will at times replace $r_0$ by a smaller constant. Let us write
$T_0 = -\ln r_0^2$, so the rescaled flow $M_\tau$ corresponding to
$L_t$ has good asymptotics at all times $\tau \geq T_0$ in the sense of
Definition~\ref{defn:singscale}.
In particular, $M_\tau$ satisfies the
following:
\begin{itemize}
  \item[(a)] For all $\tau \geq T_0$,  $M_\tau$ is the graph of
    $u(x,\tau)$ over the model $\mathcal{M}^{n,k,\ell}_\tau$ on the
    ball $B_{10\tau^{1/2}}$, with $\Vert u\Vert_{C^4}\leq \epsilon_0$, 
  \item[(b)] For any $q > \ell$ we have $\Vert \partial_{y_q} u(\cdot,
    \tau)\Vert_{L^2} <  e^{-\kappa_0 \tau}$.
\end{itemize}

We have the following, using the pseudolocality property
Proposition~\ref{prop:modelpseudo} for
$\mathcal{M}^{n,k,\ell}_\tau$. 
\begin{prop}\label{prop:pseudo10}
  Given $\epsilon > 0$ there are $\delta_1(\epsilon), C_1(\epsilon),
  T_0(\epsilon) > 0$ such that if $N_\tau$ is a RMCF such that
  $N_{\tau_0}$ is $\delta_1$-graphical over $M_{\tau_0}$ for some
  $\tau_0 > T_0$, on the ball $B_R$ for some $R > C_1$, then for all
  $\tau\in [\tau_0, \tau_0+10]$, and with
  \[  R(\tau) = \min\{e^{(\tau-\tau_0)/2}(R-C_1), 10\tau^{1/2}\}, \]
  $N_\tau$ is $\epsilon$-graphical over $M_\tau$ on the ball $B_{R(\tau)}$.
\end{prop}

Suppose that 
$N_{\tau_0}$ is $\epsilon_0$-graphical over  $M_{\tau_0}$ on
$B_{10\tau_0^{1/2}}$. Then $N'_{\tau}$ is still $\epsilon_1$-graphical
over $M_\tau$ on $B_{10\tau^{1/2}}$ for $\tau\in [\tau_0 +1,
\tau_0+10]$. Let us write $N_\tau$ as the graph of $v$ over
$M_\tau$. As before, this satisfies the equation
\[ \partial_\tau v(x,\tau) = \mathcal{L}_{M_\tau} v(x,\tau) +
  Q_{M_\tau}(v, \nabla v, \nabla^2 v). \]
We let $\tilde{v}(x,\tau) = \chi(2|x|/ (10\tau^{1/2}))$ with the
cutoff function $\chi$ as before. Then $\tilde{v}$ satisfies an
equation similar to \eqref{eq:utildeeq}, with error supported outside
of the $10\tau^{1/2}/2$-ball. Note that by
Proposition~\ref{prop:ldegendecay} we
have uniform control of the geometry of $M_\tau$ on this ball, and
therefore we can apply Corollary~\ref{cor:logSobolev}, together with
parabolic Schauder estimates,  to obtain the
following. 
 
\begin{prop}\label{prop:nonconc30} 
 Let $\tau^{1/10} \leq R(\tau) \leq
  10\tau^{1/2}$ denote an increasing function such that for all $\tau\geq T_0$ we
  have $e^{1/2}(R(\tau)-C_1) > R(\tau+1)$ for the $C_1(\epsilon_0)$ in
  Proposition~\ref{prop:pseudo10}.
  Once $\tau_0$ is sufficiently large, for $\tau\in [\tau_0+1,
  \tau_0+10]$ we have
  \[ |v(x, \tau)| + |\nabla v(x, \tau)| + |\nabla^2 v(x, \tau)| \leq C
    e^{|x|^2/8p} ( \Vert v(\cdot, \tau_0)\Vert_{L^2(R(\tau_0))} + e^{-R(\tau_0)^2/5}), \]
  on $B_{R(\tau)}$.
\end{prop}

Similarly to Proposition~\ref{prop:uqdominant2} we have the
following three annulus lemma. Below we will need the flexibility of
working on balls with different radii, so we state the result in that
generality. 
\begin{prop}\label{prop:3ann50}
  Given $\lambda\in(0,1/2)$ there exists $T_0, \epsilon_0 > 0$ such that if
  $T > T_0$, then we have the following. Let $R(\tau)$ be a function  as
  in Proposition~\ref{prop:nonconc30}. Suppose that for $i=0,1,2$ we
  have $|\nabla^j v(x,
  \tau)|\leq \delta < \epsilon_0$, on $B_{R(\tau)}$ for $\tau\in
  [T, T+2]$, and we have
  \[ \Vert v(\cdot, T+1)\Vert_{L^2(B_{R(T+1)})} \geq
    e^{\frac{1}{2}-\lambda} \Vert v(\cdot, T)\Vert_{L^2(B_{R(T)})}, \]
  as well as $\Vert v(\cdot, T)\Vert_{L^2} \geq \delta e^{-R(\tau)^2/9}$. Then
  \[ \Vert v(\cdot, T+2)\Vert_{L^2(B_{R(T+2)})} \geq e^{\frac{1}{2}-\lambda}
    \Vert v(\cdot, T+1)\Vert_{L^2(B_{R(T+1)})}. \]
\end{prop}
\begin{proof}
We argue by contradiction, similarly to Proposition~\ref{prop:uqdominant2}.
Suppose that we have a sequence $N^i_\tau$, and $T_i \to \infty$,
such that the corresponding graphicality functions $v^i$ on satisfy
\[ \Vert v^i(\cdot, T_i)\Vert_{L^2} &\leq e^{-(\frac{1}{2}-\lambda)}
  \Vert v^i(\cdot, T_i+1)\Vert_{L^2}, \\
  \Vert v^i(\cdot, T_i+2)\Vert_{L^2} &\leq e^{(\frac{1}{2}-\lambda)}
  \Vert v^i(\cdot, T_i+1)\Vert_{L^2}. \]
Each $L^2$-norm here is taken on the appropriate ball $B_{R(\tau)}$. 
We define $\mathbf{d}_i = \Vert v^i(\cdot, T_i+1)\Vert_{L^2}$. As
$i\to\infty$, we can extract a limit $v^\infty$ of the time translated
and normalized functions
$\mathbf{d}_i^{-1} v^i(\cdot, T_i+\tau)$. The limit $v^\infty$ is a
solution of the linearized equation on $\mathcal{C}_{n,k}$, and the
convergence is smooth on compact subsets of $(0,2] \times
\mathbb{R}^{n+1}$. 
Similarly to before, Proposition~\ref{prop:nonconc30} implies that the
limit $v^\infty$ would be nonzero and homogeneous of degree $\frac{1}{2}-\lambda$,
but no such solution of the linearized equation exists. 
\end{proof}
Although we did not make this explicit in the proof, note that the
constants we obtain are independent of the flow $M_\tau$, as long as
it is a rescaled mean curvature flow with an $\ell$-degenerate
$\mathcal{C}_{n,k}$-singularity, that has good asymptotics at time $T_0$.

The following is the main ingredient that we use to perturb away
degenerate singularities.  Here we will denote by
$\Vert u\Vert_{R}$ the $L^2$ norm on the $R$-ball. 
\begin{prop}\label{prop:perturb20}
  Let $c_1\in (0,1)$. Choose $\kappa_1 = \kappa_0/4$ for the $\kappa_0$ in
  Proposition~\ref{prop:uqdecay}, and suppose that $\epsilon_0, T_0$
  are chosen according to Proposition~\ref{prop:3ann50}, so that the
  result holds for $\lambda=\kappa_1, 2\kappa_1, 3\kappa_1, 1/4$. For any $T_1 \geq
  T_0$ there exists $r_0=r_0(c_1, T_1)$ with the following property. 
  Suppose that the flow $N_\tau$ satisfies the following conditions
  for some $\gamma > 0$: 
  \begin{itemize}
     \item[(i)] For $\tau \geq T_1$, as long as $c_1^{-1}\gamma T_1 e^{\tau-T_1} < \epsilon_0$,
       $N_\tau$ is the graph of $v$ over $M_\tau$ on the
       $10\tau^{1/2}$-balls, and $\Vert v\Vert_{C^4} \leq
       c_1^{-1} T_1\gamma e^{\tau-T_1}$. Suppose that $\gamma$ is small
       enough so that this holds in
       particular for $\tau\in [T_1, 3T_1]$, i.e. $\gamma < \epsilon_0
       c_1 T_1^{-1} e^{-2T_1}$. 
     \item[(ii)] $\Vert v(\cdot, T_1)\Vert_{10T_1^{1/2}} \geq c_1 \gamma$.
     \item[(iii)] $\Vert v(\cdot, T_1+ 1)\Vert_{10(T_1+1)^{1/2}} \geq
         e^{\frac{1}{2}-\kappa_1} \Vert v(\cdot, T_1)\Vert_{10T_1^{1/2}}$.
   \end{itemize}
  Let $x_0\in \mathbb{R}^{n+1}$ be a vector with no $z_\alpha$ or
  $y_i$ component for $i\leq \ell$ (i.e. it is only nonzero in the
  degenerate directions) such that $|x_0| < r_0 \gamma^{1-\kappa_1}$.
  Then the translated flow $N'_\tau = N_\tau + e^{(\tau-T_1)/2}x_0$ does not
  have an $(n,k)$-cylindrical singularity at infinity. 
\end{prop}
\begin{proof}
  The proof is somewhat similar to \cite[Proposition 13]{Sz26}, but it
  is more involved because we need to iterate the three annulus lemma in
  several stages before we reach a time at which the distance from our
  translated flow $N'_\tau$ to $M_\tau$ dominates the $O(1/\tau)$ distance from
  $M_\tau$ to $\mathcal{C}_{n,k}$. We first give an outline of the
  argument in four steps:
  \begin{itemize}[itemsep=10pt]
    \item[(1)]   To start, 
  we will use Proposition~\ref{prop:3ann50} on $10\tau^{1/2}$-balls
  to propagate the growth of $v$ until we find a time $T_2$, where
  $\Vert v(\cdot, T_2+1)\Vert > C_2 e^{-\kappa_0 T_2}$ for some large
  $C_2$. This will ensure that $v$ dominates
  $\partial_{y_q} u$ for any degenerate direction $q > \ell$.
  Note that to leading order $N_\tau$ is the graph of $u+v$ over the model flow
  $\mathcal{M}^{n,k,\ell}_\tau$. We will find that  even if we
  replace $N_\tau$ with the translated flow $N'_\tau = N_\tau +
  e^{(\tau-T_2)/2} X_0$ for a vector $X_0$ that is only nonzero in the
  degenerate $y$-components and $|X_0| < r_0$ for a small $r_0 > 0$,
  then the corresponding graphicality
  function $w$ of $N'_\tau$ over $M_\tau$ still grows at some rate
  on the time interval $[T_2, T_2+1]$. Note that translating $N_\tau$
  by $X_0$ at time $T_2$ corresponds to translating it by $x_0 =
  e^{-(T_2-T_1)/2}X_0$ at time $T_1$, so we have $|x_0| < r_0
  e^{-(T_2-T_1)/2}$, and we will find that $e^{-(T_2-T_1)/2} \geq
  \gamma^{1-\kappa_1}$. In \cite{Sz26} the argument
  ended here, since in that case the distance from $M_\tau$ to
  $\mathcal{C}_{n,k}$ was exponentially small.
  \item[(2)] Next we need to show
  that   the distance from $N'_\tau$ to $M_\tau$ keeps growing. 
  We can continue to apply
  Proposition~\ref{prop:3ann50} to the translated flow $N'_\tau$, as long as we
  maintain sufficient graphicality over $M_\tau$ on the
  $9\tau^{1/2}$-balls. We will show that this will hold as long as
  we have $\Vert w\Vert \leq \tau^{-2}$ using the same argument as in 
  Proposition~\ref{prop:ldegendecay}. The basic idea to keep in mind
  is that if $\Vert w\Vert_{L^2} < \tau^{-\alpha}$, then we can expect
  good graphicality on the $\tau^\beta$-balls, for $\beta <
  \alpha/2$. In this way we will find a
  first time $T_3$ at which $\Vert w(\cdot T_3+1)\Vert >
  (T_3+1)^{-2}$.
  \item[(3)]   At this point the $L^2$-norm is no longer
  exponentially small, and so we can continue to use
  Proposition~\ref{prop:3ann50} on smaller balls. We continue
  propagating the growth on the balls of radius $\tau^{1/5}$, which
  we can do as long as $\Vert w\Vert_{L^2} < \tau^{-1/2}$. We find a time
  $T_4$, such that $\Vert w(\cdot, T_4+1)\Vert >
  (T_4+1)^{-1/2}$.
\item[(4)] The distance from $N'_\tau$ to $M_\tau$ now
  dominates the distance $O(T_4^{-1})$ between
  $M_\tau$ and $\mathcal{C}_{n,k}$. Using this we find that the
  $L^2$-distance $\mathbf{d}_{\mathcal{C}_{n,k}}(N'_\tau)$
  from $\mathcal{C}_{n,k}$ to $N'_\tau$ is also
  growing at some rate.  Corollary~\ref{cor:slowgrowth} implies that $N'_\tau$
  cannot have a $(n,k)$-cylindrical singularity (possibly a rotation
  of $\mathcal{C}_{n,k}$) at infinity. 
\end{itemize}

We will now give the detailed proof following these steps. 

\bigskip

 \noindent{\bf Step 1.} By our assumption, $N_\tau$ is
  $\epsilon_0$-graphical over $M_\tau$ on the $10\tau^{1/2}$-balls for
  $\tau\in [T_1, 3T_1]$. If $T_1$ is large enough, then from
  Proposition~\ref{prop:Claim1} we know
  that as long as $\Vert v(\cdot, \tau) \Vert_{L^2(B_{\tau^\beta})}<
  \tau^{-19/10}$, we have that $N_\tau$ remains $\epsilon_0$-graphical
  over $M_\tau$ on $B_{10\tau^{1/2}}$. In particular this holds as
  long as $\Vert v(\cdot, \tau)\Vert_{L^2(B_{10\tau^{1/2}})} \leq
  e^{-2\kappa_1\tau}$. We claim that we can apply
  Proposition~\ref{prop:3ann50} to obtain
  \[ \label{eq:g101} \Vert v(\cdot, T_1+k+1)\Vert_{R(T_1+k+1)} \geq
    e^{\frac{1}{2}-\kappa_1} \Vert v(\cdot,
    T_1+k)\Vert_{R(T_1+k)}, \]
  as long as
  \[  \Vert v(\cdot, T_1+k)\Vert_{R(T_1+k)} \leq e^{-2\kappa_1
      (T_1+k)}, \]
  where we set $R(\tau) = 10\tau^{1/2}$. We already have
  $\epsilon_0$-graphicality on these balls,  but we
  also need to ensure that $|\nabla^i v(x, \tau)|
  \leq \delta$ on $B_{10\tau^{1/2}}$ for $\delta > 0$ such that
  \[ \label{eq:g200} \Vert v(\cdot, \tau)\Vert_{R(\tau)} \geq \delta
    e^{-R(\tau)^2/9}. \]
  By assumption we can take $\delta = c_1^{-1}\gamma T_1 e^{\tau-T_1}$. 
  Then, as long as \eqref{eq:g101} holds, we will have
  \[ \Vert v(\cdot, T_1 +k+1) \Vert_{R(T_1+k+1)} \geq
    e^{(k+1)(\frac{1}{2}-\kappa_1)} \Vert v(\cdot, T_1)\Vert_{R(T_1)} \geq
    e^{(k+1)(\frac{1}{2}-\kappa_1)} c_1\gamma, \]
  while
  \[ \delta e^{-R(T_1+k+1)^2/9} < c_1^{-1} \gamma T_1 e^{k+1} e^{ - 10
      (T_1+k+1)},  \]
  which is much smaller than $\Vert v(\cdot, T_1 +k+1)
  \Vert_{R(T_1+k+1)}$ if $T_1$ is large (depending on $c_1$). It
  follows that we can keep applying Proposition~\ref{prop:3ann50}, and
  \eqref{eq:g101} holds, until we find a value $k=k_0$ for which
  $T_2=T_1+k_0$ satisfies both
  \[ \label{eq:vT2lower} \Vert v(\cdot, T_2+1)\Vert_{R(T_2+1)} \geq e^{-2\kappa_1
      (T_2+1)}, \]
  and
  \[ \Vert v(\cdot, T_2+1)\Vert_{R(T_2+1)} \geq
    e^{\frac{1}{2}-\kappa_1} \Vert v(\cdot, T_2)\Vert_{R(T_2)}. \]
  At the same time we have
  \[ e^{-2\kappa_1 (T_1+k_0)} \geq \Vert v(\cdot, T_2)\Vert_{R(T_2)} \geq
    e^{k_0(\frac{1}{2}-\kappa_1)} \Vert v(\cdot, T_1)\Vert_{R(T_1)}
\geq  e^{k_0(\frac{1}{2}-\kappa_1)} c_1\gamma. 
\]
Rearranging, this means
\[ e^{-\frac{k_0}{2}(1+2\kappa_1)} \geq c_1\gamma e^{2\kappa_1 T_1}
  \geq c_1\gamma, \]
and so
\[ e^{-k_0/2} \geq c_1 \gamma^{1-\kappa_1}, \]
where we used that $c_1 < 1$, and $(1+2\kappa_1)^{-1} < 1-\kappa_1$ if
$\kappa_1$ is small.

We now introduce the translated flow $N'_\tau = N_\tau +
e^{(\tau-T_2)/2} x_1$, where $|x_1| < r_1$ for a small $r_1 > 0$ to be
chosen, and $x_1$ is nonzero only in the degenerate
$y_i$-components. Note that we have $e^{(\tau-T_2)/2}x_1 =
e^{(\tau-T_1)/2} e^{(T_1-T_2)/2}x_1$, so in terms of the notation in
the statement of the Proposition, we have $x_0 = e^{(T_1-T_2)/2}x_1 =
e^{-k_0/2}x_1$. By the estimates above, the requirement for $x_0$ is $|x_0| < r_1
c_1\gamma^{1-\kappa_1}$. We will show that if $r_1$ is chosen small
enough, then the translated flow $N'_\tau$ has no $(n,k)$-cylindrical
singularity at infinity.

Let us focus on the time interval $\tau\in [T_2, T_2+1]$. On the
$9\tau^{1/2}$-ball we can still write $N'_\tau$ as the graph of a
function $\tilde{v}$ over $M_\tau$. We also write $M_\tau$ as the
graph of the function $u$ over $\mathcal{M}^{n,k,\ell}_\tau$ on these
balls. In order to estimate $\tilde{v}$ it is convenient to write 
$N_\tau$ as the graph of $V$ over $\mathcal{M}^{n,k,\ell}_\tau$. Then
to leading order we have that $V=u+v$. Since
$\mathcal{M}^{n,k,\ell}_\tau$ is invariant in the degenerate
$y$-directions, the translated surface $N'_\tau$ is the graph of the
translated function $V(x - e^{(\tau-T_2)/2}x_1, \tau)$ over
$\mathcal{M}^{n,k,\ell}_\tau$, and so to leading order $N'_\tau$ is
the graph of the function
\[ (u+v)(x - e^{(\tau-T_2)/2}x_1, \tau) - u(x,\tau) \]
over $M_\tau$. For simplicity let us work with a fixed $\tau$, and
write $X_1=e^{(\tau-T_2)/2}x_1$. Using the mean value theorem we can write
\[ \label{eq:u+v} (u+v)(x - X_1, \tau) - u(x,\tau) = v(x, \tau) - X_1\cdot \nabla (u+v)
  \left(x - f(x) X_1, \tau\right), \]
  where $|f(x)|\leq 1$. We want to estimate the $L^2$-norm of the
  second term. Since $X_1$ only has degenerate $y$-components, we can
  estimate $X_1\cdot \nabla u$ using our assumption that $\Vert
  \partial_{y_q} u\Vert_{L^2} \leq e^{-\kappa_0 \tau}$. Using
  Proposition~\ref{prop:uqnonconc} this implies that
  \[ |X_1 \cdot \nabla u (x - f(x) X_1, \tau)| \leq C |X_1| e^{-\kappa_0
      \tau} e^{|x-f(x) X_1|^2/ 8p}, \]
  for a fixed $p > 1$. For any $\delta > 0$ we have
  \[ |x-f(x)X_1|^2 \leq (1+\delta)|x|^2 + C_\delta |X_1|^2, \]
  so we can choose a suitable $\delta > 0$ so that for some $1 < p' <
  p$, and a suitable constant $C$, we have 
  \[ e^{|x-f(x) X_1|^2/ 8p} \leq e^{C|X_1|^2} e^{|x|^2/8p'}. \]
  If $r_1$, and therefore $|X_1|$, are sufficiently small, then this means that 
  \[ \label{eq:x1vest}  |X_1 \cdot \nabla u (x - f(x) X_1, \tau)| \leq
    C |X_1| e^{-\kappa_0
      \tau} e^{|x|^2/8p'}. \]
   We can similarly estimate the derivative of $v(x, \tau)$ in terms
   of $\Vert v(\cdot, \tau-1)\Vert_{R(\tau-1)}$ using
   Proposition~\ref{prop:logSobolev} and
   Proposition~\ref{prop:L2pointwise}, but because the $L^2$-norm of
   $v$ is growing by \eqref{eq:g101}, we actually obtain
   \[ |\nabla^i v(x, \tau)| \leq C (e^{|x|^2/8p} \Vert v(\cdot,
     \tau)\Vert_{R(\tau)} + e^{-R(\tau)^2/9}), \]
   on the $R(\tau)$-ball. 
   Similarly to the argument for \eqref{eq:x1vest}, we get
   \[ |X_1 \cdot \nabla v(x-f(x)X_1, \tau)| \leq C |X_1| (
     e^{|x|^2/8p'} \Vert v(\cdot, \tau)\Vert_{R(\tau)} +
     e^{-10\tau}). \]
   Using this and \eqref{eq:x1vest} in \eqref{eq:u+v} we get
   \[ \Vert (\tilde{v} - v)(\cdot, \tau) \Vert_{9\tau^{1/2}} \leq
     C|X_1| (\Vert v(\cdot, \tau)\Vert_{R(\tau)} + e^{-\kappa_0
       \tau}). \]
   Note that by \eqref{eq:vT2lower} and the choice
   $\kappa_1=\kappa_0/4$, $\Vert v(\cdot, T_2+1)\Vert$ is much larger
   than $e^{-\kappa_0 T_2}$. Therefore for any $\lambda > 0$, we can 
 choose  $|X_1|$ sufficiently small, so that 
   \[ \Vert (\tilde{v} - v)(\cdot, \tau) \Vert_{9\tau^{1/2}} \leq
     \lambda \Vert v(\cdot, \tau)\Vert_{R(\tau)} + \lambda \Vert
     v(\cdot, T_2+1)\Vert_{R(T_2+1)}. \]
   It follows that
   \[ \Vert \tilde{v}(\cdot, T_2+1)\Vert_{9(T_2+1)^{1/2}} &\geq
     (1-\lambda) \Vert \tilde{v}(\cdot, T_2+1)\Vert_{R(T_2+1)} , \\
       \Vert \tilde{v}(T_2)\Vert_{9T_2^{1/2}} &\leq \Vert
       v(T_2)\Vert_{R(T_2)} + \lambda \Vert \tilde{v}(\cdot,
       T_2+1)\Vert_{R(T_2+1)} \\
       &\leq e^{-(\frac{1}{2}-\kappa_1)}\Vert v(\cdot,
       T_2+1)\Vert_{R(T_2+1)} + \lambda \Vert \tilde{v}(\cdot,
       T_2+1)\Vert_{R(T_2+1)} \\
       &\leq e^{-(\frac{1}{2} - 2\kappa_1)}\Vert
       \tilde{v}(\cdot, T_2+1)\Vert_{9(T_2+1)^{1/2}}, \]
     if $\lambda$ is chosen small enough, depending on
     $\kappa_1$. This shows that the translated flow $N'_\tau$ still
     has a growing behavior relative to $M_\tau$.

     \bigskip
     \noindent{\bf Step 2.}
     Next we want to keep applying Proposition~\ref{prop:3ann50}, with
     $\lambda = 2\kappa_1$ to propagate this growth to later
     times. We will now work on the balls of radius $R_2(\tau) =
     9\tau^{1/2}$. We can keep applying the Proposition to obtain
     \[ \label{eq:702} \Vert \tilde{v}(\cdot, T_2+k+1)\Vert_{R_2(T_2+k+1)} \geq
       e^{\frac{1}{2}-2\kappa_1} \Vert \tilde{v}(\cdot,
       T_2+k)\Vert_{R_2(T_2+k)}, \]
     as long as we maintain $\epsilon_0$-graphicality on these balls,
     as well as
     \[ \label{eq:701} \Vert \tilde{v}(\cdot, T_2+k)\Vert_{R_2(T_2+k)} \geq \epsilon_0
       e^{-R_2(T_2+k)^2/9}. \]
     The right hand side is bounded above by $\epsilon_0 e^{-9\tau}$. 
     Recall that we have the lower bound \eqref{eq:vT2lower}, and so
     because the $L^2$-norm is growing, for \eqref{eq:701} to hold, it is enough to ensure that
     \[ e^{-\kappa_1(T_2+1)} \geq \epsilon_0 e^{-9\tau}, \]
     and as long as $T_2$ is large enough, this will hold for all
     $\tau \geq T_2$. So we will have \eqref{eq:702} as long as we
     maintain $\epsilon_0$-graphicality on $B_{R_2(\tau)}$. We claim
     that Proposition~\ref{prop:Claim1} can be applied in this
     setting, viewing $N_\tau$ as a graph over $M_\tau$. Indeed the
     properties of $\mathcal{M}^{n,k,\ell}_\tau$ that were required
     are the bound $|A|^2 < \frac{1}{2} +\epsilon$ on
     $B_{10\tau^{1/2}}$ for large $\tau$, and the pseudolocality
     property of Proposition~\ref{prop:modelpseudo}. Since by
     assumption (a) at the beginning of the section $M_\tau$ has good
     graphicality over $\mathcal{M}^{n,k,\ell}_\tau$ on the
     $10\tau^{1/2}$-balls for large $\tau$, these properties will also
     hold for $M_\tau$.

     It follows that $N_\tau$ has good graphicality over $M_\tau$ on
     the $R_2(\tau)$-balls, as long as $\Vert \tilde{v}
     \Vert_{R_2(\tau)} \leq \tau^{-2}$. Iterating \eqref{eq:702}, we
     end up with a first time $T_3= T_2 + k_1$ such that we have both
     \[ \Vert \tilde{v}(\cdot, T_3+1)\Vert_{R_2(T_3+1)} \geq
       (T_3+1)^{-2}, \]
     and
     \[ \Vert \tilde{v}(\cdot, T_3+1)\Vert_{R_2(T_3+1)} \geq
       e^{\frac{1}{2}- 2\kappa_1} \Vert \tilde{v}(\cdot, T_3)
       \Vert_{R_2(T_3)}. \]

     \bigskip
     \noindent{\bf Step 3.}
     At this point the $L^2$-norm is no longer exponentially small,
     and so we can transition to working with balls with radii that
     are of smaller order than $\tau^{1/2}$. We define $R_3(\tau) =
     \tau^{1/5}$. As long as we maintain $\epsilon_0$-graphicality,
     the contribution to the $L^2$-norm from outside of the
     $\tau^{1/5}$-ball is of order $e^{-\tau^{2/5}/9}$, which is much
     smaller than $\tau^{-2}$. We will therefore have
     \[  \Vert \tilde{v}(\cdot, T_3+1)\Vert_{R_3(T_3+1)} \geq
       e^{\frac{1}{2}- 3\kappa_1} \Vert \tilde{v}(\cdot, T_3)
       \Vert_{R_3(T_3)}. \]

     We can apply Proposition~\ref{prop:3ann50} as long as we maintain
     $\epsilon_0$-graphicality on the $R_3(\tau)$-balls, and again
     using the argument of 
     Proposition~\ref{prop:Claim1}, this will
     be the case as long as $\Vert \tilde{v}\Vert_{R_3(\tau)} \leq
     \tau^{-1/2}$. In this way we end up with $T_4= T_3 + k_2$, such
     that
     \[ \label{eq:704} \Vert \tilde{v}(\cdot, T_4+1)\Vert_{R_3(T_4+1)} \geq
       (T_4+1)^{-1/2}, \]
     and
     \[ \label{eq:705}\Vert \tilde{v}(\cdot, T_4+1)\Vert_{R_3(T_4+1)} \geq
       e^{\frac{1}{2}- 3\kappa_1} \Vert \tilde{v}(\cdot, T_4)
       \Vert_{R_3(T_4)}. \]
  
     \bigskip
     \noindent{\bf Step 4.}
     At this time we have that the $L^2$-distance from $M_\tau$ to
     $N'_\tau$ is much larger than the $L^2$-distance from $M_\tau$ to
     $\mathcal{C}_{n,k}$, which is of order $\tau^{-1}$, by
     Proposition~\ref{prop:L2normalform}. 
     We will use Corollary~\ref{cor:slowgrowth}
     with $\epsilon=\frac{1}{4}$ (applied to the time translated flow
     $\tau\mapsto N'_ {\tau+T_4}$), to show that $N'_\tau$ cannot
     converge to a rotation $Q\mathcal{C}_{n,k}$ at infinity. From the
     Corollary we obtain a corresponding $L_0, \delta$. Note that
     using the pseudolocality result, Proposition~\ref{prop:pseudo10},
     there exists an $\epsilon_1 > 0$ depending on $L_0, \epsilon_1$,
     such that if $T_4$ is large enough and $N'_{T_4}$ is
     $\epsilon_1$-graphical over $M_{T_4}$ on $B_{R_3(T_4)}$, then for
     all $\tau\in [T_4+1, T_4 + L_0]$ we have that $N'_\tau$ is
     $\epsilon_0$-graphical over $M_\tau$ on
     $B_{R_3(\tau)}$. Moreover, using the argument of 
     Proposition~\ref{prop:Claim1} we can ensure that $N'_\tau$
     is actually $\epsilon_1$-graphical over $M_\tau$ as long as $\Vert
     \tilde{v}\Vert_{L^2} < \tau^{-1/2}$ if $T_0$ is large enough
     (depending on $\epsilon_1$), not just $\epsilon_0$-graphical as
     used above. In particular this applies up to $\tau=T_4$, and so
     we have $\epsilon_1$-graphicality on the $R_3(\tau)$-balls up to
     $\tau = T_4+L_0$. Then Proposition~\ref{prop:3ann50} can be
     applied $L_0$ more times, and \eqref{eq:704}, \eqref{eq:705} imply
     \[\Vert \tilde{v}(\cdot, T_4+L_0)\Vert_{R_3(T_4+L_0)} &\geq
       (T_4+L_0)^{-1/2}, \\
       \Vert \tilde{v}(\cdot, T_4+L_0)\Vert_{R_3(T_4+L_0)} &\geq
       e^{L_0(\frac{1}{2}- 3\kappa_1)} \Vert \tilde{v}(\cdot, T_4)
       \Vert_{R_3(T_4)}. \]

     We know that
     $N'_{T_4}$ is $\epsilon_0$-graphical on the $R_3(T_4) =
     T_4^{1/5}$-ball over $M_{T_4}$, and $M_{T_4}$ is
     $T_4^{-\beta}$-graphical over $\mathcal{C}_{n,k}$ on
     $B_{T_4^\beta}$, for the $\beta$ in
     Proposition~\ref{prop:tbetagraph11}. So for the given $\delta$ in
     Corollary~\ref{cor:slowgrowth}
     we can arrange that $N'_{T_4}$ is $\delta$-graphical over
     $\mathcal{C}_{n,k}$ on $B_{\delta^{-1}}$, if $T_4$ is large
     enough. To apply the corollary, it remains to verify the growth of
     $\mathbf{d}_{\mathcal{C}_{n,k}}$. 

     Let us write $N'_\tau$ as the graph of
     $\tilde{V}$ over
     $\mathcal{C}_{n,k}$ on the $\tau^\beta$-ball, for the $\beta$ in
     Proposition~\ref{prop:tbetagraph11}. To leading order we have
     $\tilde{V} = u + \tilde{v}$. The contribution to $\dd(N'_{T_4})$
     and $\dd(N'_{T_4+L_0})$ from outside of the $\tau^\beta$-balls is
     bounded by $e^{-\tau^{2\beta}/9}$, so is much smaller than
     $\tau^{-1/2}$. Similarly the $L^2$-norm of $u$ is of order
     $\tau^{-1}$. Therefore, for any $\lambda > 0$, if $T_4$ is
     sufficiently large (i.e. if $L_0$ is chosen large),  we have
     \[ \dd(N'_{T_4+L_0}) &\geq (1-\lambda) \Vert \tilde{v}(T_4+L_0)
       \Vert_{R_3(T_4+L_0)}  \\
         \dd(N'_{T_4}) &\leq \Vert \tilde{v}(T_4)\Vert_{R_3(T_4)} +
         \lambda  (T_4+L_0)^{-1/2} \\
       &\leq e^{-L_0( \frac{1}{2}-3\kappa_1)} \Vert
       \tilde{v}(T_4+L_0)\Vert_{R_3(T_4+L_0)} + \lambda \Vert
       \tilde{v}(T_4+L_0)\Vert_{R_3(T_4+L_0)} \\
     &\leq (1-\lambda)^{-1} \left(e^{-L_0( \frac{1}{2}-3\kappa_1)}  +
       \lambda\right) \dd(N'_{T_4+L_0}) \\
    &\leq e^{-\frac{1}{4}L_0} \dd(N'_{T_4+L_0}),\]
    if $\lambda$ is sufficiently small (depending on $L_0$). Using
    Corollary~\ref{cor:slowgrowth}, this implies that $N'_\tau$ cannot
    converge to a rotation $Q\mathcal{C}_{n,k}$ at infinity. 
\end{proof}

Let us consider the flow $L_t$ from the beginning of the section
again. Recall that this is a small perturbation of a translation and
rotation of the original mean curvature flow that we started
with, such that $L_t$ has an $\ell$-degenerate
$\mathcal{C}_{n,k}$-singularity at the origin. We assume that this
singularity is $\ell$-degenerate at scale $r > 0$, so by
Proposition~\ref{prop:nearbydegen} there is an $r' > 0$ such that all
of the $\ell$-degenerate $(n,k)$-singularities near the origin
of sufficiently nearby flows are $\ell$-degenerate at scale
$r'$. Letting $T_1 = -2\ln r'$ this means that all of the
corresponding rescaled flows have good asymptotics at time $T_1$. 

Let $M_\tau$ be the rescaled flow, centered at the origin,
corresponding to $L_t$, so $M_\tau$ has an $\ell$-degenerate
$\mathcal{C}_{n,k}$-singularity at infinity. For all sufficiently
small $\mathbf{a}\in \mathbb{R}^{k-\ell}$ (how small can depend on the particular
flow $M_\tau$), we define the perturbation $M^{\mathbf{a}}_{T_1+1}$ by
taking the graph of $\chi(|x| T_1^{-\beta/2}) \mathbf{a} \cdot y$ over
$M_{T_1+1}$, where $\beta$ is from
Proposition~\ref{prop:tbetagraph11}.
Here $\chi$ is a standard cutoff function as before,
such that $\chi(s) = 0$ for $s > 3/2$ and $\chi(s)=1$ for $s < 1$,
and by an abuse of notation we are writing
\[ \mathbf{a}\cdot y = \sum_{j=\ell+1}^k a_{j-\ell} y_j, \]
where the $a_1, \ldots, a_{k-\ell}$ are the components of
$\mathbf{a}$. In other words we
have a $(k-\ell)$-dimensional family of perturbation of $M_{T_1+1}$,
corresponding to the degenerate directions. We define
$M^{\mathbf{a}}_\tau$ for $\tau \geq T_1+1$ to be the corresponding
rescaled flow, and $L^{\mathbf{a}}_t$ the corresponding mean curvature
flow (for $t \geq -e^{-T_1-1}$).

We next follow a similar argument to \cite[Proposition 20]{Sz26} to show
that if $L^{\mathbf{a}_1}_t$ and $L^{\mathbf{a}_2}_t$ both have
$\ell$-degenerate $(n,k)$-cylindrical singularities near $(0,0)$, then
the degenerate $y$-components must be at least
$r_0 |\mathbf{a}_1-\mathbf{a}_2|^{1-\kappa_1}$ apart, where $r_0,
\kappa_1$ are as in Proposition~\ref{prop:perturb20}. The constant
$c_1$ used in that proposition will be determined below, and recall that
$r_0$ depends on $c_1$, as well as on $T_1$, which in turn depends on the scale $r$ at
which the singularity at $(0,0)$ was $\ell$-degenerate. 

In order to deal with translations in the nondegenerate directions, it
is useful to introduce the following notation. Suppose that $\beta\in
\mathbb{R}^{\ell-k}$, $\alpha_1\in \mathbb{R}^\ell,\alpha_2\in
\mathbb{R}^{n-k+1}$ and $s\in \mathbb{R}$. We write $\Phi = (\alpha_1,
\alpha_2, s)$, and define the translated flows
\[ L^{\mathbf{a}, \Phi, \beta}_t = L^{\mathbf{a}}_{t+s} + (\alpha_1,
  \beta, \alpha_2). \]
Here we view $(\alpha_1, \beta, \alpha_2) \in \mathbb{R}^{\ell}\times
\mathbb{R}^{k-\ell}\times \mathbb{R}^{n-k+1} = \mathbb{R}^{n+1}$. Note
that $\beta$ encodes the translation in the degenerate directions, 
$\alpha_1, \alpha_2$ encodes the remaining space directions, and
$\Phi$ also includes the time direction. 
We have the following.
\begin{prop}\label{prop:perturb22}
  Suppose that $T_1$ is sufficiently large (depending on the dimension
  and entropy bound). There is an $r_2=r_2(T_1) > 0$, and
  $a_0 > 0$ depending on the particular flow $L_t$, with the
  following property. Suppose that $|\mathbf{a}_1|, |\mathbf{a}_2| <
  a_0$, and $L^{\mathbf{a}_1}_t$, $L^{\mathbf{a}_2}_t$ have
  $\ell$-degenerate $(n,k)$-cylindrical singularities at the points
  $(y_i, z_i, t_i)$, in the $r_2$-neighborhood of
  $(0,0,0)$. Decompose $y_i = (y_i', y_i'')$ into the nondegenerate and
  degenerate components according to $\mathbb{R}^k = \mathbb{R}^\ell
  \times \mathbb{R}^{k-\ell}$. Then the degenerate components satisfy
  $|y_1'' - y_2''| \geq r_2 |\mathbf{a}_1 -
  \mathbf{a}_2|^{1-\kappa_1}$, for the $\kappa_1$ in
  Proposition~\ref{prop:perturb20}. 
\end{prop}
\begin{proof}
  From Proposition~\ref{prop:nearbydegen} we have that if
  $L_t^{\mathbf{a}}$ has an $\ell$-degenerate $(n,k)$-cylindrical singularity at $(x_0,
  t_0)$ with $|x_0| + |t_0|^{1/2} < r_2$, then we have a rotation
  $Q$ with $|Q - Id| < \Psi(r_2|T_1)$ such that $QL_t^{\mathbf{a}}$
  has an $\ell$-degenerate $\mathcal{C}_{n,k}$-singularity. It follows
  that once $r_2$ is small enough, the $\mathbb{R}^{k-\ell}$
  subspace corresponding to the degenerate directions of the different $\ell$-degenerate
  singularities that show up are all small perturbations of each
  other.

  After a translation of size at most $r_2$, and a
  $\Psi(r_2|T_1)$-rotation, we can suppose that
  $L_t^{\mathbf{a}_1}$ has an $\ell$-degenerate
  $\mathcal{C}_{n,k}$-singularity at $(0,0)$. It is enough to
  show that $L^{\mathbf{a}_2}_t$ has no $(n,k)$-cylindrical
  singularity in the $2r_2$-neighborhood of the origin, with
  degenerate $y$-components at most
  $r_2 |\mathbf{a}_2-\mathbf{a}_1|^{1-\kappa_1}$. Suppose
 therefore that $L^{\mathbf{a}_2}_t$ has an $(n,k)$-cylindrical
  singularity at a point $(y_0', y_0'', z_0, s_0)$, where
  \[  |y_0'| + |y_0''| + |z_0| + |s_0|^{1/2} &< 2r_2. \]
  We want to show that $|y_0''| \geq r_2 |\mathbf{a}_2-\mathbf{a}_1|^{1-\kappa_1}$.
   Let us define the
  translated flow
  \[ \label{eq:730}  \widetilde{L}^{\mathbf{a}_2}_t = L^{\mathbf{a}_2}_{t+s_0} + (y_0',
    0, z_0), \]
  which has an  $(n,k)$-singularity at $(0,y_0'',
  0,0)$.
  
  We let $M_\tau$ be the
  rescaled flow corresponding to $L^{\mathbf{a}_1}_t$ (centered at the
  origin), and $\widetilde{N}_\tau$ the rescaled flow corresponding to
  $\widetilde{L}^{\mathbf{a}_2}_t$. We will verify that the hypotheses
  of Proposition~\ref{prop:perturb20} hold for suitable $c_1$, once $T_1$ is
  sufficiently large and $r_2$ is small, viewing
  $\widetilde{N}_\tau$ as a graph over $M_\tau$. We will consider
  conditions (ii), (iii) at times $T_1+2$ and $T_1+3$. Note that our
  construction above defined $M^{\mathbf{a}}_\tau$ for $\tau \geq
  T_1+1$, and $\widetilde{L}^{\mathbf{a}_2}_t$ involved a further
  translation in time by $s_0$. If $s_0$ is small enough, the corresponding rescaled flow
  $\widetilde{N}_\tau$ will be defined for all $\tau \geq T_1+2$. 

  In Proposition~\ref{prop:perturb20} we will use
  \[ \gamma = T_1^{-1} e^{T_1/2} |y_0'| + |\mathbf{a}_2 -
    \mathbf{a}_1| + e^{T_1/2} |z_0| + e^{T_1} |s_0|.\]
  As we will see below, this $\gamma$ measures the size of the graphicality function of
  $\widetilde{N}_\tau$ over $M_\tau$ for $\tau\in [T_1+2, T_1+3]$, in
  an $L^2$-sense. In
  particular the $T_1^{-1}$ factor is related to
  the fact that by the normal form \eqref{eq:SXC1} of $\mathcal{M}^{n,k,\ell}_\tau$ a
  translation by $e^{T_1/2} y_0'$ in the nondegenerate directions only has an effect
  of order $T_1^{-1}e^{T_1/2} |y_0'|$ on the level of the graphicality
  function. 
  
  Let us consider the hypothesis
  (i), which needs pointwise bounds for the graphicality function.
  For this we first compare $L^{\mathbf{a}_1}_t$ to
  $L^{\mathbf{a}_2}_t$. Let us denote by $N_\tau$ the rescaled flow,
  centered at the origin, corresponding to $L^{\mathbf{a}_2}_t$. By
  construction of the flows $M^{\mathbf{a}}_\tau$ we have that $N_{T_1+1}$
  is, to leading order, the graph of $\chi(|x| T_1^{-\beta/2})
  (\mathbf{a}_2 - \mathbf{a}_1)\cdot y$ over $M_{T_1+1}$.

  For $\tau\in [T_1, T_1+3]$ we know from
  Proposition~\ref{prop:tbetagraph11} that $M_\tau$ is
  $\tau^{-\beta}$-graphical over $\mathcal{C}_{n,k}$ on the
  $\tau^\beta$ ball. For small $s$, consider the flow
  $\widetilde{M}^s_\tau$ defined by
  \[\label{eq:Mstdefn} \widetilde{M}^s_\tau = (1 - e^{\tau-T_1-1}s)^{1/2} M_{\tau - \log(1
      - e^{\tau-T_1-1}s)}, \]
  which corresponds to translating the ``unrescaled'' flow
  $L^{\mathbf{a}_1}_t$ in time by $s$. The graphicality bounds for
  $M_\tau$ over $\mathcal{C}_{n,k}$ imply that for $\tau\in [T_1,
  T_1+3]$,  on $B_{\tau^\beta}$ 
  for small $s > 0$, the hypersurface $\widetilde{M}^s_\tau$ is on the
  negative side of the graph of $-c_n|s|$, while for $s < 0$,
  $\widetilde{M}^s_\tau$ is on the positive side of the graph of $c_n
  |s|$. Here $c_n > 0$ is a dimensional constant (using that the radius of
  $\mathcal{C}_{n,k}$ is $\sqrt{2(n-k)}$). We have
  \[ \left|\chi(|x| T_1^{-\beta/2}) (\mathbf{a}_2 - \mathbf{a}_1)\cdot y
    \right| \leq 2T_1^{\beta/2} |\mathbf{a}_2 - \mathbf{a}_1| \leq
    2 T_1^{\beta/2} \gamma, \]
  and this function is supported inside the $T_1^\beta$ ball. 
  So if we set $\bar s = c_n^{-1} 2 T_1^{\beta/2} \gamma$, then
  $N_{T_1+1}$ is contained between 
  the two hypersurfaces $\widetilde{M}^{\pm \bar s}_{T_1+1}$ (globally, not
  just on the $T_1^\beta$-ball). We can use these flows as barriers,
  and it follows that for all later times
  $N_{\tau}$ is contained between $\widetilde{M}^{\pm
    \bar s}_\tau$. Because $M_\tau$ has good asymptotics at time $T_1$, we
  know that for $\tau \geq T_1$, $M_\tau$ is $\epsilon_0$-graphical
  over $\mathcal{M}^{n,k,\ell}_\tau$ on the
  $10\tau^{1/2}$-ball. Recall from Proposition~\ref{prop:models} that
  on this ball $\mathcal{M}^{n,k,\ell}_\tau$ is close to cylinders of
  different radii, uniformly bounded away from 0 and from
  above. Therefore, from
  \eqref{eq:Mstdefn} we see that as long as $e^{\tau-T_1-1}\bar s$ is
  sufficiently small, the hypersurfaces $\widetilde{M}^{\pm \bar s}_\tau$
  lie inside of the graphs of $\pm c_n^{-1} e^{\tau-T_1}\bar s$ over
  $M_\tau$, decreasing $c_n$ if necessary. It follows that
  for this range of $\tau$, $N_\tau$ is the graph of $v$ over $M_\tau$
  on the $10\tau^{1/2}$-ball, where
  \[\label{eq:731} |v(x, \tau)| \leq c_n^{-1} e^{\tau-T_1} \bar s = 2c_n^{-2} T_1^{\beta/2}
    e^{\tau-T_1} \gamma.\]
  Using interior regularity, and replacing $c_n$ with a smaller
  constant, the same estimate holds for derivatives of $v$ as well.

  It remains to compare $\widetilde{N}_\tau$ to $N_\tau$. By
  definition, using \eqref{eq:730} we have
  \[ \label{eq:732} \widetilde{N}_\tau = (1 - e^\tau s_0)^{1/2} N_{\tau -
      \log(1-e^\tau s_0)} + e^{\tau/2}(y_0', 0, z_0). \]
  As long as $e^\tau
  |s_0|$ and $e^{\tau/2}|(y_0',0,z_0)|$ are sufficiently small, it follows from \eqref{eq:732}
  that $\widetilde{N}_\tau$ is the graph of a function $w$ over
  $N_\tau$ on the $10\tau^{1/2}$-ball, with
  \[ |w| &\leq c_n^{-1} e^\tau |s_0| + e^{\tau/2} |(y_0', 0, z_0)| \leq
    c_n^{-1} e^{\tau-T_1} \gamma + T_1 e^{(\tau-T_1)/2}\gamma \leq
    c_n^{-1} T_1 e^{\tau-T_1}\gamma, \]
  if $c_n^{-1}, T_1 > 2$. 
  Combining this with \eqref{eq:731} we find that as long as $c_n^{-2}
  T_1 e^{\tau-T_1} \gamma$ is
  sufficiently small, then on the
  $10\tau^{1/2}$-ball,   $\widetilde{N}_\tau$ is the graph of a
  function $\widetilde{v}$ over $M_\tau$, with
  \[ \label{eq:741} |\widetilde{v}(x, \tau)| \leq 2c_n^{-2} T_1 e^{\tau-T_1}
    \gamma. \]
  If we choose $c_1$ so that $c_1^{-1} >  2c_n^{-2}$, then condition
  (i) in Proposition~\ref{prop:perturb20} will hold.

  To see conditions (ii), (iii), we need to study the graphicality
  function $\tilde{v}$ of $\widetilde{N}_\tau$ over $M_\tau$ for
  $\tau\in [T_1+2, T_1+3]$ in $L^2$. First we will argue that for
  large $R_0 > 0$ we can ensure that the contribution to the
  $L^2$-norm from outside of $B_{R_0}$ is negligible compared to
  $\gamma$ (independently of $T_1$).
  Note that the pointwise estimate \eqref{eq:741} is not
  quite enough for this because of the $T_1$ factor. Recall that
  $N_{T_1+1}$ is given by the graph of $\chi(|x|T_1^{-\beta/2})
  (\mathbf{a}_2-\mathbf{a}_1)\cdot y$ over $M_{T_1+1}$. So the
  graphicality function $v$ satisfies $\Vert v(\cdot,
  T_1+1)\Vert_{L^2} \leq C |\mathbf{a}_2-\mathbf{a}_1|$. Using
  \eqref{eq:732} we can find $\widetilde{T_1}$ with $|\widetilde{T_1} -
  T_1| \leq C\gamma$,  such that
  \[ \widetilde{N}_{\widetilde{T_1}+1} = e^{(\widetilde{T_1}-T_1)/2}
    N_{T_1+1} + e^{(\widetilde{T_1}+1)/2} (y_0', 0, z_0). \]
  Arguing similarly to above, it follows that the graphicality
  function $\widetilde{v}$ of
  $\widetilde{N}_\tau$ over $M_\tau$ satisfies
  \[ \Vert \tilde{v}(\cdot, \widetilde{T_1} + 1)\Vert_{L^2} \leq C
    \gamma. \]
  We can now use Proposition~\ref{prop:nonconc30} to get a pointwise
  estimate
  \[ |\tilde{v}(x, \tau)| \leq C\gamma e^{|x|^2/8p} \]
  for some $p > 1$, for $\tau\in [T_1+2, T_1+3]$. From this it follows
  using H\"older's inequality that for $\tau\in [T_1+2, T_1+3]$
  the contribution to the $L^2$-norm
  of $\tilde{v}$ from outside of $B_{R_0}$ is $\gamma \Psi(R_0^{-1})$.  As
  a result, for checking (ii), (iii), it is enough to work on a ball
  $B_{R_0}$ for a large, but fixed $R_0$. Recall that on this ball  the flow
  $M_\tau$ is $\tau^{-\beta}$-graphical over $\mathcal{C}_{n,k}$. 

  Consider first $N_\tau$ as the graph of
  $v$ over  $M_\tau$, for $\tau\in [T_1+2, T_1+3]$. If $T_1$ is
  sufficiently large, and $\mathbf{a}_1, \mathbf{a}_2$ are
  sufficiently small, then to leading order, on $B_{R_0}$,
   $v$ is given by the solution $(\mathbf{a}_2-\mathbf{a}_1)\cdot y
  e^{(\tau-T_1-1)/2}$ of the linearized equation on
  $\mathcal{C}_{n,k}$. More precisely, as we let $T_1\to \infty$, and
  $\mathbf{a}_1, \mathbf{a}_2\to 0$, then along a subsequence (after shifting time by
  $T_1+1$), the normalized functions $|\mathbf{a}_2-\mathbf{a}_1|^{-1}
  v$ converge uniformly on $B_{R_0}$
  to $\mathbf{a}\cdot y e^{(\tau-1)/2}$, where $\mathbf{a}$ is
  the limit of the normalized vectors
  $|\mathbf{a}_2-\mathbf{a}_1|^{-1}(\mathbf{a}_2-\mathbf{a}_1)$.
   Similarly, the leading order contribution to the graphicality
   function of  translating by $(0,0,z_0) e^{\tau/2}$ in space
  is $z_0\cdot z e^{\tau/2}$, while the contribution of 
  translating the unrescaled flow in time by $s_0$ is 
  $-c_{n,k} s_0 e^\tau$ for a constant $c_{n,k} > 0$.
  
  Translating by $e^{\tau/2}(y_0',0, 0)$ is more
  subtle, since the cone $\mathcal{C}_{n,k}$ is invariant under
  translations in this  direction. However we still have a lower order
  effect using that in these
  directions $M_\tau$ is non-degenerate. Using the $C^1$-normal form
  of $\mathcal{M}^{n,k,\ell}_\tau$ and the $\tau^{-3/2}$-graphicality
  of $M_\tau$ over $\mathcal{M}^{n,k,\ell}_\tau$ on $B_{R_0}$ for
  large $\tau$ (depending on $R_0$), we see that to leading order,
  translating $N_\tau$ by $e^{\tau/2}(y_0',0, 0)$ changes the
  graphicality function by $c_{n,k}' \tau^{-1} e^{\tau/2} (y_0',
  0)\cdot y$, for some $c_{n,k}' > 0$. 

  Overall we find that to leading order the graphicality function
  $\tilde v$ of
  $\widetilde{N}_\tau$ over $M_\tau$, for $\tau\in [T_1+2, T_1+3]$ on
  $B_{R_0}$ is given by
  \[ (\mathbf{a}_2-\mathbf{a}_1)\cdot y e^{(\tau-T_1-1)/2} + c_{n,k}'
    \tau^{-1} e^{\tau/2} (y_0', 0)\cdot y + e^{\tau/2} z_0\cdot z
    -c_{n,k} s_0 e^\tau, \]
  If $R_0, T_1$ are sufficiently large, and $\gamma$ is small, then
  the conditions (ii), (iii) for the $L^2$-behavior of $v$ follow.

  The conclusion of Proposition~\ref{prop:perturb20} is that
  $\widetilde{N}'_\tau = \widetilde{N}_\tau + e^{(\tau-T_1-2)/2}x_0$
  does not have an $(n,k)$-cylindrical singularity at infinity, if
  $|x_0| < r_0 \gamma^{1-\kappa_1}$, and $x_0$ is only nonzero in the
  degenerate $y$-directions. Since the flow
  $\widetilde{L}^{\mathbf{a}_2}_t$ in \eqref{eq:730}, corresponding to
  $\widetilde{N}_\tau$, has an $(n,k)$-singularity at $(0,y_0'', 0,0)$,
  it follows that $\widetilde{N}_\tau - e^{\tau/2}(0,y_0'',0)$ has an
  $(n,k)$-cylindrical singularity at infinity. Therefore
  \[ e^{(T_1+2)/2} |y_0''| \geq r_0 \gamma^{1-\kappa_1} \geq r_0
    |\mathbf{a}_2-\mathbf{a}_1|^{1-\kappa_1}. \]
  Setting $r_2 = e^{-(T_1+2)/2}r_0$, this shows $|y_0''| \geq r_2
  |\mathbf{a}_2-\mathbf{a}_1|^{1-\kappa_1}$, as required.   
\end{proof}

Using this result, we can show similarly to \cite[Proposition
21]{Sz26}, that for almost every $\mathbf{a}$ the perturbed flow
$L^{\mathbf{a}}_t$ has no degenerate $(n,k)$-singularity in the
$\epsilon$-neighborhood of the origin.
\begin{prop}\label{prop:perturb52}
   There is an $\underline r_0  > 0$ depending on the dimension and entropy bounds
   with the following property. 
    Suppose that $L_t$ has an $\ell$-degenerate
    $\mathcal{C}_{n,k}$-singularity at the origin, which is
    $\ell$-degenerate at scale $r$ for some  $r < \underline r_0$. There exists an
    $\epsilon=\epsilon(r)$, and an $a_0 > 0$ depending on $L_t$,
    such that for almost every $\mathbf{a}\in \mathbb{R}^{k-\ell}$
    with $|\mathbf{a}| < a_0$ 
    the flow $L^{\mathbf{a}}_t$ has no $\ell$-degenerate
    $(n,k)$-cylindrical singularity in the $\epsilon$-neighborhood of
    the origin.
  \end{prop}
  \begin{proof}
    The choice of $\underline r_0$ is determined so that if we set
    $T_1 > -2\log \underline r_0$, then $T_1$ is large enough for
    Proposition~\ref{prop:perturb22}. Given $r < \underline{r}_0$ we
    let $T_1 = -2\log r$, and then let $r_2, a_0$ be
    determined by Proposition~\ref{prop:perturb22}.  We set
    $\epsilon=r_2$. Let us define
    \[ \mathcal{D} = \{ (\mathbf{a}, y_0'') \,:\, &|\mathbf{a}| < a_0,
      |y_0''| < \epsilon ,  L^{\mathbf{a}}_t \text{ has an }\ell\text{-degenerate }\\
         &\qquad (n,k)\text{-cylindrical singularity at } (x_0,
         t_0), \text{ where } \\
         &\qquad |x_0| + |t_0|^{1/2} < \epsilon , \text{ and } x_0 =
         (y_0', y_0'', z_0) \text{ has } y_0'' \\
         &\qquad \text{as its degenerate
           component}\}. \]
       We also define the projection onto the parameter space:
     \[ \mathcal{S} = \{\mathbf{a}\,:\, (\mathbf{a}, y_0'')\in
       \mathcal{D}\} \subset \mathbb{R}^{k-\ell}. \]
     Let $\delta > 0$ be small, and find a maximal collection of pairs
     $(\mathbf{a}_i, y_i'')\in \mathcal{D}$, such that $|\mathbf{a}_i
     - \mathbf{a}_j| > \delta$ for $i\not=j$. By
     Proposition~\ref{prop:perturb22} we then have $|y_i'' - y_j''| \geq
     \epsilon  \delta^{1-\kappa_1}$. Since $|y_i''| < \epsilon$ and $y_i'' \in
     \mathbb{R}^{k-\ell}$, there can be at
     most $\delta^{-(k-\ell)(1-\kappa_1)}$ distinct such pairs. It
     follows that $\mathcal{S}$ can be covered with
     $\delta^{-(k-\ell)(1-\kappa_1)}$ balls of radius $\delta$, so the
     Lebesgue measure  $\mu(\mathcal{S}) \leq C \delta^{-(k-\ell)(1-\kappa_1) +
       (k-\ell)}=C\delta^{\kappa_1(k-\ell)}$. Since $\ell < k$, and
     $\delta > 0$ can be taken arbitrarily small, we get
     $\mu(\mathcal{S})=0$.  
 \end{proof}

The following summarizes our main result on perturbing away degenerate
cylindrical singularities.
\begin{prop}\label{prop:perturbdegen10}
  For sufficiently small $r > 0$ there exists a function $\Phi(r) > 0$, depending on the
  dimension and the entropy bound for the flows that appear, with the following
  property. Suppose that the flow $L_t$ has an $\ell$-degenerate
  $(n,k)$-cylindrical singularity at a point $(x,T)$, and this
  singularity is $\ell$-degenerate at scale $2r > 0$. Then there are arbitrarily small
  $C^\infty$ perturbations $L'_{T-r^2}$ of $L_{T-r^2}$ such that the
  corresponding flow $L'_t$ has no $\ell'$-degenerate
  $(n,k')$-cylindrical singularity in the parabolic neighborhood
  $Q_{\Phi(r)}(x, T)$ if either $k=k'$ and $\ell' \leq \ell$, or
  if $k' > k$.
\end{prop}
\begin{proof}
  Translating the flow in spacetime we can assume that
  $(x,T)=(0,0)$. We can assume that $r < \underline{r}_0$ for the
  $\underline{r}_0$ in Proposition~\ref{prop:perturb52}, and consider
  the $\epsilon(r)$ given by the proposition. The
  conclusion of Proposition~\ref{prop:perturb52} is that there are arbitrarily small
  $\mathbf{a}$ for which the perturbed flow $L^{\mathbf{a}}_t$ has no
  $\ell$-degenerate $(n,k)$-cylindrical singularity in
  $Q_{\epsilon(r)}(0,0)$. Note that by definition $L^{\mathbf{a}}_t$ is
  obtained by perturbing $M_{T_1+1}$, where $M_\tau =
  e^{\tau/2}L_{-e^{-\tau}}$, and $T_1 = -2\log r'$, where $M_\tau$ has
  good asymptotics at time $T_1$. We can apply this with suitable $r'
  < 2r$ so that $M_{T_1+1}$ corresponds to $L_{-r^2}$, i.e. $e^{2\log
    r' -1} = r^2$. Then Proposition~\ref{prop:perturb52} gives rise to
  a perturbation $L'_{-r^2}$ of $L_{-r^2}$ such that the corresponding
  flow $L'_t$ has no $\ell$-degenerate
  $(n,k)$-cylindrical singularity in $Q_{\epsilon(r)}(0,0)$. Using
  Proposition~\ref{prop:singlimits} we also have that if the
  perturbation is sufficiently small, then for a suitable $\Phi(r)  <\epsilon(r)$
  the perturbed flow $L'_t$ can also not have
  $\ell'$-degenerate $(n,k')$-cylindrical singularities in
  $Q_{\Phi(r)}(0,0)$ for $k' > k$, or $k'=k$ and $\ell' \leq \ell$. 
\end{proof}

\section{Proof of the main theorem}\label{sec:proof}
In this section we will prove our main result. We first define the
notion of \emph{almost mean curvature flow}, or AMCF (see also the notion
of piece-wise MCF used by Colding-Minicozzi~\cite{CM12}). 

\begin{definition}\label{defn:AMCF}
  An almost mean curvature flow in $\mathbb{R}^{n+1}$ over an interval
  $I=[T_1,T_2]$ is a 
  family of Radon measures $\nu_t$ for all $t\in I$, such that there
  is a finite set of times $T_1=t_0 < t_1 < t_2 < \ldots < t_N = T_2$ 
  with the following properties:
  \begin{enumerate}
    \item For each $0 \leq j < N$, the $\nu_t$ define an
      $n$-dimensional integral Brakke flow on the interval $(t_j,
      t_{j+1})$. 
    \item These flows are smooth embedded flows of
      hypersurfaces $L_t^j$ with uniformly bounded second
      fundamental form for $t$ near the endpoints $t_j$ and $t_{j+1}$,
      and as $t\to t_j$ and $t\to t_{j+1}$ we have smooth embedded limit
      hypersurfaces $L^j_{t_j}$ and $L^j_{t_{j+1}}$.
    \item For $j=1, \ldots, N-1$ the smooth hypersurfaces
      $L^{j-1}_{t_j}$ and $L^j_{t_j}$ are smoothly isotopic.
    \end{enumerate}
  We say that the AMCF $\nu_t$ has only nondegenerate cylindrical
  singularities, if this is true on each interval $(t_j, t_{j+1})$. 
\end{definition}

In other words an AMCF is a mean curvature flow which is modified using smooth
isotopies at finitely many times. We will only be concerned with AMCFs
with nondegenerate cylindrical singularities. For these each $\nu_t$
is given by the Hausdorff measure on a hypersurface $L_t$, which is
smooth away from finitely many points in spacetime. We expect that eventually this
definition will not be needed, and it is simply designed to
encode what we can prove at the moment. Let us restate the main result
of the paper, Theorem~\ref{thm:main}, with an additional statement
about $q$-convex flows. 

\begin{thm}\label{thm:mainAMCF}
  If $L_0\subset \mathbb{R}^{n+1}$ is a smooth, compact, embedded, mean convex
  hypersurface, then there exists an almost mean curvature flow $L_t$
  with initial condition $L_0$,  that has only nondegenerate
  cylindrical singularities, and becomes extinct in finite time. If
  the initial hypersurface is $q$-convex for some $q=1,2,\ldots, n$, then the same holds for all
  time.  
\end{thm}

When $n=1$, then this result is due to Gage-Hamilton~\cite{GH86},
although by Grayson~\cite{Gray89} the mean convexity assumption is not
required in this case. When $n=2$, the result was shown by the author
in \cite{Sz26}. Again, the mean convexity assumption is not required,
because of the work of Bamler-Kleiner~\cite{BK23} on the multiplicity
one conjecture, and Chodosh-Choi-Mantoulidis-Schulze~\cite{CCMS24, CCMS24_2} on
the singularities of generic flows in $\mathbb{R}^3$.

We will suppose
that $n > 1$, and by induction we will assume that
Theorem~\ref{thm:mainAMCF} has been
shown in lower dimensions. We first observe that, because of this assumption,
model flows $\mathcal{M}^{n', k'}_\tau$ with nondegenerate
$\mathcal{C}_{n',k'}$-singularities exist.

\begin{prop}\label{prop:nondegexist}
  Suppose that Theorem~\ref{thm:mainAMCF} holds in dimension
  $n$. Then for any $k=1, \ldots, n-1$ there exists a mean convex mean
  curvature flow $L_t^n \subset \mathbb{R}^{n+1}$ of compact
  hypersurfaces, which develops a nondegenerate
  $\mathcal{C}_{n,k}$-singularity.
\end{prop}
\begin{proof}
  Let $P\subset \mathbb{R}^{k+1}$ be a smooth compact
  hypersurface. Viewing $P$ as a codimension $(n+1-k)$ submanifold of
  $\mathbb{R}^{n+1}$ we define $S$ to be the $r$-sphere bundle inside
  the normal bundle of $P$, for small $r$. If we choose $r$
  sufficiently small, then $S$ is $k+1$-convex. By our hypothesis
  there is an AMCF with initial hypersurface $S$. We can assume that the
  $k+1$-convexity is preserved (see
  Huisken-Sinestrari~\cite{HS08}, White~\cite{White15}). Then 
  all the singularities are 
  nondegenerate $\mathcal{C}_{n, k'}$-singularities with $k' \leq k$.
  If there were no nondegenerate $\mathcal{C}_{n,k}$-singularities,
  then by Sun-Wang-Xue~\cite[Corollary 1.5]{SWX25} we would have
  $b_{n-k+1}(\Omega, \partial\Omega) = 0$, where $\Omega$ is the
  region with boundary $S = \partial\Omega$. This is a contradiction,
  since topologically $\Omega = P\times B^{n-k+1}$. 
\end{proof}

In particular we have model flows
$\mathcal{M}^{n',k'}_\tau$ for all $n' < n$ and $0 < k' < n'$ with
nondegenerate $\mathcal{C}_{n', k'}$-singularities, as required in
Proposition~\ref{prop:models}. 
In view of Proposition~\ref{prop:singlimits}, we have a hierarchy of degenerate
cylindrical singularities: given pairs $(k, \ell)$ and
$(k', \ell')$ with $0\leq \ell < k$ and $0\leq \ell' < k'$, we say that $(k', \ell')
\succ (k, \ell)$ if
\begin{itemize}
\item either $k'=k$ and $\ell' < \ell$,
\item or $k' > k$.
\end{itemize}

Given a mean curvature flow $L_t$ with only
cylindrical singularities, let us denote by $\Sigma_{k,\ell}$ the set
of $\ell$-degenerate $(n,k)$-cylindrical singularities, and let us
define
\[ \Sigma^{>}_{k, \ell} := \bigcup_{\substack{0 \leq \ell' < k' \leq n-1 \\ (k',
      \ell') \succ (k,\ell)}}
  \Sigma_{k', \ell'}. \] 
The notation suggests that $\Sigma^>_{k,\ell}$ contains the
degenerate cylindrical singularities that are ``worse'' than
$\Sigma_{k,\ell}$. Note that there cannot be any nondegenerate
cylindrical singularity in the closure of $\Sigma_{k,\ell}$, since
nondegenerate singularities are isolated in spacetime by
\cite{SX22}. From Proposition~\ref{prop:singlimits} it follows then that the
closure of $\Sigma_{k,\ell}$ is contained in
$\Sigma_{k,\ell} \cup \Sigma^>_{k,\ell}$.  When necessary, we will emphasize which flow
we are talking about by writing $\Sigma_{k,\ell}(L_t)$ and
$\Sigma_{k,\ell}^{>}(L_t)$.

The proof of Theorem~\ref{thm:mainAMCF} will follow from the
following, which allows us to inductively perturb away all of the
degenerate cylindrical singularities.

\begin{prop}\label{prop:AMCFperturb}
  Suppose that $L_t$ is a $q$-convex mean curvature flow for
  $t \geq T_1$ with $\Sigma_{k,\ell}^{>}(L_t) = \emptyset$ for $t\in
  [T_1,T_2]$, and $L_{T_1}$ smooth. Then we can approximate $L_t$
  arbitrarily well with 
  a $q$-convex AMCF $L'_t$, which satisfies the mean curvature flow for $t \geq T_2$, is
  smooth at $t=T_1, T_2$, and has only nondegenerate cylindrical
  singularities in $[T_1,T_2]$.
\end{prop}
\begin{proof}
  We will argue by induction on the type of degenerate singularity
  that appears. We can consider the base case to be where
  $\Sigma_{1,0}(L_t) \cup \Sigma^{>}_{1,0}(L_t) = \emptyset$ for $t\in
  [T_1,T_2]$, i.e. the case in which $L_t$ already
  only has nondegenerate cylindrical singularities. If there is a (nondegenerate)
  singularity at time $T_2$, then we can slightly translate in time by
  changing $L_{T_1}$ to ensure that $L_{T_2}$ becomes smooth.

  Suppose by induction that the result is already known for all $(k', \ell')
  \prec (k,\ell)$. Let $L_t$ be a mean curvature
  flow for $t\geq T_1$, with $\Sigma_{k,\ell}^> = \emptyset$ for $t\in
  [T_1, T_2]$. By the assumption we have that the set $\Sigma_{k,\ell}(L_t)$ is
  closed over the time interval $t\in [T_1, T_2]$.

  Let $X=(x,t)\in \Sigma_{k,\ell}(L_t)$, with $t\in [T_1,T_2]$. Then there
  exists an $r_X > 0$ such that $X$ is $\ell$-degenerate at scale
  $r_X$. It follows  from Proposition~\ref{prop:nearbydegen}
  that  there is an open set $U_X$ and an $r_X' > 0$ such 
  that all $X' \in U_X\cap \Sigma_{k,\ell}(L_t')$ for sufficiently
  small perturbations $L_t'$ of the flow $L_t$ are $\ell$-degenerate at
  scale $r_X'$. We can cover $\Sigma_{k, \ell}$ in $[T_1,T_2]$ with
  $U_{X_1}, \ldots, U_{X_N}$, and letting $r_0 = \frac{1}{2} \min\{r_{X_i}'\}$ we
  find that all $(x,t) \in \Sigma_{k, \ell}$ with $t\in [T_1,T_2]$ are
  $\ell$-degenerate at scale $2r_0$. Using Proposition~\ref{prop:nearbydegen},
  we can arrange that all the successive perturbations $\tilde{L}_t$ of $L_t$ that we
  consider are sufficiently small, so that they are still $q$-convex, and still
  satisfy the following condition on suitable time intervals inside
  $[T_1,T_2]$. Note that these perturbations may only be AMCFs, so we
  need to restrict to smaller time intervals on which they are genuine
  MCFs:

  \bigskip
  
  \begin{itemize}
    \item[(\dag)] $\Sigma_{k,\ell}^>(\tilde{L}_t) = \emptyset$ and every $X\in
  \Sigma_{k,\ell}(\tilde{L}_t)$ is $\ell$-degenerate at scale $2r_0$
  for times $t\in [T_1,T_2]$ at which $\tilde{L}_t$ is a MCF. 
\end{itemize}

\bigskip
  Consider the flow $L_t$. 
  Let us write $\rho_0 = \Phi(r_0/4)$
  in terms of the function from
  Proposition~\ref{prop:perturbdegen10}. Decreasing $\rho_0$ if
  necessary we can assume $\rho_0 \leq r_0/2$. Let us denote by $s_1 > T_1$
  the first time a singularity in $\Sigma_{k,\ell}(L_t)$ appears. By
  decreasing $r_0$ we can assume that $s_1 - (r_0/2)^2 > T_1$, using
  the assumption that $L_{T_1}$ is smooth.  We 
  apply the induction hypothesis to the interval $[T_1, s_1-(r_0/2)^2]$,
  so we can find a first perturbation $L^{(1)}_t$ of the flow
  $L_t$. This is an AMCF with only nondegenerate cylindrical
  singularities for $t\in [T_1, s_1 - (r_0/2)^2]$, is smooth at
  $t=s_1 - (r_0/2)^2$, and it is a MCF for $t\geq s_1 - (r_0/2)^2$ satisfying
  (\dag) for $t\in [s_1 - (r_0/2)^2, T_2]$. 

  We next construct successive perturbations of the time slice
  $L^{(1)}_{s_1-(r_0/2)^2}$ to obtain mean curvature flows $L^{(1,1)}_t, \ldots,
  L^{(1,N)}_t$ for $t \geq s_1 - (r_0/2)^2$, to eventually eliminate all
  $\ell$-degenerate $(n,k)$-cylindrical singularities in the time
  interval $I = [s_1 - (\rho_0/2)^2, s_1 + (\rho_0/2)^2]$. For this we
  inductively construct $L^{(1,j)}_t$ and points $X_j= (x_j, t_j)$ with
  $t_j \in I$, such that
  \begin{itemize}
    \item[(1)] $L^{(1,j)}_t$ has no $\ell$-degenerate $(n,k)$-cylindrical
      singularity in $Q_{\rho_0}(x_i, t_i)$ for $i\leq j$,
    \item[(2)] $L^{(1,j)}_t$ still satisfies (\dag) for $t\in [ s_1 -
      (r_0/2)^2, T_2]$,
    \item[(3)] Once $L^{(1,j)}_t$ is defined, then we choose some
      $X_{j+1}=(x_{j+1}, t_{j+1}) \in \Sigma_{(k,\ell)}(L^{(1,j)}_t)$
      with $t_{j+1}\in I$ if such a point exists. Then we use
      Proposition~\ref{prop:perturbdegen10} to find an arbitrarily
      small perturbation $L^{(1,j+1)}_{s_1-(r_0/2)^2}$ of $L^{(1,
        j)}_{s_1-(r_0/2)^2}$,  so that the
      corresponding flow satisfies (1), (2). Note that we are applying
      Proposition~\ref{prop:perturbdegen10} with $T=t_{j+1}$, and $r$
      in the notation there, such that $T-r^2=s_1 - (r_0/2)^2$. In
      particular we have $r \geq r_0/4$. 
    \end{itemize}
    This process must end in finitely many steps at some $L^{(1,N)}_t$
    which has no more $\ell$-degenerate $(n,k)$-cylindrical
    singularities with $t\in [s_1-(\rho_0/2)^2, s_1 + (\rho_0/2)^2]$,
    since each $Q_{\rho_0}(x_i, t_i)$ takes up a definite amount of
    area. Note that if our perturbations were sufficiently small, then
    we do not introduce any nondegenerate singularities in the time
    interval $[s_1 - (r_0/2)^2, s_1 - (\rho_0/2)^2]$ either. 

    The flow $L^{(1,N)}_t$ satisfies $\Sigma_{n,k}\cup \Sigma_{n,k}^>
    = \emptyset$ on the time interval $[s_1 - (r_0/2)^2, s_1 -
    (\rho_0/2)^2]$. 
    We apply the induction hypothesis to the flow
    $L^{(1,N)}_t$ on the slightly shorter time interval $[s_1 - (r_0/2)^2, s_1 +
    (\rho_0/2)^2 - (r_0/2)^2]$, and concatenate this with $L^{(1)}_t$
    for $t\in [T_1, s_1 - (r_0/2)^2]$. This way we can obtain an AMCF
    $L^{(2)}_t$ that is still an arbitrarily small perturbation of
    $L_t$, and satisfies
    \begin{itemize}
    \item[(a)] $L^{(2)}_t$ is a MCF for $t\in [s_1 +
      (\rho_0/2)^2 - (r_0/2)^2, T_2]$, satisfying (\dag).
    \item[(b)] $L^{(2)}_t$ is smooth at $t=s_1 +
      (\rho_0/2)^2 - (r_0/2)^2$.
    \item[(c)] $L^{(2)}_t$ has only nondegenerate cylindrical
      singularities for $t \in [T_1, s_1 +
    (\rho_0/2)^2 - (r_0/2)^2]$. 
    \end{itemize}
   We can now repeat the same argument, first choosing successive
   small perturbations of $L^{(2)}_{s_1+(\rho_0/2)^2 - (r_0/2)^2}$ to
   eliminate the $\ell$-degenerate $(n,k)$-cylindrical singularities
   in the time interval $[s_1, s_1 + 2(\rho_0/2)^2]$, and then
   applying the inductive hypothesis to obtain the AMCF
   $L^{(3)}_t$. Since we keep extending the time interval where only
   nondegenerate singularities occur by $(\rho_0/2)^2$ at each step, 
   after finitely many steps we obtain
   an AMCF as required by the Proposition.
\end{proof}

\begin{proof}[Proof of Theorem~\ref{thm:mainAMCF}]
   Given a $q$-convex (compact) initial hypersurface $L_0\subset
   \mathbb{R}^{n+1}$, we have that the flow $L_t$ becomes extinct in
   finite time $T_{ext}$. We can apply Proposition~\ref{prop:AMCFperturb}
   to $L_t$, on the   interval $[0,T_{ext}+1]$, and
   with $(k, \ell) = (n-1,0)$, since $\Sigma_{n-1,0}^>$ is always
   empty.
 \end{proof}

 \section{Appendix}
 In this Appendix we give the proof of
 Proposition~\ref{prop:logSobolev}, following the argument in
 \cite{Ecker}. For convenience we restate the proposition here.
 \begin{prop}
  Suppose that $M_\tau$ is a rescaled mean curvature flow of compact
  hypersurfaces in $\mathbb{R}^{n+1}$, and for
  some $\epsilon \geq 0$,  the nonnegative function
  $f(x,\tau)$ satisfies the differential inequality
  \[ 
    \partial_\tau f \leq \Delta f - \frac{1}{2} x\cdot \nabla f +
    a\cdot \nabla f + b f + \mathcal{E},  \]
  where $a(x, \tau)$ and $b(x,\tau)$ satisfy $|a| \leq \epsilon<1$ and
  $|b| \leq b_0$, and $\mathcal{E}(x,\tau) \geq 0$.
  Then we have 
   \[ \Vert f(\cdot, \tau)\Vert_{L^{p(\tau)}} \leq
     e^{(\epsilon+b_0)\tau} \Vert 
     f(\cdot, 0)\Vert_{L^{2}} + \int_0^\tau e^{(\epsilon+b_0)(\tau -
       s)} \Vert \mathcal{E}(\cdot, s)\Vert_{L^{p(s)}}\, ds, \]
   where we define $p(\tau) = 1 + \epsilon + (1-\epsilon)e^\tau$ (so
   that $p(0)=2$). The $L^p$ norms here are taken using the weight
  $(4\pi)^{-n/2} e^{-|x|^2/4}$. 
\end{prop}
\begin{proof}
  Let us recall the
  sharp log-Sobolev inequality for submanifolds of $\mathbb{R}^{n+1}$
  proved by Brendle~\cite{Brendle20}: for any positive smooth function
  on $M_\tau$ we have
  \[\label{eq:logSob} \int_{M_\tau} \phi\log\phi\,d\mu &- \int_{M_\tau} \phi |\nabla\log
    \phi|^2\, d\mu - \int_{M_\tau} \phi\left| H+
      \frac{x^\perp}{2}\right|^2\,d\mu \\
    &\leq \left(\int_{M_\tau} \phi\,d\mu\right) \log \int_{M_\tau}
    \phi\,d\mu. \]

Let us compute the evolution of the $L^p$-norm of $f$, where
$p=p(\tau)$ also depends on $\tau$. We have
\[ \partial_\tau \left( e^{-|x|^2/4} d\mathcal{H}^n\right) &= -\frac{x}{2}\cdot
  \left(H + \frac{x^\perp}{2}\right) e^{-|x|^2/4} d\mathcal{H}^n -
  e^{-|x|^2/4} H\cdot\left( H + \frac{x^\perp}{2}\right)\,
  d\mathcal{H}^n \\
  &=- \left| H + \frac{x^\perp}{2}\right|^2\, e^{-|x|^2/4}\,
  d\mathcal{H}^n,
\]
and 
\[ \partial_\tau f^p &= pf^{p-1}\partial_\tau f + \dot p \log f\, f^p
  \\
  \partial_\tau \int_{M_\tau} f^p d\mu &= \int_{M_\tau} \left( p f^{p-1}
  \partial_\tau f + \dot p \log f\, f^p - f^p \left| H +
    \frac{x^\perp}{2}\right|^2\right) d\mu, \]
so 
\[ \partial_\tau \Vert f\Vert_p &= \partial_\tau \left(\int_{M_\tau}
    f^p\, d\mu\right)^{1/p} \\
  &= \frac{1}{p} \Vert f\Vert_p^{1-p}\, \partial_\tau \int_{M_\tau} f^p
  d\mu - \frac{\dot p}{p^2} \left( \log \int_{M_\tau} f^pd\mu\right)
  \Vert f\Vert_p \\
  &\leq \frac{1}{p} \Vert f\Vert_p^{1-p} \int_{M_\tau} \dot p \log
  f\, f^p d\mu \\
  &\quad + \Vert f\Vert_p^{1-p} \int_{M_\tau} \Big( f^{p-1}(\Delta f
    - \frac{1}{2}x\cdot \nabla f +
    a\cdot \nabla f + b f + \mathcal{E}) - \frac{1}{p} f^p |H +
    x^\perp/2|^2\Big)\, d\mu \\
    &\quad - \frac{\dot p}{p} \Vert f\Vert_p \log \Vert f\Vert_p,
  \]
  where we used the differential inequality for $\partial_\tau f$. 
  Integrating by parts we have
  \[ \int_{M_\tau} f^{p-1} \Delta f\, e^{-|x|^2/4} d\mathcal{H} = \int_{M_\tau}
    \Big( -(p-1)f^{p-2} |\nabla f|^2  + f^{p-1} \nabla f \cdot
    \frac{x}{2}\Big) e^{-|x|^2/4}\, d\mathcal{H},\]
  so
  \[\int_{M_\tau} f^{p-1}(\Delta f  - \frac{1}{2}x\cdot \nabla f) d\mu
    = -(p-1) \int_{M_\tau} f^{p-2} |\nabla f|^2\, d\mu. \]

  From the log-Sobolev inequality \eqref{eq:logSob} we get
  \[ p \int_{M_\tau} f^p\log f\, d\mu &\leq \left(\int_{M_\tau} f^p\right)
    \log \int_{M_\tau} f^p\, d\mu \\
    &\quad + \int_{M_\tau} p^2 f^{p-2}
    |\nabla f|^2\, d\mu + \int_{M_\tau} f^p\left| H + x^\perp/2\right|^2\,
      d\mu. \]
    Using this in our formula for $\partial_\tau \Vert f\Vert_p$ we get
    \[ \partial_\tau \Vert f\Vert_p &\leq 
      \frac{\dot p}{p^2} \Vert f\Vert_p^{1-p}\left( p \Vert f\Vert_p^p
       \log \Vert f\Vert_p + \int_{M_\tau} \left( p^2 f^{p-2}|\nabla
         f|^2 + f^p |H+x^\perp/2|^2 \right) d\mu\right) \\
     &\quad \Vert f\Vert_p^{1-p}\int_{M_\tau} - (p-1) f^{p-2}|\nabla
     f|^2 + \epsilon f^{p-1} |\nabla f| + b_0 f^p + f^{p-1} \mathcal{E} \, d\mu \\
     &\quad - \frac{1}{p} \Vert f\Vert_p^{1-p} \int_{M_\tau} f^p
     |H+x^\perp/2|^2\, d\mu \\
     &\quad - \frac{\dot p}{p} \Vert f\Vert_p \log \Vert f\Vert_p. 
   \]
   We estimate
   \[ \epsilon f^{p-1} |\nabla f| &\leq \frac{\epsilon}{2} f^{p-2} |\nabla
     f|^2 + \frac{\epsilon}{2} f^p, \\
       \int_{M_\tau} f^{p-1}\mathcal{E} \,d\mu &\leq \left(\int_{M_\tau}
         f^p\,d\mu\right)^{\frac{p-1}{p}} \left(\int_{M_\tau}
         \mathcal{E}^p\,d\mu\right)^{1/p} = \Vert f\Vert_p^{p-1} \Vert
       \mathcal{E}\Vert_p.  \]
   Using this in the calculation above, and simplifying, we get
   \[ \partial_\tau \Vert f\Vert_p &\leq \left(\dot p -
       (p-1-\epsilon)\right)\Vert f\Vert_p^{1-p} \int_{M_\tau} f^{p-2}|\nabla
     f|^2\ d\mu \\
     &\quad + \left(\frac{\dot p}{p^2} - \frac{1}{p}\right) \Vert
     f\Vert_p^{1-p} \int_{M_\tau} f^p |H + x^\perp/2|^2\, d\mu \\
     &\quad + (b_0+\epsilon) \Vert f\Vert_p + \Vert\mathcal{E}\Vert_p. 
   \]
   We set $p(t) = 1+\epsilon + (1-\epsilon) e^\tau$ so that $\dot p =
   p-1-\epsilon$ and $p(0)=2$. Then we get
   \[ \partial_\tau \Vert f\Vert_p \leq (b_0+\epsilon) \Vert f\Vert_p
     + \Vert\mathcal{E}\Vert_p. \] 
   Integrating this differential inequality from 0 to $\tau$ we get
   \[ \Vert f(\cdot, \tau)\Vert_{L^{p(\tau)}} \leq
     e^{(\epsilon+b_0)\tau} \Vert 
     f(\cdot, 0)\Vert_{L^{2}} + \int_0^\tau e^{(\epsilon+b_0)(\tau -
       s)} \Vert \mathcal{E}(\cdot, s)\Vert_{L^{p(s)}}\, ds, \]
   as required. 
\end{proof}
  
\bibliography{mybib}{}
\bibliographystyle{plain}

\end{document}